\documentclass[12pt]{amsart}
\usepackage[cmtip,all]{xy}
\usepackage{graphicx}
\usepackage{amssymb}
\usepackage{mathrsfs}
\usepackage{amsfonts}
\usepackage{amsmath}

\usepackage[pagebackref]{hyperref}

\newtheorem{theorem}{Theorem}[section]
\newtheorem{lemma}[theorem]{Lemma}
\newtheorem{proposition}[theorem]{Proposition}
\newtheorem{corollary}[theorem]{Corollary}

\theoremstyle{definition}

 \theoremstyle{remark}

 \numberwithin{equation}{section}

\begin{document}

\title[Hardy-Sobolev-Maz'ya inequalities  on  complex hyperbolic spaces]{ Sharp Hardy-Sobolev-Maz'ya, Adams and Hardy-Adams inequalities on  the Siegel domains and complex hyperbolic spaces}
\author{Guozhen Lu}
\address{Department of Mathematics, University of Connecticut, Storrs, CT 06269, USA.}

\email{guozhen.lu@uconn.edu}

\author{Qiaohua  Yang}
\address{School of Mathematics and Statistics, Wuhan University, Wuhan, 430072, People's Republic of China.}

\email{qhyang.math@whu.edu.cn}

\thanks{The first author was partly supported by a collaboration grant from the Simons foundation and  the second author was partly  supported by   the National Natural
Science Foundation of China (No.11201346).}


%    General info

%\date{January 1, 2001 and, in revised form, June 22, 2001.}



\begin{abstract}
This paper continues the program initiated in the works  by the authors \cite{LuYang3}, \cite{ly4} and \cite{ly2} and by the authors with Li \cite{LiLuy1} and \cite{LiLuy2} to establish higher order Poincar\'e-Sobolev, Hardy-Sobolev-Maz'ya, Adams and Hardy-Adams inequalities on real hyperbolic spaces
using the method of Helgason-Fourier analysis on the hyperbolic spaces. The aim of this paper is to establish such inequalities on the  Siegel domains and complex hyperbolic spaces. Firstly, we give a
factorization theorem for the  operators on the complex hyperbolic space which is
closely related to Geller' operator, as well as the CR invariant differential operators on the Heisenberg group and CR sphere. Secondly,
by using, among other things, the Kunze-Stein phenomenon on a closed linear group $SU(1,n)$ and  Helgason-Fourier analysis techniques on the complex hyperbolic   spaces, we establish the  Poincar\'e-Sobolev, Hardy-Sobolev-Maz'ya inequality on the Siegel domain $\mathcal{U}^{n}$ and the unit ball  $\mathbb{B}_{\mathbb{C}}^{n}$.
Finally, we establish the sharp Hardy-Adams  inequalities and
sharp Adams type inequalities on Sobolev spaces of any positive fractional
order  on the complex  hyperbolic spaces. The factorization theorem we proved is of its independent interest in the Heisenberg group and CR sphere and CR invariant differential operators therein.

  \end{abstract}

\maketitle

%\section*{This is an unnumbered first-level section head}

\section{Introduction}
Let $\mathbb{B}_{\mathbb{C}}^{n}$ denote the unit ball in $\mathbb{C}^{n}$.
 In \cite{ge},  Geller
introduced a family of second order degenerate elliptic
operators arising in several variables
  \begin{equation*}%\label{2.1}
  \begin{split}
\Delta_{\alpha,\beta}=4(1-|z|^{2})\left\{\sum^{n}_{j=1}\sum^{n}_{k=1}(\delta_{j,k}-z_{j}\overline{z}_{k})\frac{\partial^{2}}{\partial z_{j}
\partial\overline{z}_{k}}+\alpha R+\beta\bar{R}-\alpha\beta\right\},
\end{split}
\end{equation*}
where  $\delta_{j,k}$ denotes the Kronecker symbol and
  \begin{equation*}%\label{2.1}
  \begin{split}
R=\sum^{n}_{j=1}z_{j}\frac{\partial}{\partial z_{j}},\;\;\overline{R}=\sum^{n}_{j=1}\overline{z}_{j}\frac{\partial}{\partial \overline{z}_{j}}.
\end{split}
\end{equation*}
If $\alpha=\beta=0$, $\Delta_{0,0}$ is the invariant Laplacian or Laplace-Beltrami
operator for the Bergman metric on $\mathbb{B}_{\mathbb{C}}^{n}$.
Set
  \begin{equation*}%\label{2.1}
  \begin{split}
\Delta'_{\alpha,\beta}=\frac{1}{4(1-|z|^{2})}\Delta_{\alpha,\beta}
=&\sum^{n}_{j=1}\sum^{n}_{k=1}(\delta_{j,k}-z_{j}\overline{z}_{k})\frac{\partial^{2}}{\partial z_{j}
\partial\overline{z}_{k}}+\alpha R+\beta\bar{R}-\alpha\beta.
\end{split}
\end{equation*}
The above family of operators $\Delta_{\alpha,\beta}$ and $\Delta'_{\alpha,\beta}$ have been considered by many
authors. For example, Geller \cite{ge} introduced such operators  to discuss the $H^{p}$ theory of Hardy spaces on the
Heisenberg group and
 Graham \cite{gra1}, \cite{gra2} has given a precise description of the
associated Dirichlet problem.  We   remark that such operators  are closely related to the
invariant differential operator on the line bundle  of $SU(1, n)/S(U(n)\times U(1))$ associated with the one-dimensional representation
of $S(U(n)\times U(1))$
 (see \cite{sh}).

\medskip

 There are analogous operators on the Siegel domain $$
\mathcal{U}^{n}:=\{z\in\mathbb{C}^{n}: \varrho(z)=\textrm{Im}z_{n}-\sum_{j=1}^{n-1}|z_{j}|^{2}>0\}.
$$
Let $t=\textrm{Re}z_{n}$ and $\mathcal{L}_{0}$ be the Folland-Stein operator on the Heisenberg group (see Section 2 for details).
The analogous operators  are then
\begin{equation*}
  P_{\alpha,\beta}=4\varrho[\varrho(\partial_{\varrho\varrho}+\partial_{t}^{2})-\mathcal{L}_{0}-(n-1+\alpha+\beta)\partial_{\varrho}
  -i(\alpha-\beta)\partial_{t}].
\end{equation*}
If $\alpha=\beta=0$, $P_{0,0}$ is also the  Laplace-Beltrami
operator for the Bergman metric on $\mathcal{U}^{n}$.

\medskip

The purpose of this paper is to establish Hardy-Sobolev-Maz'ya, Adams and Hardy-Adams  type inequalities for $\Delta_{0,0}$ on  $\mathbb{B}_{\mathbb{C}}^{n}$ with the Bergman metric, that is, on the complex hyperbolic space.
Moreover, we should mention that our result are  closely related to operator $\Delta'_{\alpha,\alpha}$ as well as the CR invariant differential operators on the Heisenberg group and the CR sphere.

\medskip

\subsection{ Higher order Hardy-Sobolev-Maz'ya inequalities on hyperbolic spaces}

 Maz'ya (\cite{maz}, Section 2.1.6) gave
a refinement of both the first order Sobolev and the Hardy inequalities in half spaces which are known as the Hardy-Sobolev-Maz'ya inequality.  It reads as follows
\begin{equation}\label{1.3}
\int_{\mathbb{R}^{n}_{+}}|\nabla u|^{2}dx-\frac{1}{4}\int_{\mathbb{R}^{n}_{+}}\frac{u^{2}}{x^{2}_{1}}dx\geq C_n\left(\int_{\mathbb{R}^{n}_{+}}
x^{\gamma}_{1}|u|^{p}dx\right)^{\frac{2}{p}},\;\;\; u\in C^{\infty}_{0}(\mathbb{R}^{n}_{+}),
\end{equation}
where $2<p\leq\frac{2n}{n-2}$, $\gamma=\frac{(n-2)p}{2}-n$, $\mathbb{R}^{n}_{+}=\{(x_{1},x_{2},\cdots,x_{n})\in\mathbb{R}^{n}: x_{1}>0\}$ and $C_n$ is a positive constant which is independent of $u$.

\medskip

We recall that  the classical Sobolev inequalities and sharp constants play an essential role in analysis and geometry.  If $k$ is a positive integer, then the $k-$th order Sobolev inequality states as follows (see \cite{lie,ct, WangX})
 \begin{equation}\label{1.1}
\int_{\mathbb{R}^{n}}|(-\Delta)^{k/2}u|^{2}dx\geq S_{n,k}
\left(\int_{\mathbb{R}^{n}}|u|^{\frac{2n}{n-2k}}dx\right)^{\frac{n-2k}{n}},\;\; u\in C^{\infty}_{0}(\mathbb{R}^{n}),\;\;1\leq k<\frac{n}{2},
\end{equation}
where
 \begin{equation}\label{1.2}
S_{n,k}=\frac{\Gamma(\frac{n+2k}{2})}{\Gamma(\frac{n-2k}{2})}\omega^{2k/n}_{n}
\end{equation}
is the best Sobolev
constant and  $\omega_{n}=\frac{2\pi^{\frac{n+1}{2}}}{\Gamma(\frac{n+1}{2})}$ is the surface measure of $\mathbb{S}^{n}=\{x\in\mathbb{R}^{n+1}: |x|=1\}$.

\medskip
 It is known that the upper half space can be regarded as a hyperbolic space. Then the Hardy-Sobolev-Maz'ya  inequality on the upper half space is equivalent to the corresponding inequality on the hyperbolic ball. By using the Helgason-Fourier analysis on hyperbolic spaces, the authors established in \cite{LuYang3}
the sharp Hardy-Sobolev-Maz'ya inequalities for higher order derivatives. To state the results,
 we now recall the ball model as a hyperbolic space.
 The unit ball
\[\mathbb{B}^{n}=\{x=(x_{1},\cdots,x_{n})\in \mathbb{R}^{n}| |x|<1\}\]
equipped with the usual Poincar\'e metric
\[
ds^{2}=\frac{4(dx^{2}_{1}+\cdots+dx^{2}_{n})}{(1-|x|^{2})^{2}}
\]
is known as the hyperbolic space of ball model.
The hyperbolic gradient is $\nabla_{\mathbb{H}}=\frac{1-|x|^{2}}{2}\nabla$ and
the
Laplace-Beltrami operator is given by
$$
\Delta_{\mathbb{H}}=\frac{1-|x|^{2}}{4}\left\{(1-|x|^{2})\sum^{n}_{j=1}\frac{\partial^{2}}{\partial
x^{2}_{j}}+2(n-2)\sum^{n}_{j=1}x_{j}\frac{\partial}{\partial
x_{j}}\right\}.
$$
The hyperbolic volume is $dV=\left(\frac{2}{1-|x|^2}\right)^ndx$.

\medskip

Then the GJMS operators on $\mathbb{B}^{n}$  is defined as follows (see \cite{FeffermanGr}, \cite{GJMS}, \cite{go},  \cite{j})
\begin{equation}\label{1.5}
 P_{k}=P_{1}(P_{1}+2)\cdot\cdots\cdot(P_{1}+k(k-1)),\;\; k\in\mathbb{N},
\end{equation}
where $P_{1}=-\Delta_{\mathbb{H}}-\frac{n(n-2)}{4}$ is the conformal Laplacian on $\mathbb{B}^{n}$ and $P_{2}$
is the well-known Paneitz operator (see \cite{Paneitz}).
 Then, the authors have proved in \cite{LuYang3} the following higher order Hardy-Sobolev-Maz'ya inequalities on the real hyperbolic spaces.

\begin{theorem}\label{HSM1}
Let $2\leq k<\frac{n}{2}$ and $2<p\leq\frac{2n}{n-2k}$. There exists a positive constant $C=C(n, k, p)$ such that for each $u\in C^{\infty}_{0}(\mathbb{B}^{n})$,
\begin{equation}\label{1.11a}
\int_{\mathbb{B}^{n}}(P_{k}u)udV- \prod^{k}_{i=1}\frac{(2i-1)^{2}}{4}\int_{\mathbb{B}^{n}}u^{2}dV\geq C\left(\int_{\mathbb{B}^{n}}|u|^{p}dV\right)^{\frac{2}{p}}.
\end{equation}
\end{theorem}

This improves substantially the  Poincar\'e-Sobolev inequalities  (\ref{1.3}) for higher order derivatives on the hyperbolic spaces $\mathbb{B}^n$ by subtracting a Hardy term on the left hand side.
If $p=\frac{2n}{n-2k}$, then by (\ref{1.6}), the sharp constant  in (\ref{1.11a}) is less than or equal to the best $k$-th order Sobolev constant $S_{n, k}$.

As an application of Theorem \ref{HSM1}, we have the following Hardy-Sobolev-Maz'ya inequalities for higher order derivatives (see \cite{LuYang3}):
\begin{theorem}\label{HSM2}
Let $2\leq k<\frac{n}{2}$ and $2<p\leq\frac{2n}{n-2k}$. There exists a positive constant $C$ such that for each $u\in C^{\infty}_{0}(\mathbb{R}^{n}_{+})$,
\begin{equation}\label{1.12}
\int_{\mathbb{R}^{n}_{+}}|\nabla^{k}u|^{2}dx- \prod^{k}_{i=1}\frac{(2i-1)^{2}}{4}\int_{\mathbb{R}^{n}_{+}}\frac{u^{2}}{x^{2k}_{1}}dx\geq C\left(\int_{\mathbb{R}^{n}_{+}}x^{\gamma}_{1}|u|^{p}dx\right)^{\frac{2}{p}},
\end{equation}
where $\gamma=\frac{(n-2k)p}{2}-n$.

\medskip

In terms of the Poincar\'e ball model $\mathbb{B}^{n}$, inequality (\ref{1.11a}) can be written as follows:
\begin{equation}\label{1.14}
\int_{\mathbb{B}^{n}}|\nabla^{k}u|^{2}dx- \prod^{k}_{i=1}(2i-1)^{2}\int_{\mathbb{B}^{n}}\frac{u^{2}}{(1-|x|^{2})^{2k}}dx\geq C\left(\int_{\mathbb{B}^{n}}(1-|x|^{2})^{\gamma}|u|^{p}dx\right)^{\frac{2}{p}}.
\end{equation}
\end{theorem}

We remark that the best constant  in the above Hardy-Sobolev-Maz'ya inequalities when $k=1$ and $n=3$ is the same as the Sobolev constant (see \cite{bfl}) and is strictly smaller than the Sobolev constant (see \cite{Hebey}). In the higher order derivative  cases, it was proved
in all the cases of $n=2k+1$, the best constants are the same as the Sobolev constants \cite{ly4} (see also \cite{Hong}).

To state our results on the complex hyperbolic spaces, let us introduce  some conventions. We denote by $\mathcal{U}^{n}$ the Siegel domain model and $\mathbb{B}_{\mathbb{C}}^{n}$ the  ball model of complex hyperbolic space (see Section 2 for details).
 We may assume $n\geq2$ since  in the case $n=1$, the complex hyperbolic space
is isomorphic to the real hyperbolic space of dimension two.
Let $\Delta_{\mathbb{B}}$ be the Laplace-Beltrami operator on the complex hyperbolic space. The  spectral gap of $-\Delta_{\mathbb{B}}$ is $n^{2}$. In fact,
we have
\begin{equation*}
  \textrm{spec}(-\Delta_{\mathbb{B}})=[n^{2},+\infty).
\end{equation*}

As one of our main theorems in this paper,
firstly, we have the following factorization theorem for the  operators on the complex hyperbolic space. This factorization theorem is important in establishing our Hardy-Sobolev-Maz'ya inequalities on complex hyperbolic spaces. We refer the reader to Section 2 for the notations.

\begin{theorem}\label{th1.6}
Let $a\in\mathbb{R}$ and $k\in\mathbb{N}\setminus\{0\}$.  In terms of  the Siegel domain model, we have,  for $u\in C^{\infty}(\mathcal{U}^{n})$,
\begin{equation}\label{a1.5}
\begin{split}
&\prod^{k}_{j=1}\left[\varrho\partial_{\varrho\varrho}+a\partial_{\varrho}+\varrho T^{2}-\mathcal{L}_{0} -i(k+1-2j)T\right](\varrho^{\frac{k-n-a}{2}} u) \\
=&4^{-k}\varrho^{-\frac{k+n+a}{2}}\prod^{k}_{j=1}\left[\Delta_{\mathbb{B}}+n^{2}-(a-k+2j-2)^{2}\right]u.
\end{split}
 \end{equation}
In terms of  the ball  model, we have,  for $f\in C^{\infty}(\mathbb{B}_{\mathbb{C}}^{n})$,
  \begin{equation}\label{a1.6}
\begin{split}
&\prod^{k}_{j=1}\left[\Delta'_{\frac{1-a-n}{2},\frac{1-a-n}{2}}+\frac{(k+1-2j)^{2}}{4}-
\frac{k+1-2j}{2}(R-\bar{R})\right][(1-|z|^{2})^{\frac{k-n-a}{2}} f] \\
=&4^{-k}(1-|z|^{2})^{-\frac{k+n+a}{2}}\prod^{k}_{j=1}\left[\Delta_{\mathbb{B}}+n^{2}-(a-k+2j-2)^{2}\right]f.
\end{split}
 \end{equation}
 \end{theorem}
We remark that the operators on the left side of (\ref{a1.5}) and (\ref{a1.6})  are  closely related to Geller'operator, as well as the CR invariant differential operators on the Heisenberg group and CR sphere, respectively  (see Section 2) and can have important applications on the Heisenberg group and the CR sphere.

\medskip

Secondly, we establish  in this paper the following Poincar\'e-Sobolev inequalities on $\mathbb{B}_{\mathbb{C}}^{n}$. For simplicity, we denote by
$$\|u\|_{p}=\left(\int_{\mathbb{B}_{\mathbb{C}}^{n}}|u|^{p}dV\right)^{\frac{1}{p}}.$$

\begin{theorem}
\label{th1.7}
Let $0< \alpha<3$ and $\zeta>0$. If $0<\beta <2n-\alpha$,  then  for $2< p\leq\frac{4n}{2n-(\alpha+\beta)}$ there exists $C>0$
such that for all $u\in C^{\infty}_{0}(\mathbb{B}_{\mathbb{C}}^{n})$,
\begin{equation}\label{1.5}
  \|u\|_{p}\leq C\|(-\Delta_{\mathbb{B}}-n^{2}+\zeta^{2})^{\beta/4}(-\Delta_{\mathbb{B}}-n^{2})^{\alpha/4}u\|_{2}.
\end{equation}
If $\beta =2n-\alpha$, then for $p>2$, we have
\begin{equation}\label{1.6}
  \|u\|_{p}\leq C\|(-\Delta_{\mathbb{B}}-n^{2}+\zeta^{2})^{\beta/4}(-\Delta_{\mathbb{B}}-n^{2})^{\alpha/4}u\|_{2}.
\end{equation}
\end{theorem}

As the application of Theorems \ref{th1.6} and {\ref{th1.7}}, we have the following Hardy-Sobolev-Maz'ya inequalities on the complex hyperbolic space, which is another main result of this paper.

\begin{theorem}
\label{th1.8}
Let $a\in\mathbb{R}$,  $1\leq k<n$ and $2<p\leq\frac{2n}{n-k}$ and $\mathbb{H}^{n-1}$ is the Heisenberg group. In terms of  the Siegal domain model, there exists a positive constant $C$ such that   for each $u\in C^{\infty}_{0}(\mathcal{U}^{n})$ we have
\begin{equation*}%\label{a1.7}
  \begin{split}
 & \int_{\mathbb{H}^{2n-1}}\int^{\infty}_{0}u\prod^{k}_{j=1}\left[-\varrho\partial_{\varrho\varrho}-a\partial_{\varrho}-\varrho T^{2}+\mathcal{L}_{0} +i(k+1-2j)T\right]u\frac{dzdtd\varrho}{\varrho^{1-a}}\\
  &-\prod^{k}_{j=1}\frac{(a-k+2j-2)^{2}}{4}\int_{\mathbb{H}^{2n-1}}\int^{\infty}_{0}\frac{u^{2}}{\varrho^{k+1-a}}dzdtd\varrho\\
 \geq&C\left(\int_{\mathbb{H}^{2n-1}}\int^{\infty}_{0}|u|^{p}\varrho^{\gamma}dzdtd\varrho\right)^{\frac{2}{p}},
  \end{split}
\end{equation*}
where $\gamma=\frac{(n-k+a)p}{2}-n-1$. In terms of  the ball  model, we have  for $f\in C_{0}^{\infty}(\mathbb{B}_{\mathbb{C}}^{n})$,
\begin{equation*}%\label{a1.7}
  \begin{split}
 &\int_{\mathbb{B}_{\mathbb{C}}^{n}}f \prod^{k}_{j=1}\left[\Delta'_{\frac{1-a-n}{2},\frac{1-a-n}{2}}+\frac{(k+1-2j)^{2}}{4}-
\frac{k+1-2j}{2}(R-\bar{R})\right]f\frac{dz}{(1-|z|^{2})^{1-a}}\\
  &-\prod^{k}_{j=1}\frac{(a-k+2j-2)^{2}}{4}\int_{\mathbb{B}_{\mathbb{C}}^{n}}\frac{f^{2}}{(1-|z|^{2})^{k+1-a}}dz\\
 \geq&C\left(\int_{\mathbb{B}_{\mathbb{C}}^{n}}|f|^{p}(1-|z|^{2})^{\gamma}dz\right)^{\frac{2}{p}}.
  \end{split}
\end{equation*}

\end{theorem}

\subsection{Adams and Hardy-Adams inequalities on hyperbolic spaces}

 We note the Trudinger-Moser inequality    is  a borderline case of the Sobolev embedding
   $W^{1,p}_0(\Omega)\subset L^q(\Omega)$ where $p<N$ when $\Omega\subset \mathbb{R}^N$ $  (N\ge2)$ is a bounded domain with $1\le q\le\frac{Np}{N-p}$,  Trudinger \cite{t} proved   that $W^{1, N}_0(\Omega)\subset L_{\varphi_N}(\Omega)$,  where $L_{\varphi_N}(\Omega)$ is the Orlicz space associated with the Young function $\varphi_N(t)=\exp(\beta|t|^{N/N-1})-1$ for some $\beta>0$ (see also Yudovich \cite{Yu}, Pohozaev \cite{Po}). The sharp form of this inequality was established by J. Moser \cite{mo} by identifying the best constant for $\beta$.

In 1988, D. Adams extended the Trudinger-Moser inequality to
high order Sobolev spaces and  proved the following
\begin{theorem}\label{th1.1}
Let $\Omega$ be a domain in $\mathbb{R}^{n}$ with finite $n$-measure
and $m$ be a positive integer less than $n$. There is a constant
$c_{0}=c_{0}(m,n)$ such that for all $u\in C^{m}(\mathbb{R}^{n})$
with support contained in $\Omega$ and $\|\nabla^{m}u\|_{n/m}\leq
1$,
 the following uniform
inequality holds
\begin{equation}\label{1.1}
\frac{1}{|\Omega|}\int_{\Omega}\exp(\beta_{0}(m,n)|u|^{n/(n-m)})dx\leq
c_{0},
\end{equation}
where
\begin{equation}\label{1.2}
\beta_{0}(m,n)=\left\{
                 \begin{array}{ll}
                   \frac{n}{\omega_{n-1}}\left[\frac{\pi^{n/2}2^{m}\Gamma((m+1)/2)}{\Gamma((n-m+1)/2)}\right]^{n/(n-m)}, & \hbox{$m$ = odd;} \\
                                       \frac{n}{\omega_{n-1}}\left[\frac{\pi^{n/2}2^{m}\Gamma(m/2)}{\Gamma((n-m)/2)}\right]^{n/(n-m)}, & \hbox{$m$ = even}
                 \end{array}
               \right.
\end{equation}
and $\omega_{n-1}$ is the surface measure of the unit sphere in
$\mathbb{R}^{n}$.
 Furthermore, the constant $\beta_{0}(m,n)$ in (\ref{1.1}) is sharp in
the sense that if $\beta_{0}(m,n)$ is replaced by any larger number,
then the integral in (\ref{1.1}) cannot be bounded uniformly by any
 constant.
\end{theorem}

 Such inequalities have been generalized in many directions by many authors.
 For other Trudinger-Moser-Adams inequalities inequalities on $\mathbb{R}^{n}$,  the Heisenberg group and CR sphere, hyperbolic spaces and Riemannian manifolds, we refer to, to just name a few, 
   \cite{be1}, \cite{be}, \cite{be2},  \cite{CC}, \cite{cy1},  \cite{cl1}, \cite{cl2},  \cite{ks},  \cite{l1}, \cite{ll2}, \cite{l1}, \cite{l2},  \cite{LiYX}, \cite{ly},  \cite{mst},  \cite{rs},
  \cite{ysk},   and many references therein.

More recently, Hardy-Adams inequalities  have been established for all the case $n\geq 3$ (see \cite{ly2,LiLuy1} and \cite{y2}, Corollary 1.4). We state it here  as follows.

\begin{theorem}\label{th1.3}
Let $n\geq3$,  $\zeta>0$ and  $0<s<3/2$.
Then there exists a constant $C_{\zeta,n}>0$ such that for all
$u\in C^{\infty}_{0}(\mathbb{B}^{n})$ with
\[\int_{\mathbb{B}^{n}}|(-\Delta_{\mathbb{H}}-(n-1)^{2}/4)^{s/2}(-\Delta_{\mathbb{H}}-(n-1)^{2}/4
+\zeta^{2})^{\frac{n-2s}{4}} u|^{2}dV\leq1.\]
there holds
\[
\int_{\mathbb{B}^{n}}(e^{\beta_{0}\left(n/2,n\right) u^{2}}-1-\beta_{0}\left(n/2,n\right)
u^{2})dV\leq C_{\zeta,n},
\]
where $
dV=\left(\frac{2}{1-|x|^{2}}\right)^{n}dx$ is  the hyperbolic volume element and $\Delta_{\mathbb{H}}$ is
the
Laplace-Beltrami operator on hyperbolic space $\mathbb{B}^{n}$.
\end{theorem}

We remark that the proof of Theorem \ref{th1.3}
relies on hard analysis of Green's functions estimates and Fourier analysis on hyperbolic spaces .
By the following  factorization theorem for the  operators on real hyperbolic space (see \cite{liu} for the ball model
and \cite{LuYang3} for the half space model)
\begin{equation}\label{1.13}
\left(\frac{1-|x|^{2}}{2}\right)^{k+\frac{n}{2}}\Delta^{k}\left[\left(\frac{1-|x|^{2}}{2}\right)^{k-\frac{n}{2}}f\right]=\prod^{k}_{i=1}
\left[\Delta_{\mathbb{H}}
+\frac{(n-2i)(n+2i-2)}{4}\right]f,
\end{equation}
we have, by Theorem \ref{th1.3},
 the following Hardy-Adams
inequalities
\begin{theorem}\label{c2}
There exists a constant $C>0$ such that for all $u\in
C^{\infty}_{0}(\mathbb{B}^{n})$ with
\[
\int_{\mathbb{B}^{n}}|\nabla^{\frac{n}{2}}
u|^{2}dx-\prod^{n/2}_{k=1}(2k-1)^{2}\int_{\mathbb{B}^{n}}\frac{u^{2}}{(1-|x|^{2})^{n}}dx\leq1,
\]
there holds
\[
\int_{\mathbb{B}^{n}}\frac{e^{\beta_{0}\left(n/2,n\right)
u^{2}}-1-\beta_{0}\left(n/2,n\right) u^{2}}{(1-|x|^{2})^{n}}dx\leq C.
\]
\end{theorem}

 We should mention that the  Hardy-Trudinger-Moser inequality on the unit ball $\mathbb{B}^{n}\subset\mathbb{R}^{n}$ when $n=2$ was first established by  Wang and Ye \cite{wy}.
In a recent paper \cite{ly}, the  authors  give a rearrangement-free argument of the result in \cite{wy} and confirm that the  result
in  \cite{wy}  holds for any bounded and convex domain in
$\mathbb{R}^{2}$  as conjectured in \cite{wy}. The method in \cite{ly} is via an argument from local inequalities to global ones  together with   the Riemann mapping theorem.

\medskip

In this paper,  we will establish the following Adams inequalities on the complex hyperbolic space.
\begin{theorem}\label{th1.10} Let $0< \alpha<3$ and $\zeta>0$.
Then there exists a constant $C>0$ such that for all
$u\in C^{\infty}_{0}(\mathbb{B}_{\mathbb{C}}^{n})$ with
\[
\|(-\Delta_{\mathbb{B}}-n^{2}+\zeta^{2})^{(2n-\alpha)/4}(-\Delta_{\mathbb{B}}-n^{2})^{\alpha/4}u\|_{2}\leq1,
\]
there holds
\[
\int_{\mathbb{B}^{n}_{\mathbb{C}}}(e^{\beta_{0}\left(n,2n\right) u^{2}}-1-\beta_{0}\left(n,2n\right)
u^{2})dV\leq C.
\]
\end{theorem}
As an application of  Theorem \ref{th1.10},  we have
 the following Hardy-Adams
inequalities on the complex hyperbolic space:  both for the complex ball model and the Siegel domain model, which is another main theorem in this paper.

\begin{theorem}\label{th1.11} Let $a\in\mathbb{R}$.
There exists a constant $C>0$ such that for all $u\in
C^{\infty}_{0}(\mathbb{B}_{\mathbb{C}}^{n})$ with
\begin{equation*}%\label{a1.7}
  \begin{split}
 &4^{n-1}\int_{\mathbb{B}_{\mathbb{C}}^{n}}f \prod^{n}_{j=1}\left[\Delta'_{\frac{1-a-n}{2},\frac{1-a-n}{2}}+\frac{(n+1-2j)^{2}}{4}-
\frac{n+1-2j}{2}(R-\bar{R})\right]f\frac{dz}{(1-|z|^{2})^{1-a}}\\
  &-\frac{1}{4}\prod^{k}_{j=1}(a-k+2j-2)^{2}\int_{\mathbb{B}_{\mathbb{C}}^{n}}\frac{u^{2}}{(1-|z|^{2})^{n+1-a}}dz\leq1,
  \end{split}
\end{equation*}
there holds
\[
\int_{\mathbb{B}_{\mathbb{C}}^{n}} \frac{e^{\beta_{0}\left(n,2n\right)
(1-|z|^{2})^{a}u^{2}}-1-\beta_{0}\left(n,2n\right)(1-|z|^{2})^{a} u^{2}}{(1-|z|^{2})^{n+1-a}}dz\leq C
\]
and
\[
\int_{\mathbb{B}_{\mathbb{C}}^{n}} e^{\beta_{0}\left(n,2n\right)
(1-|z|^{2})^{a}u^{2}}dz\leq C,\;\; \textrm{if}\;\; a\leq n+1.
\]
In terms of  the Siegel domain model, we have for all $u\in
C^{\infty}_{0}(\mathcal{U}^{n})$ with
\begin{equation*}%\label{1.7}
  \begin{split}
 & 4^{n-1}\int_{\mathbb{H}^{2n-1}}\int^{\infty}_{0}u\prod^{n}_{j=1}\left[-\varrho\partial_{\varrho\varrho}-a\partial_{\varrho}+\varrho T^{2}+\mathcal{L}_{0} +i(k+1-2j)T\right]u\frac{dzdtd\varrho}{\varrho^{1-a}}\\
  &-\frac{1}{4}\prod^{n}_{j=1}(a-n+2j-2)^{2}\int_{\mathbb{H}^{2n-1}}\int^{\infty}_{0}\frac{u^{2}}{\varrho^{n+1-a}}dzdtd\varrho\leq1,
    \end{split}
\end{equation*}
there holds
\[
 \int_{\mathbb{H}^{2n-1}}\int^{\infty}_{0}\frac{e^{\beta_{0}\left(n,2n\right)
\varrho^{a}u^{2}}-1-\beta_{0}\left(n,2n\right)\varrho^{a} u^{2}}{\varrho^{n+1-a}}dzdtd\varrho\leq C.
\]
\end{theorem}

Finally, we also  set up some Adams type inequalities on Sobolev spaces $W^{\alpha,\frac{2n}{\alpha}}(\mathbb{B}_{\mathbb{C}}^{n})$ for arbitrary positive fractional
 order $\alpha<2n$.
 More precisely, we will prove first the following local result:
\begin{theorem} \label{th1.12}
Let $n\geq2$, $0<\alpha<2n$ be an arbitrary real positive number, $p=2n/\alpha$ and $\zeta$ satisfies $\zeta>0$ if $1<p<2$
and $\zeta>2n(\frac{1}{2}-\frac{1}{p})$ if $p\geq2$.
 Then for  measurable $E$
with finite $2n$-dimensional complex hyperbolic volume measure in $\mathbb{B}^{n}_{\mathbb{C}}$, there exists $C=C(\zeta,\alpha,n,|E|)$ such that
\[
\frac{1}{|E|}\int_{E}\exp(\beta(n,\alpha)|u|^{p'})dV\leq C
\]
for  any $u\in W^{\alpha,p}(\mathbb{B}^{n}_{\mathbb{C}})$ with $\int_{\mathbb{B}^{n}_{\mathbb{C}}}
|(-\Delta_{\mathbb{B}}-n^{2}+\zeta^{2})^{\frac{\alpha}{2}} u|^{p}dV\leq1$. Here $p'=\frac{p}{p-1}$ and
\[
\beta(2n,\alpha)=\frac{2n}{\omega_{2n-1}}\left[\frac{\pi^{n}2^{\alpha}\Gamma(\alpha/2)}{\Gamma(\frac{2n-\alpha}{2})}\right]^{p'}.
\]
Furthermore, this inequality is sharp  in the sense that if $\beta(n,\alpha)$ is replaced by any
$\beta>\beta(n,\alpha)$, then the above inequality can no longer hold with some $C$ independent of $u$.
\end{theorem}

It is fairly standard by now, we can establish the following global Hardy-Adams inequalities on the complex hyperbolic spaces
using the symmetrization-free argument of the local inequalities to global ones  developed earlier by Lam and the first author
\cite{ll, ll2}.

\begin{theorem} \label{th1.13}
Let $n\geq2$,  $0<\alpha<2n$ be an arbitrary real positive number, $p=2n/\alpha$ and $\zeta$ satisfies
 $\zeta>2n\left|\frac{1}{2}-\frac{1}{p}\right|$ .
Then there exists $C=C(\zeta,\gamma,n)$ such that
\begin{equation}\label{inequality1}
\int_{\mathbb{B}^{n}_{\mathbb{C}}}\Phi_{p}(\beta(2n,\alpha)|u|^{p'})dV\leq C
\end{equation}
and
\begin{equation}\label{inequality2}
\int_{\mathbb{B}^{n}_{\mathbb{C}}}e^{\beta(2n,\alpha)|u|^{p'}}dz\leq C
\end{equation}
hold simultaneously
for  any $u\in W^{\alpha,p}(\mathbb{B}^{n}_{\mathbb{C}})$ with $\int_{\mathbb{B}^{n}_{\mathbb{C}}}|(-\Delta_{\mathbb{B}}-
n^{2}+\zeta^{2})^{\frac{\alpha}{2}} u|^{p}dV\leq1$. Here
\[
\Phi_{p}(t)=e^{t}-\sum^{j_{p}-2}_{j=0}\frac{t^{j}}{j!},\;\; j_{p}=\min \{j\in \mathbb{N}: j\geq p\}.
\]
Furthermore, this inequality is sharp  in the sense that if $\beta(2n,\alpha)$ is replaced by any
 $\beta>\beta(2n,\alpha)$, then the above inequality can no longer hold with some $C$ independent of $u$.
\end{theorem}

Notice that $2n|\frac{1}{2}-\frac{1}{p}|<n$ provided $p>1$. Choosing $\zeta=n$ in Theorem \ref{th1.8},  we have the following
\begin{corollary}\label{th1.14}
Let $n\geq2$,  $0<\alpha<2n$ be an arbitrary real positive number and $p=n/\alpha$.
There exists $C=C(\alpha,n)$ such that
\[
\int_{\mathbb{B}^{n}_{\mathbb{C}}}\Phi_{p}(\beta(2n,\alpha)|u|^{p'})dV\leq C
\]
and
\[
\int_{\mathbb{B}^{n}_{\mathbb{C}}}e^{\beta(2n,\alpha)|u|^{p'}}dz\leq C
\]
hold simultaneously
for  any $u\in W^{\alpha,p}(\mathbb{B}^{n}_{\mathbb{C}})$ with $\int_{\mathbb{B}^{n}_{\mathbb{C}}}|(-\Delta_{\mathbb{B}})^{\frac{\alpha}{2}} u|^{p}dV\leq1$.
\end{corollary}

The following remarks are in order.

First, as we have shown in the works on higher order Poincar\'e-Sobolev and Hardy-Soblev-Maz'ya inequalities by the authors (\cite{LuYang3}, \cite{ly4}), factorization theorems play an important role in establishing these sharp inequalities for higher order. Proving such a factorization theorem in the complex hyperbolic spaces is significantly more difficult than in the real hyperbolic space. This is one of our main results in this paper. Its connection between the Heisenberg group and CR sphere and the complex hyperbolic space makes this particularly interesting.

Second, as we have demonstrated in the  works by the authors with Li \cite{ly2}, \cite{LiLuy1} and \cite{LiLuy2}, as long as the Green's kernel estimates are established,  we have developed a general scheme in  \cite{ly2}, \cite{LiLuy1} and \cite{LiLuy2} to combine
D. Adams' method of using O'Neil's lemma
and the method by Lam and the first author \cite{ll, ll2} passing from the local inequalities to the global inequalities using the level set approach to
 carry out the proofs of the Adams and Hardy-Adams inequalities on the hyperbolic spaces. This method is fairly standard by now and it is easy to see that these general schemes  work well for more general settings almost identically, especially for Riemannian symmetric spaces of noncompact type since the sharp estimates of Bessel-Green-Riesz kernels is well known (see \cite{an2,anj}).
 In fact, if we denote by $\Delta_{M}$ the Laplace-Beltrami operaor on a Riemannian symmetric space of noncompact type and let
 \begin{equation*}
   k_{\alpha}=(-\Delta_{M}-\varrho^{2})^{-\frac{\alpha}{2}}\;\; \textrm{and}\;\;k_{\zeta,\alpha}=(-\Delta_{M}-\varrho^{2}+\zeta^{2})^{-\frac{\alpha}{2}},
 \end{equation*}
where $\varrho$ is the half-sum of the positive roots of $M$, then for $|x|>0$,
\begin{equation*}
\begin{split}
  k_{\alpha}(x)\thicksim& |x|^{\alpha-l-2|\Sigma^{++}|}\varphi_{0}(|x|),\;\;0<\alpha<\min\{l+|\Sigma_{0}^{+}|,N\};\\
  k_{\zeta,\alpha}(x)\thicksim& |x|^{\frac{\alpha-l-1}{2}-|\Sigma^{++}|}\varphi_{0}(|x|),\; 0<\alpha<N,
  \end{split}
\end{equation*}
 where $l$ is the rank of $M$, $\Sigma^{++}$ is the  set of positive simple roots,  $N=\dim M$ and $\varphi_{0}$
is spherical function (see \cite{anj}, Theorem 4.2.2). The $L^{p}-L^{q}$ estimates for $f\ast  k_{\alpha}$ and $f\ast  k_{\zeta,\alpha}$ are  well known (see \cite{an2} and \cite{co2}). For example, there holds the Poincar\'e-Sobolev inequality for $f\in C_{0}^{\infty}(M)$ (see \cite{co2}, Theorem 6.1):
\begin{equation}\label{b1.1}
  \|f\|_{L^{p}(M)}^{2}\leq \int_{M}|\sqrt{\Delta_{M}-\varrho^{2}}f|^{2}dV,\;\;2<p\leq\frac{2N}{N-2}.
\end{equation}
The Adams' inequality on bounded domain then follows by combining such estimates and  Adams' method of using O'Neil's lemma. To obtain Hardy-Adams inequality, one needs to
get the  estimates for $k_{\alpha}\ast k_{\zeta,\alpha}$.
Though the estimates of  $k_{\alpha}\ast k_{\zeta,\alpha}$ for $|x|\rightarrow0$ is trivial (the method is similar to that in Euclidean space), it is  nontrivial for $|x|\rightarrow\infty$, as well as the $L^{p}-L^{q}$ estimates for $f\ast  (k_{\alpha}\ast k_{\zeta,\alpha})$. In fact,  in this paper, we show some of estimates of $f\ast(k_{\alpha}\ast k_{\zeta,\alpha})$ through Funk-Hecke formula  and Kunze-Stein phenomenon for $SU(n,1)$.
To overcome this problem, one can use the the Plancherel formula on $M$ and the $L^{2}$ estimates of $k_{\alpha}\ast k_{\zeta,\alpha}$ outside of the origin, as we do for the real and complex hyperbolic space (see Lemma \ref{lm4.1} or \cite{LiLuy2}, Lemma 4.1).
The Hardy-Adams inequality then follows by  combining this fact, Adams' method and the method by Lam and the first author \cite{ll, ll2} passing from the local inequalities to the global inequalities.
 Therefore, in this paper, we will not give all the details when we prove the Adams and Hardy-Adams inequalities on the complex hyperbolic space (Theorems \ref{th1.12}, \ref{th1.13}) and simply refer the interested reader to  \cite{ly2}, \cite{LiLuy1} and \cite{LiLuy2} to fill out all the details.

Third,
in a recent  paper \cite{fly}, Joshua Flynn and the  authors have shown  analogous   factorization theorems on the quaternionic hyperbolic
space and the Cayley  hyperbolic plane, and use these to establish the similar Hardy-Sobolev-Maz'ya inequalities and Hardy-Adams inequalities.
Therefore, such inequalities are established for all  rank one
symmetric spaces of non-compact type (i.e.  the real, complex and quaternionic hyperbolic
spaces, and the Cayley  hyperbolic plane). Though, as we have illustrated  above, showing the Adams and Hardy-Adams inequalities can be standard by now by following the general scheme in \cite{ly2}, \cite{LiLuy1} and \cite{LiLuy2} and this paper, it remains a nontrivial problem to establish the Hardy-Sobolev-Maz'ya inequalities on
general symmetric spaces of higher ranks.

The organization of the paper is as follows:  In Section 2, we recall some necessary preliminary facts of complex  hyperbolic spaces.  We shall prove the factorization theorem, namely Theorem \ref{th1.6},  in   Section 3. Sharp estimates of Bessel-Green-Riesz kernels  of certain fractional Laplacians and their rearrangement estimates on complex hyperbolic spaces are given  in Section 4 and Section 5, respectively. The higher order Hardy-Sobolev-Maz'ya inequalities, namely Theorem \ref{th1.7} and \ref{th1.8},  are proved in Section 6.  In Section 7, we prove  the Hardy-Adams inequality on complex  hyperbolic spaces, namely Theorem \ref{th1.10}
and \ref{th1.11}.
In Section 8, we show the Adams type inequality on complex  hyperbolic spaces, namely Theorem \ref{th1.12}
and \ref{th1.13}.

\section{Notations and preliminaries}
%We begin by quoting some preliminary facts which will be needed in
%the sequel and  refer to \cite{fs,gra1,gra2,he,he2} for more information about this subject.

We firstly recall  some geometric properties of complex hyperbolic space. Consider $\mathbb{C}^{n+1}$ equipped with the Hermitian product
$$F(z,w)=-z_{0}\bar{w}_{0}+\sum^{n}_{j=1}z_{j}\bar{w}_{j}.$$
Let $M=\{z\in\mathbb{C}^{n+1}: F(z,z)=-1\}$ be a real hypersurface in $\mathbb{C}^{n+1}$ and $U(1,n)$ be a closed linear group defined as
$$U(1,n)=\{A\in GL(n+1,\mathbb{C}): F(Az,Aw)=F(z,w),z,w\in\mathbb{C}\}.$$
It is known that  $U(1,n)$ is connected and acts transitively on $M$. Furthermore, the group $U(1)=\{e^{i\theta}: 0\leq \theta <2\pi\}$
acts freely on $M$ by $z\rightarrow e^{i\theta}z$. Let $M'=M/ U(1)$. Then $M'$  is the base space of the fibration
  $$U(1)\rightarrow M\rightarrow M'.$$
We note $M'$ is  simply connected.   Denote by $P_{\pi}: M\rightarrow M'$ the natural projection.
The metric induced by $\textrm{Re} F$ on $M$ has signature $(1,2n)$. However,  its restriction to the
orthogonal of the fiber is positive definite. In fact, if we define an inner product $g$ in each tangent space $T_{p}M',\;p\in M'$ by
$$g(X,Y)=\textrm{Re}F(X,Y),\;\;X,Y\in T_{p}M',$$
then $g$ is  a Riemannian
metric  $M'$.
The complex structure $w\rightarrow iw$ on $T'_{z}=\{w\in \mathbb{C}^{n+1}: F(z,w)=0\}$ are compatible with the action of $U(1)$. It induce an
almost complex structure $J$ on $M'$ such that the metric $g$ is Hermitian with respect to $J$.
Moreover, $(M',g)$ is  a complete Kaehler manifold and  has constant holomorphic section
  curvature $-4$ (see \cite{kn}, Chapter XI, Example 10.7).

We remark  that the complex hyperbolic space is also holomorphically isometric to the Hermitian symmetric space $SU(1,n)/S(U(1)\times U(n))$, where
$$SU(1,n)=\{A\in U(1,n): \det A=1\}. $$
 In fact, the closed linear group $SU(1,n)$ acts transitively on $M'$ and the isotropy group at $P_{\pi}(e_{0})$ is
$S(U(1)\times U(n))$, where $e_{0}=(1,0,\cdots,0)$.

There are also other   models of complex  hyperbolic space, for example,
 the Siegel domain  model $\mathcal{U}^{n}$ and the  ball model $\mathbb{B}^{n}_{\mathbb{C}}$. We note these models are holomorphically isometric to each other
 because
 any two simply connected complete
Kaehler manifolds of constant holomorphic sectional curvature $-4$ are
holomorphically isometric to each other (see \cite{kn}, Chapter IX, Theorem 7.9).
%Before giving other  models of complex  hyperbolic space, we  recall some  fact of Heisenberg group and CR sphere,  which are the boundary
%of Siegel domain   $\mathcal{U}^{n}$ and the  ball model $\mathbb{B}^{n}_{\mathbb{C}}$, respectively.

%\subsection{Heisenberg group and CR sphere}

\subsection{The Siegel domain model  $\mathcal{U}^{n}$}
The Siegel domain  $\mathcal{U}^{n}\subset \mathbb{C}^{n}$ is defined as
$$
\mathcal{U}^{n}:=\{z\in\mathbb{C}^{n}: \varrho(z)>0\},
$$
where
\begin{equation}\label{2.1}
\varrho(z)=\textrm{Im}z_{n}-\sum_{j=1}^{n-1}|z_{j}|^{2}.
\end{equation}
Its boundary $\partial \mathcal{U}^{n}:=\{z\in\mathbb{C}^{n}: \varrho(z)=0\}$ can be identified with the Heisenberg group. Recall that  the Heisenberg group $\mathbb{H}^{2n-1}=(\mathbb{C}^{n-1}\times \mathbb{R},\circ)$ is a nilpotent group of step two  with the group law
\begin{equation*}
  (z,t)\circ(z',t')=(z+z',t+t'+2 \textrm{Im}(z,z')),
\end{equation*}
where $z,z'\in \mathbb{C}^{n-1}$ and $(z,z')$ is the Hermite inner product
\begin{equation*}
(z,z')=\sum^{n}_{j=1}z_{j}\bar{z}'_{j}.
\end{equation*}
Set $z_{j}=x_{j}+iy_{j} (1\leq j\leq n-1)$ and define
\begin{equation*}
  X_{j}=\frac{\partial}{\partial x_{j}}+ 2y_{j}\frac{\partial}{\partial t},\;
    Y_{j}=\frac{\partial}{\partial y_{j}}- 2x_{j}\frac{\partial}{\partial t}\;\textrm{for}\;j=1,\cdots,n-1, \; T=\frac{\partial}{\partial t}.
\end{equation*}
The $2n-1$ vector fields $X_{1},\cdots, X_{n-1},Y_{1},\cdots, Y_{n-1}, T$ are
 left-invariant and form a basis for Lie algebra of $\mathbb{H}^{2n-1}$. Furthermore, we have the commutation relations
 \begin{equation*}
   [X_{j},Y_{j}]=-4T,\;j=1,\cdots,n-1,
 \end{equation*}
and that all other commutations vanish. The sub-Laplacian on $\mathbb{H}^{2n-1}$ is given by
\begin{equation}\label{2.2}
  \Delta_{b}=\frac{1}{4}\sum^{n-1}_{j=1}(X^{2}_{j}+Y^{2}_{j}).
\end{equation}
Set $\mathcal{L}_{0} =-\Delta_{b}$. The Bergman metric on $\mathcal{U}^{n}$ is the metric with Kaehler form $  \omega=\frac{i}{2}\partial\bar{\partial}\log\frac{1}{\varrho}$.
Through the coordinates $(z,t,\varrho)\in \mathbb{H}^{2n-1}\times (0,+\infty) \backsimeq\mathcal{U}^{n}$, we  can give
the volume form in  Bergman metric  as follows (see \cite{gra1})
\begin{equation*}
dV=\frac{1}{4\varrho^{n+1}}dzdtd\varrho.
\end{equation*}
 The Laplace-Beltrami operator is given by
\begin{equation*}
  \Delta_{\mathbb{B}}=4\varrho[\varrho(\partial_{\varrho\varrho}+T^{2})-\mathcal{L}_{0} -(n-1)\partial_{\varrho}].
\end{equation*}

\subsection{The  ball model $\mathbb{B}_{\mathbb{C}}^{n}$}
It is given by the unit ball
\[\mathbb{B}_{\mathbb{C}}^{n}=\{z=(z_{1},\cdots,z_{n})\in \mathbb{C}^{n}: |z|<1\}\]
equipped with the Kaehler metric
\[
ds^{2}=-\partial\bar{\partial}\log (1-|z|^{2}).
\]
The volume form in this Bergman metric  is
\[
dV=\frac{dz}{(1-|z|^{2})^{n+1}},
\]
where $dz$ is the usual Euclidean volume form. The Laplace-Beltrami operator is given by
\[
\Delta_{\mathbb{B}}=4(1-|z|^{2})\sum^{n}_{j,k=1}(\delta_{jk}-z_{j}\bar{z}_{k})\frac{\partial^{2}}{\partial z_{j}\partial \bar{z}_{k}},
\]
where
\[
\delta_{j,k}=\left\{
               \begin{array}{ll}
                 1, & \hbox{$j=k$;} \\
                 0, & \hbox{$j\neq k$.}
               \end{array}
             \right.
\]
The Cayley  transform $\mathcal{C}: \mathbb{B}_{\mathbb{C}}^{n}\rightarrow\mathcal{U}^{n}$, defined by
\begin{equation}\label{b2.5}
\mathcal{C}(z)=\left(\frac{z_{1}}{1+z_{n}}, \cdots, \frac{z_{n-1}}{1+z_{n}},i\frac{1-z_{n}}{1+z_{n}}\right),
\end{equation}
is an isometry of $\mathbb{B}_{\mathbb{C}}^{n}$ onto the space $\mathcal{U}^{n}$. The geodesic distance between $0$ and $z\in \mathbb{B}_{\mathbb{C}}^{n}$
is
\begin{equation}\label{2.4}
  \rho(0,z)=\frac{1}{2}\log\frac{1+|z|}{1-|z|}.
\end{equation}
For simplicity, we write $\rho(z)= \rho(0,z)$. A function $f$ is said to be radial if $f(z)=\widetilde{f}(\rho(z))$ for some function $\widetilde{f}$.
In terms of  polar coordinates, we can rewrite  the volume form as following
\begin{equation}\label{2.5}
  dV=(\sinh\rho)^{2n-1}\cosh\rho d\rho d\sigma,
\end{equation}
where $d\sigma$ is the  Lebesgue measure on the sphere $\mathbb{S}^{2n-1}=
\{z=(z_{1},\cdots,z_{n})\in \mathbb{C}^{n}: |z|=1\}$.

 Recall that  Geller's operator $\Delta_{\alpha,\beta}$  is given by
 \begin{equation}\label{b2.7}
   \Delta_{\alpha,\beta}=4(1-|z|^{2})\left[\sum^{n}_{j,k=1}(\delta_{jk}-z_{j}\bar{z}_{k})\frac{\partial^{2}}{\partial z_{j}\partial \bar{z}_{k}}
   +\alpha\sum^{n}_{j=1}z_{j}\frac{\partial}{\partial z_{j}}+\beta\sum^{n}_{j=1}\overline{z}_{j}\frac{\partial}{\partial \overline{z}_{j}}-\alpha\beta \right].
   \end{equation}
   There is a radial-tangential expression  for $\Delta_{\alpha,\beta}$ obtained in \cite{ge} and we state it as follows:
  \begin{equation*}%\label{2.1}
  \begin{split}
\Delta_{\alpha,\beta}
=&4(1-|z|^{2})\left[\frac{1-|z|^{2}}{|z|^{2}}R\overline{R}-\frac{1}{|z|^{2}}\mathcal{L}'_{0}+\frac{n-1}{2}\cdot\frac{1}{|z|^{2}}(R+\overline{R})
+\alpha R+\beta\bar{R}-\alpha\beta\right],
\end{split}
\end{equation*}
where
  \begin{equation*}%\label{2.1}
  \begin{split}
R=\sum^{n}_{j=1}z_{j}\frac{\partial}{\partial z_{j}},\;\;\overline{R}=\sum^{n}_{j=1}\overline{z}_{j}\frac{\partial}{\partial \overline{z}_{j}}
\end{split}
\end{equation*}
and $\mathcal{L}'_{0}$ is the Folland-Stein operator \cite{fs} defined as follows:
\begin{equation*}
  \mathcal{L}'_{0}=-\frac{1}{2}\sum_{j<k}\left(M_{jk}\overline{M}_{jk}+\overline{M}_{jk}M_{jk}\right),\;\;M_{j,k}=z_{j}\partial_{\overline{z}_{k}}
-\overline{z}_{k}\partial_{z_{j}}.
\end{equation*}

\subsection{ CR invariant differential operators on the Heisenberg group and CR sphere}
It is known that  the CR invariant
sub-Laplacian of Jerison and Lee (\cite{jl1}-\cite{jl4})  in  CR geometry  plays a role analogous to that of the
conformal Laplacian in Riemaniann  geometry. In the case of CR sphere  $\mathbb{S}^{2n-1}$, the  CR invariant
sub-Laplacian is defined as
$ \mathcal{L}'_{0}+\frac{(n-1)^{2}}{4}$.
Let $\partial\mathcal{C}$ be the restriction of Cayley  transform $\mathcal{C}$ defined in (\ref{b2.5}) to the boundary.
The relationship between $\mathcal{L}_{0}=-\Delta_{b}$, the sub-Laplacian on the Heisenberg group, and the $ \mathcal{L}'_{0}+\frac{(n-1)^{2}}{4}$
is
\[
\left(\mathcal{L}'_{0}+\frac{(n-1)^{2}}{4}\right)\left(|J_{\partial\mathcal{C}}|^{\frac{Q-2}{2Q}}(F\circ \partial\mathcal{C})\right)=
|J_{\partial\mathcal{C}}|^{\frac{Q-2}{2Q}}\mathcal{L}_{0}F,\;\;F\in C^{\infty}(\mathbb{H}^{2n-1}),
\]
where $J_{\partial\mathcal{C}}$ is the Jacobian determinant of $\partial\mathcal{C}$ and $Q=2n$ is the homogeneous dimension.

The CR invariant differential operators for high order are the product of  the Folland-Stein
operators. In fact, if we denote the  Folland-Stein
operators $\mathcal{L}_{\alpha}$ by
$$\mathcal{L}_{\alpha}=\mathcal{L}_{0}+i\alpha T,$$
then the  CR invariant differential operator of $2k$ is, up to a constant,
\begin{equation}\label{}
  P_{2k}=\prod^{k}_{j=1}\mathcal{L}_{k+1-2j}.
\end{equation}
A simple calculation shows
$$P_{1}=\mathcal{L}_{0}=-\Delta_{b},\;\;P_{2}=\mathcal{L}^{2}_{0}+T^{2}=\Delta^{2}_{b}+T^{2}.$$
We remark that the operator $P_{2k}$ has been firstly obtained  by Cowling \cite{co1} via the computation of the group Fourier transform of $(|z|^{4}+t^{2})^{2k-Q}$.

Similarly, there is an analogous operator  on the CR sphere. Let
$$\mathcal{T}=\frac{i}{2}(R-\bar{R})$$
be the transversal direction. The  CR invariant differential operator of order $2k$ on $\mathbb{S}^{2n-1}$ is given by, up to a constant,
$$P'_{2k}=\prod^{k}_{j=0}\left(\mathcal{L}'_{0}+\frac{(n-1)^{2}}{4}-\frac{(k+1-2j)^{2}}{4}-(k+1-2j)i\mathcal{T}\right).$$
The relationship between $P_{2k}$ and the $P'_{2k}$
is (see \cite{gra3})
\[
P'_{2k}\left(|J_{\partial\mathcal{C}}|^{\frac{Q-2k}{2Q}}(F\circ \partial\mathcal{C})\right)=
|J_{\partial\mathcal{C}}|^{\frac{Q-2k}{2Q}}P_{2k}F,\;\;F\in C^{\infty}(\mathbb{H}^{2n-1}).
\]
For the conformally invariant sharp  Sobolev inequality on the Heisenberg group and CR sphere, we refer to Jerison-Lee \cite{jl3} and Frank-Lieb \cite{fr}.

\subsection{The automorphisms}
For each $a\in\mathbb{B}_{\mathbb{C}}^{n}$, we define the  automorphisms $\varphi_{a}$ of $\mathbb{B}_{\mathbb{C}}^{n}$ by (see e.g. \cite{ru})
\[
\varphi_{a}(z)=\frac{a-P_{a}(z)-\sqrt{1-|a|^{2}}Q_{a}(z)}{1-(z,a)},
\]
where
\[
P_{a}(z)=\left\{
           \begin{array}{ll}
             \frac{(z,a)}{|a|^{2}}a, & \hbox{$a\neq0$;} \\
             0, & \hbox{$a=0$,}
           \end{array}
         \right.
\]
and $Q_{a}(z)=z-P_{a}(z)$. It is easy to check $\varphi_{a}(0)=a$ and $\varphi_{a}(a)=0$. Furthermore, $\varphi_{a}$ has the following properties:
\begin{equation}\label{2.6}
\begin{split}
1-|\varphi_{a}(z)|^{2}=&\frac{(1-|a|^{2})(1-|z|^{2})}{|1-(z,a)|^{2}}\\
\sinh\varphi_{a}(z)=&\frac{|\varphi_{a}(z)|}{\sqrt{1-|\varphi_{a}(z)|^{2}}}=\frac{\sqrt{|z-a|^{2}+|(z,a)|^{2}-|z|^{2}|a|^{2}}}{\sqrt{(1-|a|^{2})(1-|z|^{2})}};\\
\cosh\varphi_{a}(z)=&\frac{1}{\sqrt{1-|\varphi_{a}(z)|^{2}}}=\frac{|1-(z,a)|}{\sqrt{(1-|a|^{2})(1-|z|^{2})}}.
\end{split}
\end{equation}
Since $\varphi_{a}$ is an isometry,  the
distance  from $z$ to $a$ is
\begin{equation*}
\rho(x,a)=\rho(\varphi_{a}(z))=\frac{1}{2}\log\frac{1+|\varphi_{a}(z)|}{1-|\varphi_{a}(z)|}.
\end{equation*}
The convolution of
measurable functions $f$ and $g$ on $\mathbb{B}_{\mathbb{C}}^{n}$ is
\begin{equation}\label{2.7}
(f\ast g)(z)=\int_{\mathbb{B}_{\mathbb{C}}^{n}}f(\varphi_{z}(w))g(w) dV(w)
\end{equation}
provided this integral exists. It is easy to check that if $f$ is radial, then
\begin{equation}\label{2.8}
f\ast g=g\ast f.
\end{equation}

\subsection{Helgason-Fourier transform on the complex hyperbolic spaces}
In this subsection we recall some  basics of Helgason-Fourier analysis on complex hyperbolic spaces in term of ball model and refer the reader  to \cite{ gv,he,he2, te, th, ow} for more information about
this subject.

It is known that the Poisson kernel on $\mathbb{B}_{\mathbb{C}}^{n}$ is given by (see \cite{jo}, Lemma 2.1)
\[
\left(\frac{1-|z|^{2}}{|1-(z,\zeta)|^{2}}\right)^{n}, \;\; x\in \mathbb{B}_{\mathbb{C}}^{n},\;\;\zeta\in\mathbb{S}^{2n-1}.
\]
Therefore, if we set
\[
e_{\lambda,\zeta}(z)=\left(\frac{1-|z|^{2}}{|1-(z,\zeta)|^{2}}\right)^{\frac{n+i\lambda}{2}}, \;\; x\in \mathbb{B}_{\mathbb{C}}^{n},\;\;\lambda\in\mathbb{R},\;\;\zeta\in\mathbb{S}^{2n-1},
\]
then the Helgason-Fourier transform of a function  $f$  on $\mathbb{B}_{\mathbb{C}}^{n}$ can be defined as
\[
\widehat{f}(\lambda,\zeta)=\int_{\mathbb{B}_{\mathbb{C}}^{n}} f(z)e_{-\lambda,\zeta}(z)dV
\]
provided this integral exists. It is easy to check that if $f,g\in C^{\infty}_{0}(\mathbb{B}_{\mathbb{C}}^{n})$ and $g$ is radial, then
 $$\widehat{(f\ast g)}=\widehat{f}\cdot\widehat{g}.$$
The following inversion formula holds for $f\in C^{\infty}_{0}(\mathbb{B}_{\mathbb{C}}^{n})$:
\[
f(z)=C_{n}\int^{+\infty}_{-\infty}\int_{\mathbb{S}^{2n-1}} \widehat{f}(\lambda,\zeta)e_{\lambda,\zeta}(z)|\mathfrak{c}(\lambda)|^{-2}d\lambda d\sigma(\varsigma),
\]
where $C_{n}$ is a positive constant  and $\mathfrak{c}(\lambda)$ is the  Harish-Chandra $\mathfrak{c}$-function given by
\[
\mathfrak{c}(\lambda)=\frac{2^{n/2-i\lambda}\Gamma(n/2)\Gamma(i\lambda)}{\Gamma(\frac{n/2+i\lambda}{2})\Gamma(\frac{n/2+i\lambda}{2})}.
\]
Similarly, there holds the Plancherel formula
\begin{equation}\label{2.11}
\int_{\mathbb{B}_{\mathbb{C}}^{n}}|f(z)|^{2}dV=C_{n}\int^{+\infty}_{-\infty}\int_{\mathbb{S}^{2n-1}}|\widehat{f}(\lambda,\zeta)|^{2}
|\mathfrak{c}(\lambda)|^{-2}d\lambda d\sigma(\varsigma).
\end{equation}

Since $e_{\lambda,\zeta}(z)$ is an eigenfunction of $\Delta_{\mathbb{B}}$ with eigenvalue $-n^{2}-\lambda^{2}$, we have, for
$f\in C^{\infty}_{0}(\mathbb{B}_{\mathbb{C}}^{n})$,
\[
\widehat{\Delta_{\mathbb{B}}f}(\lambda,\zeta)=-(n^{2}+\lambda^{2})\widehat{f}(\lambda,\zeta).
\]
Therefore, in analogy with the Euclidean setting, we define the fractional
Laplacian on hyperbolic space as follows:
\begin{equation}\label{2.12}
\widehat{(-\Delta_{\mathbb{B}})^{\gamma}f}(\lambda,\zeta)=(n^{2}+\lambda^{2})^{\gamma}\widehat{f}(\lambda,\zeta),\;\;\gamma\in \mathbb{R}.
\end{equation}
The fractional Sobolev spaces  on  $\mathbb{B}_{\mathbb{C}}^{n}$ are defined by
$$W^{\alpha,p}(\mathbb{B}_{\mathbb{C}}^{n})=(\textrm{I}-\Delta_{\mathbb{B}})^{-\frac{\alpha}{2}}L^{p}\;\;\;\;(1<p<\infty,\;\alpha\in \mathbb{R}).$$
In \cite{an}, J.-P. Anker proved a  H\"ormander-Mikhlin type multiplier theorem in the context of
Riemannian symmetric spaces of the
noncompact type. In the case of complex
hyperbolic space, the H\"ormander-Mikhlin type multiplier theorem  reads as follows:
\begin{theorem}\label{th2.1}
Let $1<p<\infty$,  $\kappa$ be a   tempered radial distribution on $\mathbb{B}_{\mathbb{C}}^{n}$ and  $m$ be its  Fourier transformation. Then $f\ast\kappa$
is a bounded operator on $L^{p}(\mathbb{B}_{\mathbb{C}}^{n})$, provided that

(i) $m$ extends to a holomorphic function inside the tube $\mathfrak{T}=\{\lambda\in \mathbb{C}:|\textrm{Im}\, \, \lambda|\leq2n|\frac{1}{p}-\frac{1}{2}|\}$;

(ii) $\partial^{i}_{\lambda} m(\lambda)\,\, (i=0,1,\cdots,\left[2n|\frac{1}{p}-\frac{1}{2}|\right]+1)$ extends continuously to the whole of
$\mathfrak{T}$, with
\[
\sup_{\lambda\in\mathfrak{T}}(1+|\lambda|)^{i}|\partial^{i}_{\lambda} m(\lambda)|<\infty
\]
\end{theorem}

Observe  that $(-\Delta_{\mathbb{B}}-n^{2}+\zeta^{2})^{-\frac{\alpha}{2}}$ has the symbol $(\lambda^{2}+\zeta^{2})^{-\frac{\alpha}{2}}$ which can be
extended to a holomorphic function inside the tube $
\{\lambda\in \mathbb{C}:|\textrm{Im}\lambda|\leq \zeta\}$. Therefore, by Theorem \ref{th2.1}, there exists a constant $C>0$ such that
for $f\in L^{p}(\mathbb{B}_{\mathbb{C}}^{n})$
\[
\|(-\Delta_{\mathbb{B}}-n^{2}+\zeta^{2})^{-\frac{\alpha}{2}}f\|_{p}\leq C\|f\|_{p}
\]
provided  $\zeta>2n|\frac{1}{2}-\frac{1}{p}|$ and $\alpha>0$. This leads to the following Sobolev type inequality immediately.
\begin{corollary}\label{lm2.2}
Let $1<p<\infty$,  $\alpha>0$ and $\zeta>2n|\frac{1}{2}-\frac{1}{p}|$. Then there exists a positive constant
$C$ such that for $f\in W^{\alpha,p}(\mathbb{B}_{\mathbb{C}}^{n})$,
\begin{equation*}%\label{2.13}
C\|f\|_{p}\leq\|(-\Delta_{\mathbb{B}}-n^{2}+\zeta^{2})^{\frac{\alpha}{2}}f\|_{p}.
\end{equation*}
\end{corollary}

For more information about the fractional
Laplacian on  noncompact symmetric spaces, we refer to \cite{an}, \cite{an2},\cite{co}.

\subsection{Funk-Hecke formula}It is known that $L^{2}(\mathbb{S}^{2n-1})$ can be decomposed as follows
\[
L^{2}(\mathbb{S}^{2n-1})=\bigoplus^{\infty}_{j,k=0}\mathcal{H}_{j,k},
\]
where $\mathcal{H}_{j,k}$ is the space of
restrictions to $\mathbb{S}^{2n-1}$ of harmonic polynomials $p(z,\bar{z})$ which are homogeneous of degree $j$ in $z$ and degree $k$ in $\bar{z}$.
The
Funk-Hecke formula of Frank and Lieb reads as follow (see \cite{fr}):
\begin{equation}\label{2.13}
  \begin{split}
\int_{\mathbb{S}^{2n-1}}&K((\xi,\eta))Y_{j,k}(\eta)d\sigma(\eta)=\lambda_{j,k}Y_{j,k},\;\; Y_{j,k}\in \mathcal{H}_{j,k},\\
\lambda_{j,k}=&\frac{\pi^{m}m!}{2^{n-1+|j-k|/2}(m+n-2)!}\int^{1}_{-1}\left(\int^{\pi}_{-\pi}K(e^{-i\theta}\sqrt{(1+t)/2})e^{i(j-k)\theta}d\theta\right)\cdot\\ &(1-t)^{n-2}(1+t)^{|j-k|/2}P_{m}^{(n-2,|j-k|)}(t)dt,\;\;\;m=\min\{j,k\},
\end{split}
\end{equation}
where $K$ is an integrable function on the unit ball in $\mathbb{C}$ and $P^{(\alpha,\beta)}_{m}$ is the Jacobi
polynomials.
Next, we will prove the following
\begin{proposition}\label{pr2.3}
Let $0<\alpha<2n$. There holds, for $0<r<1$ and   $\xi\in\mathbb{S}^{2n-1}$,
\begin{equation}\label{1}
  \int_{\mathbb{S}^{2n-1}}\frac{1}{|1-(r\xi,\eta)|^{\alpha}}d\sigma(\eta)=\frac{2\pi}{\Gamma(n)}F(\alpha/2,\alpha/2;n;r^{2}),
\end{equation}
where $F(a,b;c;z)$ is the hypergeometric function
defined by
  \begin{equation}\label{2.1}
  \begin{split}
F(a,b;c;z)=\sum^{\infty}_{k=0}\frac{(a)_{k}(b)_{k}}{(c)_{k}}\frac{z^{k}}{k!},\;\;c\neq0,-1,\cdots,-n,\cdots,
\end{split}
\end{equation}
  and $(a)_{k}$ is the rising Pochhammer symbol defined by
$$
(a)_{0}=1,\;(a)_{k}=a(a+1)\cdots(a+k-1), \;k\geq1.
$$
\end{proposition}
\begin{proof}
By Funk-Hecke formula, we have
\begin{equation}\label{2.15}
  \begin{split}
&\int_{\mathbb{S}^{2n-1}}\frac{1}{|1-(r\xi,\eta)|^{\alpha}}d\sigma(\eta)\\
=&\frac{1}{2^{n-1}(n-2)!}\int^{1}_{-1}
\left(\int^{\pi}_{-\pi}\frac{1}{|1-r\sqrt{\frac{1+t}{2}}
e^{-i\theta}|^{\alpha}}d\theta\right)(1-t)^{n-2}P^{n-2,0}_{0}(t)dt.
\end{split}
\end{equation}
We compute
\begin{equation}\label{2.15}
  \begin{split}
\int^{\pi}_{-\pi}\frac{1}{|1-r\sqrt{\frac{1+t}{2}}
e^{-i\theta}|^{\alpha}}d\theta=&\int^{\pi}_{-\pi}
\left(1-2r\sqrt{\frac{1+t}{2}}\cos\theta+\frac{1+t}{2}r^{2}\right)^{-\frac{\alpha}{2}}d\theta\\
=&\frac{2\pi}{\Gamma^{2}(\alpha/2)}\sum^{\infty}_{\mu=0}
\frac{\Gamma^{2}(\alpha/2+\mu)}{(\mu!)^{2}}r^{2\mu}\left(\frac{1+t}{2}\right)^{\mu}.
\end{split}
\end{equation}
Here we use the fact (see \cite{fr}, (5.11))
\[
\int^{\pi}_{-\pi}
\left(1-2r\cos\theta+r^{2}\right)^{-\frac{\alpha}{2}}d\theta=\frac{2\pi}{\Gamma^{2}(\alpha/2)}\sum^{\infty}_{\mu=0}
\frac{\Gamma^{2}(\alpha/2+\mu)}{(\mu!)^{2}}r^{2\mu}.
\]
Therefore, we have
\begin{equation}\label{2.15}
  \begin{split}
&\int_{\mathbb{S}^{2n-1}}\frac{1}{|1-(r\xi,\eta)|^{\alpha}}d\sigma(\eta)\\
=&\frac{1}{2^{n-1}(n-2)!}
\frac{2\pi}{\Gamma^{2}(\alpha/2)}\sum^{\infty}_{\mu=0}\frac{\Gamma^{2}(\alpha/2+\mu)}{(\mu!)^{2}}r^{2\mu}
\int^{1}_{-1}\left(\frac{1+t}{2}\right)^{\mu}
(1-t)^{n-2}P^{n-2,0}_{0}(t)dt\\
=&\frac{2\pi}{\Gamma^{2}(\alpha/2)}\sum^{\infty}_{\mu=0}\frac{\Gamma^{2}(\alpha/2+\mu)}{(n-1+\mu)!}\cdot\frac{r^{2\mu}}{\mu!}.
\end{split}
\end{equation}
To get the last equation, we use the fact (see \cite{fr}, (5.12))
\begin{equation*}%\label{2.15}
  \begin{split}
&\int^{t}_{-1}(1-t)^{n-1}(1+t)^{|j-k|+\mu}P_{m}^{(n-1,|j-k|)}(t)dt\\
=&\left\{
   \begin{array}{ll}
     0, & \hbox{if\;$\mu<m$;} \\
     2^{|j-k|+n+\mu}\frac{\mu!}{m!(\mu-m)!}\frac{(|j-k|+\mu)!(m+n-1)!}{(|j-k|+m+n+\mu)!}, & \hbox{if\; $\mu\geq m$.}
   \end{array}
 \right.
\end{split}
\end{equation*}
Therefore, we have, by  (\ref{2.15}),
\begin{equation*}%\label{2.15}
  \begin{split}
\int_{\mathbb{S}^{2n-1}}\frac{1}{|1-(r\xi,\eta)|^{\alpha}}d\sigma(\eta)=&\frac{2\pi}{\Gamma(n)}F(\alpha/2,\alpha/2;n;r^{2}).
\end{split}
\end{equation*}
This completes the proof.

\end{proof}

\section{ A factorization theorem for the  Operators on complex hyperbolic space: Proof of Theorem \ref{th1.6}}

Factorization theorem plays an important role in establishing the Hardy-Sobolev-Maz'ya inequalities.
The main purpose of   this section is to prove Theorem \ref{th1.6}.

 Firstly, we consider the  Siegel domain model.
The proof depends on the following three lemmas.
\begin{lemma}\label{lm9.2}
Let $a\in\mathbb{R}$ and $f\in C^{\infty}(\mathcal{U}^{n})$. There holds
\begin{equation}\label{9.2}
\begin{split}
&\left[\varrho\partial_{\varrho\varrho}+a\partial_{\varrho}+\varrho T^{2}+ \Delta_{b}\right](\varrho^{\frac{1-n-a}{2}} f) \\
=&\varrho^{-\frac{1+n+a}{2}}\left\{\varrho[\varrho(\partial_{\varrho\varrho}+T^{2})+ \Delta_{b}-(n-1)\partial_{\varrho}]+\frac{n^{2}}{4}-\frac{(a-1)^{2}}{4}\right\}f,
\end{split}
 \end{equation}

\end{lemma}
\begin{proof}
A simple calculation shows, for each $\beta\in\mathbb{R}$, there holds
  \begin{equation}\label{9.3}
\begin{split}
&\varrho^{\beta+1}\left[\varrho\partial_{\varrho\varrho}+a\partial_{\varrho}+\varrho T^{2}+ \Delta_{b}
\right](\varrho^{-\beta} f)\\
=&\varrho[\varrho(\partial_{\varrho\varrho}+T^{2})+ \Delta_{b}-(2\beta-a)\partial_{\varrho}]f+\beta(\beta+1-a)f.
\end{split}
 \end{equation}
 Choosing $\beta =\frac{n-1+a}{2}$ in (\ref{9.3}), we have
  \begin{equation*}
\begin{split}
&\varrho^{\frac{1+n+a}{2}}
\left[\varrho\partial_{\varrho\varrho}+a\partial_{\varrho}+\varrho T^{2}+ \Delta_{b}\right](\varrho^{\frac{1-n-a}{2}} f)\\
=&\left\{\varrho[\varrho(\partial_{\varrho\varrho}+T^{2})+ \Delta_{b}-(n-1)\partial_{\varrho}]+\frac{n^{2}}{4}-\frac{(a-1)^{2}}{4}\right\}f.
\end{split}
 \end{equation*}
 The desired result follows.
 \end{proof}

 \begin{lemma}\label{lm9.3}
Let $\beta\in\mathbb{R}$.  There holds
\begin{equation}\label{9.4}
\begin{split}
&\left[\varrho\partial_{\varrho\varrho}+(a+\beta)\partial_{\varrho}+\varrho T^{2}+ \Delta_{b}\right]
\left\{\left[\varrho\partial_{\varrho\varrho}+(a-1)\partial_{\varrho}+\varrho T^{2}+ \Delta_{b}\right]
^{2}+(\beta-1)^{2}T^{2}\right\} \\
=&\left\{\left[\varrho\partial_{\varrho\varrho}+a\partial_{\varrho}+\varrho T^{2}+ \Delta_{b}\right]
^{2}+\beta^{2}T^{2}\right\}\left[\varrho\partial_{\varrho\varrho}+(a+\beta-2)\partial_{\varrho}+\varrho T^{2}+ \Delta_{b}\right]
\end{split}
 \end{equation}
\end{lemma}
\begin{proof}
Since
\begin{equation}\label{9.5}
\begin{split}
&\partial_{\varrho}
\left[\varrho\partial_{\varrho\varrho}+(a-1)\partial_{\varrho}+\varrho T^{2}+ \Delta_{b}\right]
 =\left[\varrho\partial_{\varrho\varrho}+(a-1)\partial_{\varrho}+\varrho T^{2}+ \Delta_{b}\right]\partial_{\varrho}+T^{2},
\end{split}
 \end{equation}
we have
\begin{equation}\label{9.6}
\begin{split}
&\partial_{\varrho}
\left[\varrho\partial_{\varrho\varrho}+(a-1)\partial_{\varrho}+\varrho T^{2}+ \Delta_{b}\right]^{2}\\
=&\left[\varrho\partial_{\varrho\varrho}+a\partial_{\varrho}+\varrho T^{2}+ \Delta_{b}\right]\partial_{\varrho}\left[\varrho\partial_{\varrho\varrho}+(a-1)\partial_{\varrho}+\varrho T^{2}+ \Delta_{b}\right]\\
&
+\left[\varrho\partial_{\varrho\varrho}+(a-1)\partial_{\varrho}+\varrho T^{2}+ \Delta_{b}\right]T^{2}\\
=&\left[\varrho\partial_{\varrho\varrho}+a\partial_{\varrho}+\varrho T^{2}+ \Delta_{b}\right]^{2}\partial_{\varrho}+
\left[\varrho\partial_{\varrho\varrho}+a\partial_{\varrho}+\varrho T^{2}+ \Delta_{b}\right]T^{2}\\
&+\left[\varrho\partial_{\varrho\varrho}+(a-1)\partial_{\varrho}+\varrho T^{2}+ \Delta_{b}\right]T^{2}\\
=&\left[\varrho\partial_{\varrho\varrho}+a\partial_{\varrho}+\varrho T^{2}+ \Delta_{b}\right]^{2}\partial_{\varrho}+
2\left[\varrho\partial_{\varrho\varrho}+a\partial_{\varrho}+\varrho T^{2}+ \Delta_{b}\right]T^{2}-T^{2}\partial_{\varrho}.
\end{split}
 \end{equation}
 Similarly,
 \begin{equation}\label{9.7}
\begin{split}
&\left[\varrho\partial_{\varrho\varrho}+(a-1)\partial_{\varrho}+\varrho T^{2}+ \Delta_{b}\right]
^{2}\\
=&\left[\varrho\partial_{\varrho\varrho}+a\partial_{\varrho}+\varrho T^{2}+ \Delta_{b}\right]
\left[\varrho\partial_{\varrho\varrho}+(a-1)\partial_{\varrho}+\varrho T^{2}+ \Delta_{b}\right]-\\
&\partial_{\varrho}\left[\varrho\partial_{\varrho\varrho}+(a-1)\partial_{\varrho}+\varrho T^{2}+ \Delta_{b}\right]\\
=&\left[\varrho\partial_{\varrho\varrho}+a\partial_{\varrho}+\varrho T^{2}+ \Delta_{b}\right]^{2}-2\left[\varrho\partial_{\varrho\varrho}+a\partial_{\varrho}+\varrho T^{2}+ \Delta_{b}\right]\partial_{\varrho}
-T^{2}.
\end{split}
 \end{equation}
Therefore, we obtain, by using (\ref{9.6}) and (\ref{9.7}),
\begin{equation*}%\label{9.8}
\begin{split}
&\left[\varrho\partial_{\varrho\varrho}+(a+\beta)\partial_{\varrho}+\varrho T^{2}+ \Delta_{b}\right]
\left\{\left[\varrho\partial_{\varrho\varrho}+(a-1)\partial_{\varrho}+\varrho T^{2}+ \Delta_{b}\right]
^{2}+(\beta-1)^{2}T^{2}\right\} \\
=&\left[\varrho\partial_{\varrho\varrho}+a\partial_{\varrho}+\varrho T^{2}+ \Delta_{b}\right]
\left\{\left[\varrho\partial_{\varrho\varrho}+(a-1)\partial_{\varrho}+\varrho T^{2}+ \Delta_{b}\right]
^{2}+(\beta-1)^{2}T^{2}\right\} +
\\
&\beta\partial_{\varrho}\left\{\left[\varrho\partial_{\varrho\varrho}+(a-1)\partial_{\varrho}+\varrho T^{2}+ \Delta_{b}\right]
^{2}+(\beta-1)^{2}T^{2}\right\} \\
=&\left[\varrho\partial_{\varrho\varrho}+a\partial_{\varrho}+\varrho T^{2}+ \Delta_{b}\right]\cdot\\
&\left\{\left[\varrho\partial_{\varrho\varrho}+a\partial_{\varrho}+\varrho T^{2}+ \Delta_{b}\right]^{2}-2\left[\varrho\partial_{\varrho\varrho}+a\partial_{\varrho}+\varrho T^{2}+ \Delta_{b}\right]\partial_{\varrho}
+\beta(\beta-2)T^{2}\right\}\\
&+\beta\left\{
\left[\varrho\partial_{\varrho\varrho}+a\partial_{\varrho}+\varrho T^{2}+ \Delta_{b}\right]^{2}\partial_{\varrho}+
2\left[\varrho\partial_{\varrho\varrho}+a\partial_{\varrho}+\varrho T^{2}+ \Delta_{b}\right]T^{2}+\beta(\beta-2)T^{2}\partial_{\varrho}\right\}\\
=&\left\{\left[\varrho\partial_{\varrho\varrho}+a\partial_{\varrho}+\varrho T^{2}+ \Delta_{b}\right]
^{2}+\beta^{2}T^{2}\right\}\left[\varrho\partial_{\varrho\varrho}+(a+\beta-2)\partial_{\varrho}+\varrho T^{2}+ \Delta_{b}\right].
\end{split}
 \end{equation*}
This completes the proof of Lemma \ref{lm9.3}.
\end{proof}

\begin{lemma}\label{lm9.4}
There holds, for $k\in\mathbb{N}\setminus\{0\}$,
\begin{equation}\label{9.8}
\begin{split}
&\left[\varrho\partial_{\varrho\varrho}+(a+2k)\partial_{\varrho}+\varrho T^{2}+ \Delta_{b}\right]\prod^{k}_{j=1}\left\{\left[\varrho\partial_{\varrho\varrho}+a\partial_{\varrho}+\varrho T^{2}+ \Delta_{b}\right]
^{2}+(2j-1)^{2}T^{2}\right\}\\
=&(\varrho\partial_{\varrho\varrho}+a\partial_{\varrho}+\varrho T^{2}+ \Delta_{b})\prod^{k}_{j=1}\left\{\left[\varrho\partial_{\varrho\varrho}+a\partial_{\varrho}+\varrho T^{2}+ \Delta_{b}\right]
^{2}+4j^{2}T^{2}\right\}
\end{split}
 \end{equation}
 and
 \begin{equation}\label{9.9}
\begin{split}
&\left[\varrho\partial_{\varrho\varrho}+(a+2k)\partial_{\varrho}+\varrho T^{2}+ \Delta_{b}\right]\\
&\left\{
(\varrho\partial_{\varrho\varrho}+a\partial_{\varrho}+\varrho T^{2}+ \Delta_{b})
\prod^{k-1}_{j=1}\left[\left(\varrho\partial_{\varrho\varrho}+a\partial_{\varrho}+\varrho T^{2}+ \Delta_{b}\right)
^{2}+4j^{2}T^{2}\right]\right\}\\
=&\prod^{k}_{j=1}\left\{\left[\varrho\partial_{\varrho\varrho}+a\partial_{\varrho}+\varrho T^{2}+ \Delta_{b}\right]
^{2}+(2j-1)^{2}T^{2}\right\}.
\end{split}
 \end{equation}
\end{lemma}
\begin{proof}
By Lemma \ref{lm9.3}, we have
 \begin{equation*}%\label{9.8}
\begin{split}
&\left[\varrho\partial_{\varrho\varrho}+(a+2k)\partial_{\varrho}+\varrho T^{2}+ \Delta_{b}\right]\prod^{k}_{j=1}\left\{\left[\varrho\partial_{\varrho\varrho}+a\partial_{\varrho}+\varrho T^{2}+ \Delta_{b}\right]
^{2}+(2j-1)^{2}T^{2}\right\}\\
=&\left[\varrho\partial_{\varrho\varrho}+(a+2k)\partial_{\varrho}+\varrho T^{2}+ \Delta_{b}\right]\left\{\left[\varrho\partial_{\varrho\varrho}+a\partial_{\varrho}+\varrho T^{2}+ \Delta_{b}\right]
^{2}+(2k-1)^{2}T^{2}\right\}\\
&\prod^{k-1}_{j=1}\left\{\left[\varrho\partial_{\varrho\varrho}+a\partial_{\varrho}+\varrho T^{2}+ \Delta_{b}\right]
^{2}+(2j-1)^{2}T^{2}\right\}\\
=&\left\{\left[\varrho\partial_{\varrho\varrho}+a\partial_{\varrho}+\varrho T^{2}+ \Delta_{b}\right]
^{2}+4k^{2}T^{2}\right\}\\
&\left[\varrho\partial_{\varrho\varrho}+(a+2k-2)\partial_{\varrho}+\varrho T^{2}+ \Delta_{b}\right]\prod^{k-1}_{j=1}\left\{\left[\varrho\partial_{\varrho\varrho}+a\partial_{\varrho}+\varrho T^{2}+ \Delta_{b}\right]
^{2}+(2j-1)^{2}T^{2}\right\}.
\end{split}
 \end{equation*}
Repeating this process in a suitable manner we get (\ref{9.8}). The proof of (\ref{9.9}) is similar and we omit it. The proof of Lemma \ref{lm9.4}
is thereby completed.
\end{proof}

\begin{proposition}\label{pr8.4}
Let $a\in\mathbb{R}$ and $f\in C^{\infty}(\mathcal{U}^{n})$. We have,  for $k\in\mathbb{N}\setminus\{0\}$,
\begin{equation}\label{9.1}
\begin{split}
&4^{k}\varrho^{\frac{k+n+a}{2}}\prod^{k}_{j=1}\left[\varrho\partial_{\varrho\varrho}+a\partial_{\varrho}+\varrho T^{2}+ \Delta_{b}-i(k+1-2j)T\right](\varrho^{\frac{k-n-a}{2}} f) \\
=&\prod^{k}_{j=1}\left[\Delta_{\mathbb{B}}+n^{2}-(a-k+2j-2)^{2}\right]f.
\end{split}
 \end{equation}
\end{proposition}

\begin{proof}
It is enough to show
\begin{equation}\label{9.10}
\begin{split}
&\prod^{k}_{j=1}\left[\varrho\partial_{\varrho\varrho}+a\partial_{\varrho}+\varrho T^{2}+ \Delta_{b}-i(k+1-2j)T\right](\varrho^{\frac{k-n-a}{2}} f) \\
=&\varrho^{-\frac{k+n+a}{2}}\prod^{k}_{j=1}\left\{\varrho[\varrho(\partial_{\varrho\varrho}+T^{2})+ \Delta_{b}-(n-1)\partial_{\varrho}]+\frac{n^{2}}{4}-\frac{(a-k+2j-2)^{2}}{4}\right\}f,
\end{split}
 \end{equation}
 We shall prove (\ref{9.10}) by induction. By Lemma \ref{lm9.2},  (\ref{9.10}) is valid for $k=1$.
Now assume  (\ref{9.10}) is valid for $k=l$, i.e.,
 \begin{equation}\label{9.11}
\begin{split}
&\prod^{l}_{j=1}\left[\varrho\partial_{\varrho\varrho}+a\partial_{\varrho}+\varrho T^{2}+ \Delta_{b}-i(l+1-2j)T\right](\varrho^{\frac{l-n-a}{2}} f) \\
=&\varrho^{-\frac{l+n+a}{2}}\prod^{l}_{j=1}\left\{\varrho[\varrho(\partial_{\varrho\varrho}+T^{2})+ \Delta_{b}-(n-1)\partial_{\varrho}]+\frac{n^{2}}{4}-\frac{(a-l+2j-2)^{2}}{4}\right\}f.
\end{split}
 \end{equation}
Replacing $a$ by $a-1$ in (\ref{9.11}), we get
  \begin{equation}\label{9.12}
\begin{split}
&\prod^{l}_{j=1}\left[\varrho\partial_{\varrho\varrho}+(a-1)\partial_{\varrho}+\varrho T^{2}+ \Delta_{b}-i(l+1-2j)T\right](\varrho^{\frac{l+1-n-a}{2}} f) \\
=&\varrho^{-\frac{l+n+a-1}{2}}\prod^{l}_{j=1}\left\{\varrho[\varrho(\partial_{\varrho\varrho}+T^{2})+ \Delta_{b}-(n-1)\partial_{\varrho}]+\frac{n^{2}}{4}-\frac{(a-l+2j-3)^{2}}{4}\right\}f.
\end{split}
 \end{equation}
 If $l$ is even, then by (\ref{9.8}), we have
  \begin{equation*}%\label{9.12}
\begin{split}
&\left[\varrho\partial_{\varrho\varrho}+(a+l)\partial_{\varrho}+\varrho T^{2}+ \Delta_{b}\right]\prod^{l}_{j=1}\left[\varrho\partial_{\varrho\varrho}+(a-1)\partial_{\varrho}+\varrho T^{2}+ \Delta_{b}-i(l+1-2j)T\right] \\
=&\left[\varrho\partial_{\varrho\varrho}+(a+l)\partial_{\varrho}+\varrho T^{2}+ \Delta_{b}\right]\prod^{l/2}_{j=1}
\left\{\left[\varrho\partial_{\varrho\varrho}+a\partial_{\varrho}+\varrho T^{2}+ \Delta_{b}\right]
^{2}+(2j-1)^{2}T^{2}\right\}\\
=&(\varrho\partial_{\varrho\varrho}+a\partial_{\varrho}+\varrho T^{2}+ \Delta_{b})\prod^{l/2}_{j=1}\left\{\left[\varrho\partial_{\varrho\varrho}+a\partial_{\varrho}+\varrho T^{2}+ \Delta_{b}\right]
^{2}+4j^{2}T^{2}\right\}\\
=&\prod^{l+1}_{j=1}\left[\varrho\partial_{\varrho\varrho}+a\partial_{\varrho}+\varrho T^{2}+ \Delta_{b}-i(l+2-2j)T\right]
\end{split}
 \end{equation*}
and by Lemma \ref{lm9.2},

\begin{equation*}%\label{9.8}
\begin{split}
&\left[\varrho\partial_{\varrho\varrho}+(a+l)\partial_{\varrho}+\varrho T^{2}+ \Delta_{b}\right]\\
&\left\{\varrho^{-\frac{l+n+a-1}{2}}\prod^{l}_{j=1}\left\{\varrho[\varrho(\partial_{\varrho\varrho}+T^{2})+ \Delta_{b}-(n-1)\partial_{\varrho}]+\frac{n^{2}}{4}-\frac{(a-l+2j-3)^{2}}{4}\right\}f\right\}\\
=&\varrho^{-\frac{l+n+a+1}{2}}\prod^{l}_{j=1}\left\{\varrho[\varrho(\partial_{\varrho\varrho}+T^{2})+ \Delta_{b}-(n-1)\partial_{\varrho}]+\frac{n^{2}}{4}-\frac{(a-l+2j-3)^{2}}{4}\right\}f.
\end{split}
 \end{equation*}

Similarly, if $l$ is odd, then by (\ref{9.9}) and Lemma \ref{lm9.2} we get that (\ref{9.11}) is also valid for $l+1$. These complete the proof of Proposition \ref{pr8.4}.
\end{proof}

Next we consider the ball model.
\begin{lemma}\label{lm8.5}
Let $s\in\mathbb{R}$ and $u\in C^{\infty}(\mathbb{B}_{\mathbb{C}}^{n})$. There holds
 \begin{equation*}%\label{2.1}
  \begin{split}
\Delta'_{s-n,s-n}\left[(1-|z|^{2})^{s-n}u\right]
=&4^{-1}(1-|z|^{2})^{s-n-1}\left[\Delta_{\mathbb{B}}+4s(n-s)\right]u.
\end{split}
\end{equation*}
\end{lemma}
\begin{proof}
A simple calculation shows
 \begin{equation*}%\label{2.1}
  \begin{split}
 R\left[(1-|z|^{2})^{n-s}u\right]=&-(n-s)(1-|z|^{2})^{n-1-s}|z|^{2}u+(1-|z|^{2})^{n-s}Ru;\\
\overline{R}\left[(1-|z|^{2})^{n-s}u\right]=&-(n-s)(1-|z|^{2})^{n-1-s}|z|^{2}u+(1-|z|^{2})^{n-s}\overline{R}u;\\
R\overline{R}\left[(1-|z|^{2})^{n-s}u\right]=&-(n-s)\left[-(n-1-s)(1-|z|^{2})^{n-s-2}|z|^{4}+
(1-|z|^{2})^{n-1-s}|z|^{2}\right]u\\
&-(n-s)(1-|z|^{2})^{n-1-s}|z|^{2}Ru+(1-|z|^{2})^{n-s}R\overline{R}u\\
&-(n-s)(1-|z|^{2})^{n-1-s}|z|^{2}\overline{R}u\\
=&(1-|z|^{2})^{n-s}R\overline{R}u-(n-s)(1-|z|^{2})^{n-1-s}|z|^{2}(\overline{R}+R)u\\
&-(n-s)\left[-(n-1-s)(1-|z|^{2})^{n-s-2}|z|^{4}+
(1-|z|^{2})^{n-1-s}|z|^{2}\right]u.
\end{split}
\end{equation*}
Therefore,
  \begin{equation*}%\label{2.1}
  \begin{split}
&\left\{\frac{1-|z|^{2}}{|z|^{2}}R\overline{R}-\frac{1}{|z|^{2}}\mathcal{L}'_{0}+\frac{n-1}{2}
\cdot\frac{1}{|z|^{2}}(R+\overline{R})-(n-s)(\overline{R}+R)-(n-s)^{2}\right\}\left[(1-|z|^{2})^{s-n}u\right]\\
=&(1-|z|^{2})^{s-n-1}\left[\frac{(1-|z|^{2})^{2}}{|z|^{2}}R\overline{R}-\frac{1-|z|^{2}}{|z|^{2}}\mathcal{L}'_{0}+\frac{n-1}{2}
\cdot\frac{1-|z|^{2}}{|z|^{2}}(R+\overline{R})+s(n-s)
\right]u.\\
\end{split}
\end{equation*}
That is
 \begin{equation*}%\label{2.1}
  \begin{split}
\Delta'_{s-n,s-n}\left[(1-|z|^{2})^{s-n}u\right]
=&4^{-1}(1-|z|^{2})^{s-n-1}\left[\Delta_{\mathbb{B}}+4s(n-s)\right]u.
\end{split}
\end{equation*}
This prove Lemma \ref{lm8.5}.
\end{proof}
\begin{lemma}\label{lm8.6}
Let $a\in\mathbb{R}$ and $l\in\mathbb{R}$. There holds, for $u\in C^{\infty}(\mathbb{B}_{\mathbb{C}}^{n})$,
 \begin{equation*}
\begin{split}
&\Delta'_{\frac{1-a-n-l}{2},\frac{1-a-n-l}{2}}\left[\left(\Delta'_{\frac{2-a-n}{2},\frac{2-a-n}{2}}+\frac{(l-1)^{2}}{4}\right)^{2}-\frac{(l-1)^{2}}{4}(R-\bar{R})^{2}
\right] \\
=&\left[\left(\Delta'_{\frac{1-a-n}{2},\frac{1-a-n}{2}}+\frac{l^{2}}{4}\right)^{2}-\frac{l^{2}}{4}(R-\bar{R})^{2}
\right]\Delta'_{\frac{3-a-n-l}{2},\frac{3-a-n-l}{2}}.
\end{split}
 \end{equation*}
 \end{lemma}
\begin{proof} since
 \begin{equation}\label{8.16}
\begin{split}
\Delta'_{\frac{2-a-n}{2},\frac{2-a-n}{2}}+\frac{(l-1)^{2}}{4}=&\Delta'_{\frac{1-a-n}{2},\frac{1-a-n}{2}}
+\frac{l^{2}+2n+2a-2l-2}{4}+\frac{1}{2}
(R+\bar{R});\\
\Delta'_{\frac{1-a-n-l}{2},\frac{1-a-n-l}{2}}=&\Delta'_{\frac{1-a-n}{2},\frac{1-a-n}{2}}-\frac{l}{2}(R+\bar{R})-\frac{l^{2}-2(1-a-n)l}{4},
\end{split}
 \end{equation}
 we have
 \begin{equation*}
\begin{split}
&\left(\Delta'_{\frac{1-a-n-l}{2},\frac{1-a-n-l}{2}}\right)\left(\Delta'_{\frac{2-a-n}{2},\frac{2-a-n}{2}}+\frac{(l-1)^{2}}{4})\right)\\
=&\left(\Delta'_{\frac{1-a-n}{2},\frac{1-a-n}{2}}+\frac{l^{2}}{4}-\frac{l}{2}(R+\bar{R})-\frac{l^{2}-(1-a-n)l}{2}\right)\\
&\left(\Delta'_{\frac{1-a-n}{2},\frac{1-a-n}{2}}+\frac{l^{2}}{4}
+\frac{1}{2}
(R+\bar{R})+\frac{n+a-l-1}{2}\right)\\
=&\left(\Delta'_{\frac{1-a-n}{2},\frac{1-a-n}{2}}+\frac{l^{2}}{4}\right)^{2}-
\frac{l}{2}(R+\bar{R})\left(\Delta'_{\frac{1-a-n}{2},\frac{1-a-n}{2}}+\frac{l^{2}}{4}\right)+\frac{1}{2}
\left(\Delta'_{\frac{1-a-n}{2},\frac{1-a-n}{2}}+\frac{l^{2}}{4}\right)(R+\bar{R})\\
&+\frac{-l^{2}-(a+n)l+n+a-1}{2}\left(\Delta'_{\frac{1-a-n}{2},\frac{1-a-n}{2}}+\frac{l^{2}}{4}\right)\\
&-\frac{l}{4}(R+\bar{R})^{2}-\frac{l(n+a-1)}{2}(R+\bar{R})-\frac{l}{4}\left[(n+a-1)^{2}-l^{2}\right].
\end{split}
 \end{equation*}
Using the identity
  \begin{equation*}%\label{2.1}
  \begin{split}
\Delta'_{\alpha,\alpha}\Delta'_{\beta,\beta}-\Delta'_{\beta,\beta}\Delta'_{\alpha,\alpha}
=(\alpha-\beta)\left[(R+\bar{R})\Delta'_{\beta,\beta}-\Delta'_{\beta,\beta}(R+\bar{R})\right]=-2(\alpha-\beta)(\Delta'_{0,0}+R\bar{R}),
\end{split}
\end{equation*}
we obtain
\begin{equation*}
\begin{split}
&\left(\Delta'_{\frac{1-a-n-l}{2},\frac{1-a-n-l}{2}}\right)\left(\Delta'_{\frac{2-a-n}{2},\frac{2-a-n}{2}}+\frac{(l-1)^{2}}{4})\right)\\
=&\left(\Delta'_{\frac{1-a-n}{2},\frac{1-a-n}{2}}+\frac{l^{2}}{4}\right)^{2}-\frac{l-1}{2}\left(\Delta'_{\frac{1-a-n}{2},\frac{1-a-n}{2}}+\frac{l^{2}}{4}\right)
(R+\bar{R})
+l\Delta'_{0,0}+lR\bar{R}\\
&+\frac{-l^{2}-(a+n)l+n+a-1}{2}\left(\Delta'_{\frac{1-a-n}{2},\frac{1-a-n}{2}}+\frac{l^{2}}{4}\right)\\
&-\frac{l}{4}(R+\bar{R})^{2}-\frac{l(n+a-1)}{2}(R+\bar{R})-\frac{l}{4}\left[(n+a-1)^{2}-l^{2}\right]\\
=&\left(\Delta'_{\frac{1-a-n}{2},\frac{1-a-n}{2}}+\frac{l^{2}}{4}\right)^{2}-\frac{l-1}{2}\left(\Delta'_{\frac{1-a-n}{2},\frac{1-a-n}{2}}+\frac{l^{2}}{4}\right)
(R+\bar{R})
+l\Delta'_{0,0}\\
&+\frac{-l^{2}-(a+n)l+n+a-1}{2}\left(\Delta'_{\frac{1-a-n}{2},\frac{1-a-n}{2}}+\frac{l^{2}}{4}\right)\\
&-\frac{l}{4}(R-\bar{R})^{2}-\frac{l(n+a-1)}{2}(R+\bar{R})-\frac{l}{4}\left[(n+a-1)^{2}-l^{2}\right]\\
=&\left(\Delta'_{\frac{1-a-n}{2},\frac{1-a-n}{2}}+\frac{l^{2}}{4}\right)^{2}-\frac{l-1}{2}\left(\Delta'_{\frac{1-a-n}{2},\frac{1-a-n}{2}}+\frac{l^{2}}{4}\right)
(R+\bar{R}+a+n+l-1)-\frac{l}{4}(R-\bar{R})^{2}.
\end{split}
 \end{equation*}
Therefore,
 \begin{equation*}
\begin{split}
&\Delta'_{\frac{1-a-n-l}{2},\frac{1-a-n-l}{2}}\left(\Delta'_{\frac{2-a-n}{2},\frac{2-a-n}{2}}+\frac{(l-1)^{2}}{4}\right)^{2}-
\left(\Delta'_{\frac{1-a-n}{2},\frac{1-a-n}{2}}+\frac{l^{2}}{4}\right)^{2}\Delta'_{\frac{3-a-n-l}{2},\frac{3-a-n-l}{2}}\\
=&\left(\Delta'_{\frac{1-a-n}{2},\frac{1-a-n}{2}}+\frac{l^{2}}{4}\right)^{2}\left(
\Delta'_{\frac{2-a-n}{2},\frac{2-a-n}{2}}+\frac{(l-1)^{2}}{4}-\Delta'_{\frac{3-a-n-l}{2},\frac{3-a-n-l}{2}}\right)\\
&-\frac{l-1}{2}\left(\Delta'_{\frac{1-a-n}{2},\frac{1-a-n}{2}}+\frac{l^{2}}{4}\right)
(R+\bar{R}+a+n+l-1)\left(\Delta'_{\frac{2-a-n}{2},\frac{2-a-n}{2}}+\frac{(l-1)^{2}}{4}\right)\\
&-\frac{l}{4}(R-\bar{R})^{2}\left(\Delta'_{\frac{2-a-n}{2},\frac{2-a-n}{2}}+\frac{(l-1)^{2}}{4}\right)\\
=&\frac{l-1}{2}\left(\Delta'_{\frac{1-a-n}{2},\frac{1-a-n}{2}}+\frac{l^{2}}{4}\right)^{2}\left(R+\bar{R}+a+n+l-3\right)\\
&-\frac{l-1}{2}\left(\Delta'_{\frac{1-a-n}{2},\frac{1-a-n}{2}}+\frac{l^{2}}{4}\right)
\left(\Delta'_{\frac{2-a-n}{2},\frac{2-a-n}{2}}+\frac{(l-1)^{2}}{4}\right)(R+\bar{R}+a+n+l-1)\\
&+(l-1)\left(\Delta'_{\frac{1-a-n}{2},\frac{1-a-n}{2}}+\frac{l^{2}}{4}\right)\left(\Delta'_{0,0}+R\bar{R}\right)\\
&-\frac{l}{4}(R-\bar{R})^{2}\left(\Delta'_{\frac{2-a-n}{2},\frac{2-a-n}{2}}+\frac{(l-1)^{2}}{4}\right)\\
=&\frac{l-1}{2}\left(\Delta'_{\frac{1-a-n}{2},\frac{1-a-n}{2}}+\frac{l^{2}}{4}\right)\left(-\frac{R+\bar{R}}{2}+\frac{l+1-a-n}{2}\right)\left(R+\bar{R}
+a+n+l-1\right)\\
&-(l-1)\left(\Delta'_{\frac{1-a-n}{2},\frac{1-a-n}{2}}+\frac{l^{2}}{4}\right)^{2}
+(l-1)\left(\Delta'_{\frac{1-a-n}{2},\frac{1-a-n}{2}}+\frac{l^{2}}{4}\right)\left(\Delta'_{0,0}+R\bar{R}\right)\\
&-\frac{l}{4}(R-\bar{R})^{2}\left(\Delta'_{\frac{2-a-n}{2},\frac{2-a-n}{2}}+\frac{(l-1)^{2}}{4}\right)\\
=&-\frac{l-1}{4}(R-\bar{R})^{2}\left(\Delta'_{\frac{1-a-n}{2},\frac{1-a-n}{2}}+\frac{l^{2}}{4}\right)
--\frac{l}{4}(R-\bar{R})^{2}\left(\Delta'_{\frac{2-a-n}{2},\frac{2-a-n}{2}}+\frac{(l-1)^{2}}{4}\right)\\
=&\Delta'_{\frac{1-a-n-l}{2},\frac{1-a-n-l}{2}}\frac{(l-1)^{2}}{4}(R-\bar{R})^{2}
-\frac{l^{2}}{4}(R-\bar{R})^{2}
\Delta'_{\frac{3-a-n-l}{2},\frac{3-a-n-l}{2}}\\
=&\frac{(R-\bar{R})^{2}}{4}\left[(l-1)^{2}\Delta'_{\frac{1-a-n-l}{2},\frac{1-a-n-l}{2}}-l^{2}\Delta'_{\frac{3-a-n-l}{2},\frac{3-a-n-l}{2}}\right].
\end{split}
 \end{equation*}
This proves  Lemma \ref{lm8.6}
\end{proof}

\begin{proposition}\label{pr8.6}
Let $a\in\mathbb{R}$ and $u\in C^{\infty}(\mathbb{B}_{\mathbb{C}}^{n})$. We have,  for $k\in\mathbb{N}\setminus\{0\}$,
 \begin{equation}\label{8.13}
\begin{split}
&\prod^{k}_{j=1}\left[\Delta'_{\frac{1-a-n}{2},\frac{1-a-n}{2}}+\frac{(k+1-2j)^{2}}{4}-\frac{k+1-2j}{2}(R-\bar{R})\right][(1-|z|^{2})^{\frac{k-n-a}{2}} f] \\
=&4^{-k}(1-|z|^{2})^{-\frac{k+n+a}{2}}\prod^{k}_{j=1}\left[\Delta_{\mathbb{B}}+n^{2}-(a-k+2j-2)^{2}\right]f.
\end{split}
 \end{equation}
\end{proposition}
\begin{proof}
By Letting $s=\frac{n+1-a}{2}$ in Lemma \ref{lm8.5}, we obtain that (\ref{8.13}) is valid for $k=1$.
Assume it is valid for $k=l$,
 \begin{equation}\label{8.14}
\begin{split}
&\prod^{l}_{j=1}\left[\Delta'_{\frac{1-a-n}{2},\frac{1-a-n}{2}}+\frac{(l+1-2j)^{2}}{4}-\frac{l+1-2j}{2}(R-\bar{R})\right][(1-|z|^{2})^{\frac{l-n-a}{2}} f] \\
=&4^{-l}(1-|z|^{2})^{-\frac{l+n+a}{2}}\prod^{l}_{j=1}\left[\Delta_{\mathbb{B}}+n^{2}-(a-l+2j-2)^{2}\right]f.
\end{split}
 \end{equation}
Replacing $a$ by $a-1$ in (\ref{8.14}), we get
 \begin{equation*}
\begin{split}
&\prod^{l}_{j=1}\left[\Delta'_{\frac{2-a-n}{2},\frac{2-a-n}{2}}+\frac{(l+1-2j)^{2}}{4}-\frac{l+1-2j}{2}(R-\bar{R})\right][(1-|z|^{2})^{\frac{l+1-n-a}{2}} f] \\
=&4^{-l}(1-|z|^{2})^{-\frac{l+n+a-1}{2}}\prod^{l}_{j=1}\left[\Delta_{\mathbb{B}}+n^{2}-(a-1-l+2j-2)^{2}\right]f.
\end{split}
 \end{equation*}
Therefore, by Lemma \ref{lm8.5}, we get
 \begin{equation*}%\label{8.15}
\begin{split}
&\Delta'_{\frac{1-a-n-l}{2},\frac{1-a-n-l}{2}}\prod^{l}_{j=1}\left[\Delta'_{\frac{2-a-n}{2},\frac{2-a-n}{2}}+\frac{(l+1-2j)^{2}}{4}-\frac{l+1-2j}{2}(R-\bar{R})
\right][(1-|z|^{2})^{\frac{l+1-n-a}{2}} f] \\
=&\Delta'_{\frac{1-a-n-l}{2},\frac{1-a-n-l}{2}}\left[4^{-l}(1-|z|^{2})^{-\frac{l+n+a-1}{2}}\prod^{l}_{j=1}\left[\Delta_{\mathbb{B}}+n^{2}-(a-1-l+2j-2)^{2}\right]f
\right]\\
=&4^{-l-1}(1-|z|^{2})^{-\frac{l+n+a+1}{2}}\prod^{l+1}_{j=1}\left[\Delta_{\mathbb{B}}+n^{2}-(a-1-l+2j-2)^{2}\right]f.
\end{split}
 \end{equation*}
On the other hand, if $l$ is even, then by Lemma \ref{lm8.6}, we have
 \begin{equation*}%\label{8.15}
\begin{split}
&\Delta'_{\frac{1-a-n-l}{2},\frac{1-a-n-l}{2}}\prod^{l}_{j=1}\left[\Delta'_{\frac{2-a-n}{2},\frac{2-a-n}{2}}+\frac{(l+1-2j)^{2}}{4}-\frac{l+1-2j}{2}(R-\bar{R})
\right]\\
=&\Delta'_{\frac{1-a-n-l}{2},\frac{1-a-n-l}{2}}\prod^{l/2}_{j=1}\left[\left(\Delta'_{\frac{2-a-n}{2},\frac{2-a-n}{2}}+\frac{(l+1-2j)^{2}}{4}\right)^{2}
-\frac{(l+1-2j)^{2}}{4}(R-\bar{R})^{2}
\right]\\
=&\left[\left(\Delta'_{\frac{1-a-n}{2},\frac{1-a-n}{2}}+\frac{l^{2}}{4}\right)^{2}-\frac{l^{2}}{4}(R-\bar{R})^{2}
\right]\Delta'_{\frac{3-a-n-l}{2},\frac{3-a-n-l}{2}}\\
&\prod^{l/2}_{j=2}\left[\left(\Delta'_{\frac{2-a-n}{2},\frac{2-a-n}{2}}+\frac{(l+1-2j)^{2}}{4}\right)^{2}
-\frac{(l+1-2j)^{2}}{4}(R-\bar{R})^{2}
\right].
\end{split}
 \end{equation*}
Repeating this process in a suitable manner we get
\begin{equation*}%\label{8.15}
\begin{split}
&\Delta'_{\frac{1-a-n-l}{2},\frac{1-a-n-l}{2}}\prod^{l}_{j=1}\left[\Delta'_{\frac{2-a-n}{2},\frac{2-a-n}{2}}+\frac{(l+1-2j)^{2}}{4}-\frac{l+1-2j}{2}(R-\bar{R})
\right]\\
=&\prod^{l+1}_{j=1}\left[\Delta'_{\frac{1-a-n}{2},\frac{1-a-n}{2}}+\frac{(l+2-2j)^{2}}{4}-\frac{l+2-2j}{2}(R-\bar{R})\right]
\end{split}
 \end{equation*}
Therefore, we have proved that if $l$ is even, then
 \begin{equation}\label{8.16}
\begin{split}
&\prod^{l+1}_{j=1}\left[\Delta'_{\frac{1-a-n}{2},\frac{1-a-n}{2}}+\frac{(l+2-2j)^{2}}{4}-\frac{l+2-2j}{2}(R-\bar{R})\right][(1-|z|^{2})^{\frac{l+1-n-a}{2}} f] \\
=&4^{-l-1}(1-|z|^{2})^{-\frac{l+1+n+a}{2}}\prod^{l+1}_{j=1}\left[\Delta_{\mathbb{B}}+n^{2}-(a-l+2j-3)^{2}\right]f.
\end{split}
 \end{equation}
 The proof of   (\ref{8.16}) for $l$ odd  is similar and  we omit it.  These complete the proof of Proposition \ref{pr8.6}.
\end{proof}

\textbf{Proof of Theorem \ref{th1.6}} Combining Proposition \ref{pr8.4} and \ref{pr8.6}.

\section{asymptotic Estimates of Bessel-Green-Riesz kernels }
In what follows, $a \lesssim b$ or $a=O(b)$ will stand for $a\leq Cb$ with a positive constant $C$ and $a \thicksim b$ stand for $C^{-1}b\leq a\leq Cb$.
We shall frequently use the following
\begin{equation}\label{3.1}
\begin{split}
\cosh r-1\thicksim r^{2},\;\;\sinh r-r\thicksim r^{3},\;\;\;\;0<r<1,\;\; \textrm{and}\;\;
\sinh r\thicksim\cosh r\thicksim e^{r},\;\;r\geq1.
\end{split}
\end{equation}
We also need the following identity of convolution   on Euclidean space (see \cite{s})
\begin{equation}\label{3.2}
\begin{split}
\int_{\mathbb{R}^{n}}|x|^{\alpha-n}|y-x|^{\beta-n}dx=
\frac{\gamma_{n}(\alpha)\gamma_{n}(\beta)}{\gamma_{n}(\alpha+\beta)}|y|^{\alpha+\beta-n},\;\;0<\alpha<n, \;0<\beta<n, \;0<\alpha+\beta<n.
\end{split}
\end{equation}
where
\begin{equation}\label{3.3}
\begin{split}
\gamma_{n}(\alpha)=\pi^{n/2}2^{\alpha}\Gamma(\alpha/2)/\Gamma\left(\frac{n}{2}-\frac{\alpha}{2}\right),\;\; 0<\alpha<n.
\end{split}
\end{equation}

\subsection{Heat kernel and Bessel-Green-Riesz kernels on real hyperbolic space}
In this subsection we   recall some fact about asymptotic estimates of Bessel-Green-Riesz kernels on real hyperbolic space.
Denote by $h_{t}(\rho;n)$ the heat kernel on real hyperbolic space of dimension $n$ which is
 given explicitly by the following formulas (see e.g. \cite{d,gn}):
\begin{itemize}
  \item If $n=2m+1$, then
 \begin{equation}\label{3.4}
h_{t}(\rho;2m+1)=2^{-m-1}\pi^{-m-1/2}t^{-\frac{1}{2}}e^{-m^{2}t}
  \left(-\frac{1}{\sinh \rho}\frac{\partial}{\partial \rho}\right)^{m}e^{-\frac{\rho^{2}}{4t}}.
 \end{equation}
  \item If $n=2m$, then
  \begin{equation}\label{3.5}
h_{t}(\rho;2m)=(2\pi)^{-m-\frac{1}{2}}t^{-\frac{1}{2}}e^{-\frac{(2m-1)^{2}}{4}t}
\int^{+\infty}_{\rho}\frac{\sinh r}{\sqrt{\cosh r-\cosh\rho}}\left(-\frac{1}{\sinh r}\frac{\partial}{\partial r}\right)^{m}e^{-\frac{r^{2}}{4t}}dr.
  \end{equation}
\end{itemize}
The asymptotic estimates of Bessel-Green-Riesz kernels on
real hyperbolic space satisfies (see e.g. \cite{LiLuy2}):
\begin{itemize}
  \item $\zeta>0$:
\begin{equation}\label{3.6}
\begin{split}
(-\Delta_{\mathbb{H}}-(n-1)^{2}/4+\zeta^{2})^{-\alpha/2}\leq&\frac{1}{\gamma_{n}(\alpha)}\cdot\frac{1}
{\rho^{n-\alpha}}+O\left(\frac{1}{\rho^{n-\alpha-\epsilon}}\right),   \;0<\alpha<n,\;\;0<\rho<1\\
(-\Delta_{\mathbb{H}}-(n-1)^{2}/4+\zeta^{2})^{-\alpha/2} \thicksim& \rho^{\frac{\alpha-2}{2}}e^{-\zeta\rho-\frac{n-1}{2}\rho},
\;\;\;\;\;\;\;\;\;\;\;\;\;\;\;\;\;\;\;\;\;\;\;\alpha>0,
\;\;\rho\geq1;
\end{split}
\end{equation}
  \item $\zeta=0$:
\begin{equation}\label{3.7}
\begin{split}
(-\Delta_{\mathbb{H}}-(n-1)^{2}/4)^{-\alpha/2}\leq&\frac{1}{\gamma_{n}(\alpha)}\cdot\frac{1}
{\rho^{n-\alpha}}+O\left(\frac{1}{\rho^{n-\alpha-2}}\right),   \;\;\;0<\alpha<3,\;\;0<\rho<1\\
(-\Delta_{\mathbb{H}}-(n-1)^{2}/4)^{-\alpha/2}\thicksim& \rho^{\alpha-2}e^{-\frac{n-1}{2}\rho},
\;\;\;\;\;\;\;\;\;\;\;\;\;\;\;\;\;\;\;\;\;\;\;\;\;\;\;\;\;0<\alpha<3,
\;\;\rho\geq1;
\end{split}
\end{equation}

\end{itemize}
where $0<\zeta'<\zeta$ and   $0<\epsilon<\min \{1,n-\alpha\}$.

In particular, in case $\alpha=2$, we have the following explicit formula (see  \cite{mat} for $\nu\geq\frac{n-1}{2}$, \cite{li} for $\nu>0$
and \cite{ly4} for $\nu=0$)
\begin{equation}\label{3.8}
\begin{split}
&(\nu^{2}-(n-1)^{2}/4-\Delta_{\mathbb{H}})^{-1}\\
=&\frac{(2\pi)^{-\frac{n}{2}}\Gamma(\frac{n-1}{2}+\nu)}{2^{\nu+\frac{1}{2}}
\Gamma(\nu+\frac{1}{2})}(\sinh\rho)^{2-n}\int^{\pi}_{0}(\cosh\rho+\cos t)^{\frac{n-3}{2}-\nu}
(\sin t)^{2\nu}dt.
\end{split}
\end{equation}

\subsection{Heat kernel and Bessel-Green-Riesz kernels on complex hyperbolic space}
Denote by $e^{t\Delta_{\mathbb{B}}}$ the heat kernel on $\mathbb{B}_{\mathbb{C}}^{n}$. It is known that $e^{t\Delta_{\mathbb{B}}}$  is given explicitly by the following formulas (see  \cite{lo})
\begin{equation}\label{3.9}
  e^{t\Delta_{\mathbb{B}}}=2^{-n+\frac{1}{2}}\pi^{-n-\frac{1}{2}}\frac{e^{-n^{2}t}}{t^{\frac{1}{2}}}\int^{+\infty}_{\rho}
\frac{\sinh r}{\sqrt{\cosh 2r-\cosh2\rho}}\left(-\frac{1}{\sinh r}\frac{\partial}{\partial r}\right)^{n}
e^{-\frac{r^{2}}{4t}}dr.
\end{equation}
For simplicity, we denote by
\begin{equation*}%\label{3.4}
\begin{split}
k_{\zeta,\alpha}=&(-\Delta_{\mathbb{B}}-n^{2}+\zeta^{2})^{-\alpha/2},\;\;0<\alpha<2n, \;\zeta>0;\\
k_{\alpha}=&(-\Delta_{\mathbb{B}}-n^{2})^{-\alpha/2},\;\;\;\;\;\;\;\;\;\;0<\alpha<3.
\end{split}
\end{equation*}
The following asymptotic estimates of  Bessel-Green-Riesz kernels for large $\rho$ is given by Anker and Ji (\cite{anj}, page 1083,  (iii)):
\begin{itemize}
  \item $\zeta>0$:
\begin{equation}\label{3.10}
k_{\zeta,\alpha}\thicksim \rho^{\frac{\alpha-2}{2}}e^{-\zeta\rho-n\rho},\;\;\;\alpha>0,\;\;\rho\geq1
\end{equation}
  \item $\zeta=0$:
\begin{equation}\label{3.11}
k_{\alpha}\thicksim \rho^{\alpha-2}e^{-n\rho},\;\;0<\alpha<3,\;\;\rho\geq1;
\end{equation}
\end{itemize}
 In order to obtain the sharp constant of fractional Hardy-Adams type inequalities, we need to give
 more refined and optimal estimates of such kernels for both small and large $\rho$.

In particular, by (\ref{3.4}) and (\ref{3.1}), we have
\begin{equation}\label{3.12}
  e^{t\Delta_{\mathbb{B}}}=2^{\frac{3}{2}}\int^{+\infty}_{\rho}
\frac{\sinh r}{\sqrt{\cosh 2r-\cosh2\rho}}h_{t}(r;2n+1)dr.
\end{equation}
Combing (\ref{3.8}) and (\ref{3.12}) yields the following  explicit formula of Green's function:
\begin{equation*}
\begin{split}
&(\nu^{2}-n^{2}-\Delta_{\mathbb{B}})^{-1}\\
=&\frac{(2\pi)^{-\frac{2n+1}{2}}\Gamma(n+\nu)}{2^{\nu-1}
\Gamma(\nu+\frac{1}{2})}\int^{+\infty}_{\rho}
\frac{(\sinh r)^{2-2n}}{\sqrt{\cosh 2r-\cosh2\rho}} \left(\int^{\pi}_{0}(\cosh r+\cos t)^{n-1-\nu}
(\sin t)^{2\nu}dt\right)dr.
\end{split}
\end{equation*}

\subsection{Estimate of $k_{\alpha}$, $0<\alpha<3$}
We firstly prove the following lemma:
\begin{lemma}\label{lm3.1}
Let $\beta>0$. Then for $\rho>0$, we have
\begin{equation}\label{lm}
  \int^{+\infty}_{\rho}\frac{\cosh r}{(\sinh r)^{\beta}\sqrt{\cosh 2r-\cosh2\rho}}dr=\frac{\Gamma(\frac{1}{2})\Gamma(\frac{\beta}{2})}{2\sqrt{2}\Gamma(\frac{1+\beta}{2})(\sinh\rho)^{\beta}}.
\end{equation}
\end{lemma}
\begin{proof}
Substituting $t=\cosh 2r-\cosh2\rho$, we have
\begin{equation*}
  \begin{split}
\int^{+\infty}_{\rho}\frac{\cosh r}{(\sinh r)^{\beta}\sqrt{\cosh 2r-\cosh2\rho}}dr
=&\frac{1}{4}\int^{+\infty}_{0}\frac{1}{\sqrt{t}}\cdot\frac{1}{(t/2+\sinh^{2}\rho)^{\frac{1+\beta}{2}}}dt\\
=&\frac{1}{2\sqrt{2}(\sinh\rho)^{\beta}}\int^{+\infty}_{0}s^{-\frac{1}{2}}(1+s)^{-\frac{1+\beta}{2}}ds.
  \end{split}
\end{equation*}
To get last equation, we use the substitution $t=2s\sinh^{2}\rho$. Therefore, by using the fact
\begin{equation*}%\label{lm}
\int^{+\infty}_{0}s^{x-1}(1+s)^{-x-y}ds=\frac{\Gamma(x)\Gamma(y)}{\Gamma(x+y)},\; x>0,\;y>0,
\end{equation*}
we have
\begin{equation*}
  \begin{split}
&\int^{+\infty}_{\rho}\frac{\cosh r}{(\sinh r)^{\beta}\sqrt{\cosh 2r-\cosh2\rho}}dr=\frac{\Gamma(\frac{1}{2})\Gamma(\frac{\beta}{2})}{2\sqrt{2}\Gamma(\frac{1+\beta}{2})(\sinh\rho)^{\beta}}.
  \end{split}
\end{equation*}
This completes the proof of Lemma \ref{lm3.1}.

\end{proof}

Now we give the asymptotic estimates of $k_{\alpha}$ as $\rho\rightarrow0$. The main result is the following:
\begin{lemma}
Let $0<\alpha<3$ and $n\geq2$. There holds
\begin{equation}\label{3.13}
\begin{split}
k_{\alpha}\leq&\frac{1}
{\gamma_{2n}(\alpha)\rho^{2n-\alpha}}+O\left(\frac{1}{\rho^{2n-\alpha-1}}\right),\;\;0<\rho<1.
\end{split}
\end{equation}

\end{lemma}

\begin{proof}

By the Mellin type expression and (\ref{3.12}), we have, for $\rho\in (0,1)$,
\begin{equation*}%\label{3.15}
\begin{split}
k_{\alpha}(\rho)=&\frac{1}{\Gamma(\alpha/2)}\int^{\infty}_{0}t^{\frac{\alpha}{2}-1}e^{t(\Delta_{\mathbb{B}}+n^{2})}dt\\
=&2^{\frac{3}{2}}\int^{+\infty}_{\rho}
\frac{\sinh r}{\sqrt{\cosh 2r-\cosh2\rho}}\left(\frac{1}{\Gamma(\alpha/2)}\int^{\infty}_{0}t^{\frac{\alpha}{2}-1}e^{n^{2}t}h_{t}(r;2n+1)dt\right)dr\\
=:&A_{1}+A_{2},
\end{split}
\end{equation*}
where
\begin{equation*}%\label{3.13}
\begin{split}
A_{1}=&2^{\frac{3}{2}}\int^{2}_{\rho}
\frac{\sinh r}{\sqrt{\cosh 2r-\cosh2\rho}}\left(\frac{1}{\Gamma(\alpha/2)}\int^{\infty}_{0}t^{\frac{\alpha}{2}-1}h_{t}(r;2n+1)dt\right)dr;\\
A_{2}=&2^{\frac{3}{2}}\int^{+\infty}_{2}
\frac{\sinh r}{\sqrt{\cosh 2r-\cosh2\rho}}\left(\frac{1}{\Gamma(\alpha/2)}\int^{\infty}_{0}t^{\frac{\alpha}{2}-1}h_{t}(r;2n+1)dt\right)dr.
\end{split}
\end{equation*}
By (\ref{3.7}), we have
\begin{equation*}%\label{3.13}
\begin{split}
A_{1}\leq&2^{\frac{3}{2}}\int^{2}_{\rho}
\frac{\sinh r}{\sqrt{\cosh 2r-\cosh2\rho}}\left(\frac{1}{\gamma_{2n+1}(\alpha)}\cdot\frac{1}
{r^{2n+1-\alpha}}+O\left(\frac{1}{r^{2n-\alpha-1}}\right)\right)dr\\
&\leq2^{\frac{3}{2}}\int^{2}_{\rho}
\frac{\sinh r}{\sqrt{\cosh 2r-\cosh2\rho}}\left(\frac{1}
{\gamma_{2n+1}(\alpha)(\sinh r)^{2n+1-\alpha}}+O\left(\frac{1}{(\sinh r)^{2n-\alpha-1}}\right)\right)dr.
\end{split}
\end{equation*}
By Lemma \ref{lm3.1}, we obtain
\begin{equation*}%\label{3.13}
\begin{split}
&2^{\frac{3}{2}}\int^{2}_{\rho}
\frac{\sinh r}{\sqrt{\cosh 2r-\cosh2\rho}}\cdot\frac{1}
{\gamma_{2n+1}(\alpha)(\sinh r)^{2n+1-\alpha}}dr\\
=&\frac{2^{\frac{3}{2}}}{\gamma_{2n+1}(\alpha)}\int^{2}_{\rho}
\frac{1}{\sqrt{\cosh 2r-\cosh2\rho}}\cdot\frac{1}
{(\sinh r)^{2n-\alpha}}dr\\
\leq&\frac{2^{\frac{3}{2}}}{\gamma_{2n+1}(\alpha)}\int^{+\infty}_{\rho}
\frac{\cosh r}{\sqrt{\cosh 2r-\cosh2\rho}}\cdot\frac{1}
{(\sinh r)^{2n-\alpha}}dr\\
=&\frac{2^{\frac{3}{2}}}{\gamma_{2n+1}(\alpha)}
\cdot\frac{\Gamma(\frac{1}{2})\Gamma(\frac{2n-\alpha}{2})}{2\sqrt{2}\Gamma(\frac{1+2n-\alpha}{2})(\sinh\rho)^{2n-\alpha}}=\frac{1}
{\gamma_{2n}(\alpha)(\sinh\rho)^{2n-\alpha}}.
\end{split}
\end{equation*}
Similarly,
\begin{equation*}%\label{3.13}
\begin{split}
&\int^{2}_{\rho}
\frac{\sinh r}{\sqrt{\cosh 2r-\cosh2\rho}}\cdot\frac{1}{(\sinh r)^{2n-\alpha-1}}dr\\
\leq&\int^{2}_{\rho}
\frac{\cosh  r}{\sqrt{\cosh 2r-\cosh2\rho}}\cdot\frac{1}{(\sinh r)^{2n-\alpha-1}}dr\lesssim\frac{1}{(\sinh \rho)^{2n-\alpha-1}}.
\end{split}
\end{equation*}
Therefore,
\begin{equation*}%\label{3.13}
\begin{split}
A_{1}\leq&\frac{1}
{\gamma_{2n}(\alpha)(\sinh\rho)^{2n-\alpha}}+O\left(\frac{1}{(\sinh \rho)^{2n-\alpha-1}}\right).
\end{split}
\end{equation*}
We claim $A_{2}$ is bounded for $0<\rho<1$. In fact, we have, for $0<\rho<1$,
\begin{equation*}%\label{3.13}
\begin{split}
A_{2}\thicksim&\int^{+\infty}_{2}
\frac{\sinh r}{\sqrt{\cosh 2r-\cosh2\rho}}r^{\alpha-2}e^{-nr}dr\\
\thicksim&\int^{+\infty}_{2}
r^{\alpha-2}e^{-nr}dr\thicksim1.
\end{split}
\end{equation*}
Therefore, we have
\begin{equation*}%\label{3.13}
\begin{split}
k_{\alpha}=A_{1}+A_{2}\leq&\frac{1}
{\gamma_{2n}(\alpha)(\sinh\rho)^{2n-\alpha}}+O\left(\frac{1}{(\sinh \rho)^{2n-\alpha-1}}\right),\;\;0<\rho<1.
\end{split}
\end{equation*}
The desired result follows.
\end{proof}

\subsection{Estimate of $k_{\zeta,\alpha}$}
\begin{lemma}
Let $\alpha>0$ and $0<\varepsilon <\min\{1, 2n-\alpha\}$. There holds
\begin{equation}\label{3.15}
\begin{split}
k_{\alpha}\leq&\frac{1}
{\gamma_{2n}(\alpha)\rho^{2n-\alpha}}+O\left(\frac{1}{\rho^{2n-\alpha-\varepsilon}}\right),\;\;0<\rho<1.
\end{split}
\end{equation}
\end{lemma}
\begin{proof}
Also by the Mellin type expression and (\ref{3.12}), we have, for $\rho\in (0,1)$,
\begin{equation*}%\label{3.12}
\begin{split}
k_{\zeta,\alpha}=&\frac{1}{\Gamma(\alpha/2)}\int^{\infty}_{0}t^{\frac{\alpha}{2}-1}e^{t(\Delta_{\mathbb{B}}+n^{2}-\zeta^{2})}dt\\
=&2^{\frac{3}{2}}\int^{+\infty}_{\rho}
\frac{\sinh r}{\sqrt{\cosh 2r-\cosh2\rho}}\left(\frac{1}{\Gamma(\alpha/2)}\int^{\infty}_{0}t^{\frac{\alpha}{2}-1}e^{n^{2}t-\zeta^{2}t}h_{t}(r;2n+1)dt\right)dr\\
=:&A_{3}+A_{4},
\end{split}
\end{equation*}
where
\begin{equation*}%\label{3.12}
\begin{split}
A_{3}=&2^{\frac{3}{2}}\int^{2}_{\rho}
\frac{\sinh r}{\sqrt{\cosh 2r-\cosh2\rho}}\left(\frac{1}{\Gamma(\alpha/2)}\int^{\infty}_{0}t^{\frac{\alpha}{2}-1}e^{n^{2}t-\zeta^{2}t}h_{t}(r;2n+1)dt\right)dr;\\
A_{4}=&2^{\frac{3}{2}}\int^{+\infty}_{2}
\frac{\sinh r}{\sqrt{\cosh 2r-\cosh2\rho}}\left(\frac{1}{\Gamma(\alpha/2)}\int^{\infty}_{0}t^{\frac{\alpha}{2}-1}e^{n^{2}t-\zeta^{2}t}h_{t}(r;2n+1)dt\right)dr.
\end{split}
\end{equation*}
By (\ref{3.6}), we have
\begin{equation*}%\label{3.12}
\begin{split}
A_{3}\leq&2^{\frac{3}{2}}\int^{2}_{\rho}
\frac{\sinh r}{\sqrt{\cosh 2r-\cosh2\rho}}\left(\frac{1}{\gamma_{2n+1}(\alpha)}\cdot\frac{1}
{r^{2n+1-\alpha}}+O\left(\frac{1}{r^{2n+1-\alpha-\epsilon}}\right)\right)dr\\
\leq&\frac{2^{\frac{3}{2}}}{\gamma_{2n+1}(\alpha)}\int^{2}_{\rho}
\frac{\sinh r}{\sqrt{\cosh 2r-\cosh2\rho}}\left(\frac{1}
{(\sinh r)^{2n+1-\alpha}}+O\left(\frac{1}{(\sinh r)^{2n+1-\alpha-\epsilon}}\right)\right)dr.
\end{split}
\end{equation*}
Since
\begin{equation*}%\label{3.12}
\begin{split}
&\frac{2^{\frac{3}{2}}}{\gamma_{2n+1}(\alpha)}\int^{2}_{\rho}
\frac{\sinh r}{\sqrt{\cosh 2r-\cosh2\rho}}\cdot\frac{1}
{(\sinh r)^{2n+1-\alpha}}dr\\
\leq&\frac{2^{\frac{3}{2}}}{\gamma_{2n+1}(\alpha)}\int^{+\infty}_{\rho}
\frac{\cosh r}{\sqrt{\cosh 2r-\cosh2\rho}}\cdot\frac{1}
{(\sinh r)^{2n-\alpha}}dr\\
=&\frac{2^{\frac{3}{2}}}{\gamma_{2n+1}(\alpha)}
\cdot\frac{\Gamma(\frac{1}{2})\Gamma(\frac{2n-\alpha}{2})}{2\sqrt{2}\Gamma(\frac{1+2n-\alpha}{2})(\sinh\rho)^{2n-\alpha}}=\frac{1}
{\gamma_{2n}(\alpha)(\sinh\rho)^{2n-\alpha}}
\end{split}
\end{equation*}
and
\begin{equation*}%\label{3.12}
\begin{split}
&\int^{2}_{\rho}
\frac{\sinh r}{\sqrt{\cosh 2r-\cosh2\rho}}\cdot\frac{1}
{(\sinh r)^{2n+1-\alpha-\varepsilon}}dr\\
\lesssim&\frac{2^{\frac{3}{2}}}{\gamma_{2n+1}(\alpha)}\int^{+\infty}_{\rho}
\frac{\cosh r}{\sqrt{\cosh 2r-\cosh2\rho}}\cdot\frac{1}
{(\sinh r)^{2n-\alpha-\varepsilon}}dr\\
\thicksim&\frac{1}
{(\sinh\rho)^{2n-\alpha-\varepsilon}},
\end{split}
\end{equation*}
we get
\begin{equation*}%\label{3.12}
\begin{split}
A_{3}\leq&\frac{1}
{\gamma_{2n}(\alpha)(\sinh\rho)^{2n-\alpha}}+O\left(\frac{1}
{(\sinh\rho)^{2n-\alpha-\varepsilon}}\right).
\end{split}
\end{equation*}
On the other hand,
\begin{equation*}%\label{3.12}
\begin{split}
A_{4}\thicksim&\int^{+\infty}_{2}
\frac{\sinh r}{\sqrt{\cosh 2r-\cosh2\rho}}r^{\frac{\alpha-2}{2}}e^{-\zeta r-nr}dr\backsim
\int^{+\infty}_{2}
r^{\frac{\alpha-2}{2}}e^{-\zeta r-nr}dr\thicksim1.
\end{split}
\end{equation*}
Therefore, we have
\[
k_{\zeta,\alpha}=A_{3}+A_{4}\leq\frac{1}
{\gamma_{2n}(\alpha)(\sinh\rho)^{2n-\alpha}}+O\left(\frac{1}
{(\sinh\rho)^{2n-\alpha-\varepsilon}}\right),\;\;0<\rho<1.
\]
The desired result follows.
\end{proof}
\subsection{Estimate of $k_{\alpha}\ast k_{\zeta,\beta}$,   $0<\alpha<3$, $\zeta>0$,  $0<\beta<2n-\alpha$. }
Before we give the proof of the main result in this subsection, we need the following two lemmas.
\begin{lemma}\label{lm3.4}
Let $0<\alpha<2n$, $0<\beta<2n$ and $\lambda_{1}+\lambda_{2}>\alpha+\beta-2n$.
If $0<\alpha+\beta<2n-1$, then for $0<\rho<1$,
\[
\frac{1}{(\sinh\rho)^{2n-\alpha}(\cosh\rho)^{\lambda_{1}}}
\ast\frac{1}{(\sinh\rho)^{2n-\beta}(\cosh\rho)^{\lambda_{2}}}\leq\frac{\gamma_{2n}(\alpha)\gamma_{2n}(\beta)}{\gamma_{2n}(\alpha+\beta)}
\frac{1}{\rho^{n-\alpha-\beta}}+O\left(\frac{1}{\rho^{n-\alpha-\beta-1}}\right).
\]
If $2n-1\leq\alpha+\beta<2n$, then for $0<\rho<1$ and $0<\epsilon<2n-\alpha-\beta$,
\[
\frac{1}{(\sinh\rho)^{2n-\alpha}(\cosh\rho)^{\lambda_{1}}}
\ast\frac{1}{(\sinh\rho)^{2n-\beta}(\cosh\rho)^{\lambda_{2}}}\leq\frac{\gamma_{2n}(\alpha)\gamma_{2n}(\beta)}{\gamma_{2n}(\alpha+\beta)}
\frac{1}{\rho^{n-\alpha-\beta}}+O\left(\frac{1}{\rho^{n-\alpha-\beta-\epsilon}}\right).
\]
\end{lemma}
\begin{proof}
We compute
\begin{equation}\label{3.29}
\begin{split}
&\frac{1}{(\sinh\rho)^{2n-\alpha}(\cosh\rho)^{\lambda_{1}}}
\ast\frac{1}{(\sinh\rho)^{2n-\beta}(\cosh\rho)^{\lambda_{2}}}\\
=&\int_{\mathbb{B}^{n}_{\mathbb{C}}}\left(\frac{\sqrt{1-|z|^{2}}}{|z|}\right)^{2n-\alpha}
(1-|z|^{2})^{\frac{\lambda_{1}}{2}}\left(\frac{(1-|z|^{2})(1-|a|^{2})}{|z-a|^{2}+|(z,a)|^{2}-|z|^{2}|a|^{2}}\right)^{\frac{2n-\beta}{2}}\cdot\\
&\left(\frac{(1-|z|^{2})(1-|a|^{2})}{|1-(z,a)|^{2}}\right)^{\frac{\lambda_{2}}{2}}\frac{dz}{(1-|z|^{2})^{n+1}}\\
=&(1-|a|^{2})^{\frac{2n-\beta+\lambda_{2}}{2}}\int_{\mathbb{B}^{n}_{\mathbb{C}}}\frac{1}{|z|^{2n-\alpha}}\left(\frac{1}{|z-a|^{2}+|(z,a)|^{2}-|z|^{2}|a|^{2}}
\right)^{\frac{2n-\beta}{2}}\cdot\\
&\frac{dz}{|1-(z,a)|^{\lambda_{2}}(1-|z|^{2})^{1+\frac{\alpha+\beta-2n-\lambda_{1}-\lambda_{2}}{2}}}\\
=&\left(\cosh\rho(a)\right)^{-(2n-\beta+\lambda_{2})}\cdot\left(A_{5}+A_{6}\right),
\end{split}
\end{equation}
where
\begin{equation*}
  \begin{split}
A_{5}=&\int_{\{|z|\leq\frac{1}{2}\}}\frac{1}{|z|^{2n-\alpha}}\left(\frac{1}{|z-a|^{2}+|(z,a)|^{2}-|z|^{2}|a|^{2}}
\right)^{\frac{2n-\beta}{2}}
\frac{dz}{|1-(z,a)|^{\lambda_{2}}(1-|z|^{2})^{1+\frac{\alpha+\beta-2n-\lambda_{1}-\lambda_{2}}{2}}};\\
A_{6}=&\int_{\{1>|z|>\frac{1}{2}\}}\frac{1}{|z|^{2n-\alpha}}\left(\frac{1}{|z-a|^{2}+|(z,a)|^{2}-|z|^{2}|a|^{2}}
\right)^{\frac{2n-\beta}{2}}
\frac{dz}{|1-(z,a)|^{\lambda_{2}}(1-|z|^{2})^{1+\frac{\alpha+\beta-2n-\lambda_{1}-\lambda_{2}}{2}}}.
   \end{split}
\end{equation*}
We note if $\rho(a)<1$ and $|z|\leq\frac{1}{2}$, then
\begin{equation*}
  |1-(z,a)|^{\lambda_{2}}(1-|z|^{2})^{1+\frac{\alpha+\beta-2n-\lambda_{1}-\lambda_{2}}{2}}=1+O(|z|).
\end{equation*}
On the other hand,
\begin{equation*}
  \begin{split}
|(z,a)|^{2}-|z|^{2}|a|^{2}=&|(z,z)+(z,a-z)|^{2}-|z|^{2}|a-z+z|^{2}\\
=&|z|^{4}+|(z,a-z)|^{2}+2|z|^{2}\textrm{Re}(z,a-z)-|z|^{2}[|a-z|^{2}+|z|^{2}+2 \textrm{Re}(z,a-z)]\\
=&|z|^{2}|z-a|^{2}\left[\left|\left(\frac{z}{|z|},\frac{z-a}{|z-a|}\right)\right|^{2}-1\right].
  \end{split}
\end{equation*}
Therefore, we have
\begin{equation*}
  \begin{split}
|z-a|^{2}+|(z,a)|^{2}-|z|^{2}|a|^{2}=&|z-a|^{2} \left\{1+|z|^{2}\left[\left|\left(\frac{z}{|z|},\frac{z-a}{|z-a|}\right)\right|^{2}-1\right]\right\}\\
=&|z-a|^{2}\left[1+O(|z|^{2})\right].
  \end{split}
\end{equation*}
Thus,  if $0<\alpha+\beta<2n-1$, then by using (\ref{3.2}), we get
\begin{equation}\label{3.17}
\begin{split}
  A_{5}= & \int_{\{|z|\leq\frac{1}{2}\}}\frac{1}{|z|^{2n-\alpha}}\frac{1}{|z-a|^{2n-\beta}}\left[1+O(|z|)\right]dz\\
\leq&\int_{\mathbb{C}^{n}}\frac{1}{|z|^{2n-\alpha}}\frac{1}{|z-a|^{2n-\beta}}dz+
O\left(\int_{\mathbb{C}^{n}}\frac{1}{|z|^{2n-\alpha-1}}\frac{1}{|z-a|^{2n-\beta}}dz\right)\\
=&\frac{\gamma_{2n}(\alpha)\gamma_{2n}(\beta)}{\gamma_{2n}(\alpha+\beta)}\frac{1}{|a|^{2n-\alpha-\beta}}
+O\left(\frac{1}{|a|^{2n-\alpha-\beta-1}}\right).
\end{split}
\end{equation}
Similarly, if $2n-1<\alpha+\beta<2n$, then
\begin{equation}\label{3.18}
\begin{split}
  A_{5}= & \int_{\{|z|\leq\frac{1}{2}\}}\frac{1}{|z|^{2n-\alpha}}\frac{1}{|z-a|^{2n-\beta}}\left[1+O(|z|^{\varepsilon})\right]dz\\
\leq&\int_{\mathbb{C}^{n}}\frac{1}{|z|^{2n-\alpha}}\frac{1}{|z-a|^{2n-\beta}}dz+
O\left(\int_{\mathbb{C}^{n}}\frac{1}{|z|^{2n-\alpha-\varepsilon}}\frac{1}{|z-a|^{2n-\beta}}dz\right)\\
=&\frac{\gamma_{2n}(\alpha)\gamma_{2n}(\beta)}{\gamma_{2n}(\alpha+\beta)}\frac{1}{|a|^{2n-\alpha-\beta}}
+O\left(\frac{1}{|a|^{2n-\alpha-\beta-\varepsilon}}\right).
\end{split}
\end{equation}
Now we give the estimate of $A_{6}.$  In fact,
\begin{equation}\label{3.19}
  A_{6}\thicksim
\int_{\{1>|z|>\frac{1}{2}\}}
\frac{dz}{(1-|z|^{2})^{1+\frac{\alpha+\beta-2n-\lambda_{1}-\lambda_{2}}{2}}}<\infty.
\end{equation}
The result follows by combining (\ref{3.17})-(\ref{3.19}).

\end{proof}

\begin{lemma}\label{lm3.5}
Let $0<\alpha<2n$, $0<\beta<2n$ and $\lambda_{1}+\lambda_{2}>\alpha+\beta-2n$.
If $\lambda_{2}-\beta<\lambda_{1}-\alpha$, then for $\rho>1$,
\[
\frac{1}{(\sinh\rho)^{2n-\alpha}(\cosh\rho)^{\lambda_{1}}}
\ast\frac{1}{(\sinh\rho)^{2n-\beta}(\cosh\rho)^{\lambda_{2}}}\thicksim e^{-(2n-\beta+\lambda_{2})\rho}.
\]
\end{lemma}
\begin{proof}
By , we have
\begin{equation}\label{3.29}
\begin{split}
&\frac{1}{(\sinh\rho)^{2n-\alpha}(\cosh\rho)^{\lambda_{1}}}
\ast\frac{1}{(\sinh\rho)^{2n-\beta}(\cosh\rho)^{\lambda_{2}}}\\
=&(\cosh\rho)^{-(2n-\beta+\lambda_{2})\rho}\int_{\mathbb{B}^{n}_{\mathbb{C}}}\frac{1}{|z|^{2n-\alpha}}\left(\frac{1}{|z-a|^{2}+|(z,a)|^{2}-|z|^{2}|a|^{2}}
\right)^{\frac{2n-\beta}{2}}\cdot\\
&\frac{dz}{|1-(z,a)|^{\lambda_{2}}(1-|z|^{2})^{1+\frac{\alpha+\beta-2n-\lambda_{1}-\lambda_{2}}{2}}}.
\end{split}
\end{equation}
Set
\[
F(a)=\int_{\mathbb{S}^{2n-1}}\left(\frac{1}{|z-a|^{2}+|(z,a)|^{2}-|z|^{2}|a|^{2}}
\right)^{\frac{2n-\beta}{2}}|1-(z,a)|^{-\lambda_{2}}d\sigma.
\]
It is easy to check $F(a)$ depends only on $|a|$. Furthermore, by Proposition \ref{pr2.3},
\begin{equation}\label{3.29}
\begin{split}
\lim_{|a|\rightarrow 1-}F(a)=&\int_{\mathbb{S}^{2n-1}}
|1-(z,a)|^{-(2n-\beta+\lambda_{2})}d\sigma\\
=&F((2n-\beta+\lambda_{2})/2,(2n-\beta+\lambda_{2})/2;n;|z|^{2}).
\end{split}
\end{equation}
Therefore,
\begin{equation*}%\label{3.29}
\begin{split}
&\lim_{|a|\rightarrow 1-}\int_{\mathbb{B}^{n}_{\mathbb{C}}}\frac{1}{|z|^{2n-\alpha}}\left(\frac{1}{|z-a|^{2}+|(z,a)|^{2}-|z|^{2}|a|^{2}}
\right)^{\frac{2n-\beta}{2}}\frac{dz}{|1-(z,a)|^{\lambda_{2}}(1-|z|^{2})^{1+\frac{\alpha+\beta-2n-\lambda_{1}-\lambda_{2}}{2}}}\\
=&\int^{1}_{0}r^{\alpha-1}(1-r^{2})^{-1-\frac{\alpha+\beta-2n-\lambda_{1}-\lambda_{2}}{2}}F((2n-\beta+\lambda_{2})/2,(2n-\beta+\lambda_{2})/2;n;r^{2})dr\\
=&\frac{1}{2}\int^{1}_{0}t^{\frac{\alpha}{2}-1}(1-t)^{-1-\frac{\alpha+\beta-2n-\lambda_{1}-\lambda_{2}}{2}}
F((2n-\beta+\lambda_{2})/2,(2n-\beta+\lambda_{2})/2;n;t)dr\\
=&\frac{\Gamma(\frac{\alpha}{2})\Gamma(\frac{2n+\lambda_{1}+\lambda_{2}-\alpha-\beta}{2})}
{2\Gamma(\frac{2n+\lambda_{1}+\lambda_{2}-\beta}{2})}\;_{3}F_{2}((2n-\beta+\lambda_{2})/2,(2n-\beta+\lambda_{2})/2,\frac{\alpha}{2};n,
\frac{2n+\lambda_{1}+\lambda_{2}-\alpha-\beta}{2};1),
\end{split}
\end{equation*}
where $\;_{p}F_{q}(a_{1},\cdots,a_{p};b_{1},\cdots,b_{q};z)$ is the generalized hypergeometric series defined by
\[
\;_{p}F_{q}(a_{1},\cdots,a_{p};b_{1},\cdots,b_{q};z)=\sum^{\infty}_{k=0}\frac{(a_{1})_{k}\cdots (a_{p})_{k}}{(b_{1})_{k}\cdots (b_{q})_{(k)}}\frac{z^{k}}{k!}.
\]
To get the last equation, we use the following (see \cite{gr},  Page 813, 7.512(4))
\begin{equation}\label{3.29}
\begin{split}
\int^{1}_{0}x^{\mu-1}(1-x)^{\nu-1}F(a,b;c;x)dx=&\frac{\Gamma(\mu)\Gamma(\nu)}{\Gamma(\mu+\nu)}\;_{3}F_{2}(a,b,\mu;c,\nu;1),\\
&  \textrm{Re}\mu>0,   \textrm{Re} \nu>0, \textrm{Re}  (c+\nu-a-b)>0.
\end{split}
\end{equation}
The desired result follows.
\end{proof}
Now we can give the estimates of $k_{\alpha}\ast k_{\zeta,\beta}$. The main results are the following two lemmas.

\begin{lemma} \label{lm3.6}
Let $n\geq2$, $\zeta>0$, $0<\alpha<3$ and $0<\beta<2n-\alpha$. Then for   $0<\epsilon_{2}<\min\{1,2n-\alpha-\beta,\zeta\}$, we have
\begin{equation}\label{3.24}
k_{\alpha}\ast k_{\zeta,\beta}\leq \frac{1}{\gamma_{2n}(\alpha+\beta)}
\cdot\frac{1}{\rho^{2n-\alpha-\beta}}+O\left(\frac{1}{\rho^{2n-\alpha-\beta-\epsilon_{2}}}\right),\;\;
0<\rho<1.
\end{equation}
and
\begin{equation}\label{3.25}
k_{\alpha}\ast k_{\zeta,\beta}\lesssim  e^{(\varepsilon_{2}-n)\rho},\;\;
0<\rho<1.
\end{equation}
\end{lemma}
\begin{proof}
Recall that
\begin{equation*}%\label{3.19}
\begin{split}
k_{\alpha}\lesssim &\frac{1}
{\gamma_{2n}(\alpha)(\sinh\rho)^{2n-\alpha}}+O\left(\frac{1}{(\sinh \rho)^{2n-\alpha-1}}\right),\;\;0<\rho<1,
\\
k_{\alpha}\thicksim &
\rho^{\alpha-2}e^{-n\rho}\lesssim e^{\epsilon_{1}\rho-n\rho},\;\;0<\alpha<3,\;\;\rho\geq1,
\end{split}
\end{equation*}
where $\epsilon_{1}>0$. Therefore, we have the  following global estimate of $k_{\alpha}$:
for $\rho>0$,
\begin{equation*}%\label{3.20}
\begin{split}
k_{\alpha}\leq
\frac{(\cosh\rho)^{n-\alpha+\varepsilon_{1}}}
{\gamma_{2n}(\alpha)(\sinh\rho)^{2n-\alpha}}
+O\left(\frac{(\cosh\rho)^{n-\alpha+\varepsilon_{1}-1}}
{(\sinh\rho)^{2n-\alpha-1}}\right).
\end{split}
\end{equation*}
Similarly,
for $\rho>0$,
\begin{equation*}%\label{3.21}
\begin{split}
k_{\zeta,\beta}\leq
\frac{(\cosh\rho)^{n-\beta-\zeta'}}
{\gamma_{2n}(\beta)(\sinh\rho)^{2n-\beta}}
+O\left(\frac{(\cosh\rho)^{n-\beta-\zeta'-\varepsilon_{2}}}
{(\sinh\rho)^{2n-\beta-\varepsilon_{2}}}\right)..
\end{split}
\end{equation*}
If we choose $\epsilon_{1}<\zeta'$, then by Lemma \ref{lm3.4},  we have, for $0<\rho<1$,
\[
k_{\alpha}\ast k_{\zeta,\beta}\leq
\frac{1}{\gamma_{2n}(\alpha+\beta)}
\cdot\frac{1}{\rho^{2n-\alpha-\beta}}+O\left(\frac{1}{\rho^{2n-\alpha-\beta-\epsilon_{2}}}\right).
\]
Similarly, by Lemma \ref{lm3.5}, we get (\ref{3.25}).
These complete  the proof of Lemma \ref{lm3.6}.
\end{proof}

\begin{lemma}\label{lm3.7}
Let $n\geq2$,  $\zeta>0$, $0<\alpha<3$ and $0<\beta<2n-\alpha$. For each $\zeta'\in(0,\zeta)$, we have
\[
k_{\alpha}\ast k_{\zeta,\beta}\lesssim e^{-\zeta'\rho-n\rho}+\rho^{\alpha-2}e^{-n\rho}\ast k_{\zeta,\beta},\;\;\rho>1.
\]
\end{lemma}
\begin{proof} We have, by (\ref{3.6}),
\begin{equation*}%\label{3.20}
\begin{split}
k_{\alpha}\ast k_{\zeta,\beta}=&\int_{\mathbb{B}_{\mathbb{C}}^{n}}k_{\alpha}(\rho(z))k_{\zeta,\beta}(\rho(z,a))dV(z)\\
=&\int_{\rho(z)<1/2}k_{\alpha}(\rho(z))k_{\zeta,\beta}(\rho(z,a))dV(z)+\int_{\rho(z)\geq1/2}k_{\alpha}(\rho(z))k_{\zeta,\beta}(\rho(z,a))dV(z)\\
\lesssim&
\int_{\rho(z)<1/2}k_{\alpha}(\rho(z))k_{\zeta,\beta}(\rho(z,a))dV(z)+
\int_{\rho(z)\geq1/2}\rho(z)^{\alpha-2}e^{-n\rho(z)}k_{\zeta,\beta}(\rho(z,a))dV(z)\\
\leq&\int_{\rho(z)<1/2}k_{\alpha}(\rho(z))k_{\zeta,\beta}(\rho(z,a))dV(z) +
\rho^{\alpha-2}e^{-z\rho}\ast k_{\zeta,\beta}.
\end{split}
\end{equation*}
Next we shall show that
\[
\int_{\rho(z)<1/2}k_{\alpha}(\rho(z))k_{\zeta,\beta}(\rho(z,a))dV(z)\lesssim e^{-\zeta'\rho(a)-n\rho(a)}, \;\rho(a)\geq1.
\]

Notice that, if $\rho(z)<1/2$, then $\rho(z,a)\geq\rho(z)-\rho(a)\geq1/2$ since $\rho(y)\geq1$.
We have, by ,
\begin{equation*}%\label{4.10}
\begin{split}
k_{\alpha}(\rho(z))\lesssim& \frac{1}{\rho(z)^{2n-\alpha}}\thicksim \frac{1}{|z|^{2n-\alpha}},\;\;\;\;\;
\;\;\;\;\; \;\;\;\;\; \;\;\;\;\; \;\;\;\;\; \;\;\;\;\; \;\;\;\;\; \;\;\;\;\; \;\;\rho(z)<1/2;\\
k_{\zeta,\beta}(\rho(z,a))\lesssim& e^{-\zeta'\rho(z,a)-n\rho(z,a)}\thicksim
\left(\cosh\rho(z,a)\right)^{-(n+\zeta')},\;\;\;\;\rho(z,a)\geq1/2.
\end{split}
\end{equation*}
Therefore,  we have
\begin{equation*}%\label{4.10}
\begin{split}
&\int_{\rho(z)<1/2}k_{\alpha}(\rho(z))k_{\zeta,\beta}(\rho(z,a))dV(z)\\
\lesssim&\int_{\rho(z)<1/2}\frac{1}{|z|^{2n-\alpha}}
\left(\cosh\rho(z,a)\right)^{-(n+\zeta')}\left(\frac{1}{1-|z|^{2}}\right)^{n+1}dz\\
=&\int_{\rho(z)<1/2}\frac{1}{|z|^{2n-\alpha}}\left(\frac{\sqrt{(1-|a|^{2})(1-|z|^{2})}}{|1-(z,a)|}\right)^{n+\zeta'}\left(\frac{1}{1-|z|^{2}}\right)^{n+1}dz\\
\thicksim&(\sqrt{1-|a|^{2}})^{n+\zeta'}
\int_{\rho(z)<1/2}\frac{1}{|z|^{2n-\alpha}}dz\\
\thicksim& e^{-\zeta'\rho-n\rho},\;\;\;\;\;\rho(y)\geq1.
\end{split}
\end{equation*}
 This completes the proof of Lemma \ref{lm3.7}.

\end{proof}

\section{rearrangement of real functions on $\mathbb{B}_{\mathbb{C}}^{n}$}
We now recall the rearrangement of a real functions on $\mathbb{B}_{\mathbb{C}}^{n}$.  Suppose $f$ is
a real  function on $\mathbb{B}_{\mathbb{C}}^{n}$. The non-increasing rearrangement of $f$
is defined by
\begin{equation*}
f^{\ast}(t)=\inf\{s>0: \lambda_{f}(s)\leq t\},
\end{equation*}
where $$\lambda_{f}(s)=|\{x\in \mathbb{B}_{\mathbb{C}}^{n}: |f(z)|>s\}|=\int_{\{x\in \mathbb{B}_{\mathbb{C}}^{n}: |f(z)|>s\}}dV.$$
 Here we use the
notation $|\Sigma|$ for the measure of a measurable set
$\Sigma\subset \mathbb{B}_{\mathbb{C}}^{n}$. Set
\[
f^{\ast\ast}(t)=\frac{1}{t}\int^{t}_{0}f^{\ast}(s)ds.
\]
Denote by
$L^{p,q}(\Omega)$ the Lorentz space of those function $f$ on $\Omega$ satisfying
\[
\|f\|_{L^{p,q}(\Omega)}=\left\{
                          \begin{array}{ll}
                            \left\|t^{\frac{1}{p}-\frac{1}{q}}f^{\ast}(t)\right\|_{L^{q}(0,|\Omega|)}, & \hbox{$1\leq q<\infty$;} \\
                           \sup\limits_{t>0}f^{\ast}(t)t^{1/p} , & \hbox{$q=\infty$}
                          \end{array}
                        \right.
\]
is finite.
For simplicity, we denote by $\|f\|_{p,q}=\|f\|_{L^{p,q}(\mathbb{B}_{\mathbb{C}}^{n})}$.
Similarly, denote by
\[
\|f\|^{\ast}_{L^{p,q}(\Omega)}=\left\{
                          \begin{array}{ll}
                            \left\|t^{\frac{1}{p}-\frac{1}{q}}f^{\ast\ast}(t)\right\|_{L^{q}(0,|\Omega|)}, & \hbox{$1\leq q<\infty$;} \\
                           \sup\limits_{t>0}f^{\ast\ast}(t)t^{1/p} , & \hbox{$q=\infty$}.
                          \end{array}
                        \right.
\]
We  have the following generalization of Young's inequality for convolution (see \cite{on},  Theorem 2.6):
\begin{equation}\label{4.1}
\begin{split}
\|f\ast g\|^{\ast}_{L^{r,s}}\leq&C\|f\|^{\ast}_{L^{p_{1},q_{1}}}\|g\|^{\ast}_{L^{p_{2},q_{2}}}, \;\;f\in L(p_{1},q_{1}),\;\;g\in L(p_{2},q_{2}),
\end{split}
\end{equation}
where $C>0$ and  $r,s,p_{1},p_{2},q_{1}$ and $q_{2}$ satisfy
\[
\frac{1}{p_{1}}+\frac{1}{p_{2}}-1=\frac{1}{r}>0, \;\; s\geq1\;\; \textrm{and}\;\;\frac{1}{q_{1}}+\frac{1}{q_{2}}\geq\frac{1}{s}.
\]
Notice that for $1<q<+\infty$ and $0< r\leq +\infty$, we have
\begin{equation}\label{4.2}
\|f\ast g\|_{L^{q,r}}  \leq\|f\ast g\|^{\ast}_{L^{q,r}}\leq \frac{q}{q-1}\|f\ast g\|_{L^{q,r}}.
\end{equation}
The proof for $1\leq r<+\infty$ can be found in \cite{on} while the rest case has been proved by Yap (see \cite{yap}, Theorem 3.4). is is an exercise in calculus (see also \cite{hu}, equation (2.2) on
page 258). Combining (\ref{4.1}) and (\ref{4.2}) yields the following:
\begin{proposition}
Let $1<r,p_{1},p_{2}<+\infty$ and $1\leq s,q_{1},q_{2}\leq\infty$. If
\[
\frac{1}{p_{1}}+\frac{1}{p_{2}}-1=\frac{1}{r}\;\; \textrm{and}\;\;\frac{1}{q_{1}}+\frac{1}{q_{2}}\geq\frac{1}{s},
\]
then there exists $C>0$ such that

\begin{equation}\label{4.3}
\begin{split}
\|f\ast g\|_{L^{r,s}}\leq&C\|f\|_{L^{p_{1},q_{1}}}\|g\|_{L^{p_{2},q_{2}}}, \;\;f\in L(p_{1},q_{1}),\;\;g\in L(p_{2},q_{2}).
\end{split}
\end{equation}

\end{proposition}

In the previous sections, we obtain the following asymptotic estimates of $k_{\alpha}$ and $k_{\zeta,\alpha}$:
\begin{itemize}
  \item $\zeta>0$:
\begin{equation}\label{4.4}
\begin{split}
k_{\zeta,\alpha}\leq&\frac{1}{\gamma_{2n}(\alpha)}\cdot\frac{1}
{\rho^{2n-\alpha}}+O\left(\frac{1}{\rho^{2n-\alpha-\epsilon}}\right),   \;0<\alpha<n,\;\;0<\rho<1\\
k_{\zeta,\alpha}\thicksim& \rho^{\frac{\alpha-2}{2}}e^{-\zeta\rho-n\rho},
\;\;\;\;\;\;\;\;\;\;\;\;\;\;\;\;\;\;\;\;\;\;\;\alpha>0,
\;\;\rho\geq1;
\end{split}
\end{equation}
  \item $\zeta=0$:
\begin{equation}\label{4.5}
\begin{split}
k_{\alpha}\leq&\frac{1}{\gamma_{2n}(\alpha)}\cdot\frac{1}
{\rho^{2n-\alpha}}+O\left(\frac{1}{\rho^{2n-\alpha-1}}\right),   \;\;\;0<\alpha<3,\;\;0<\rho<1\\
k_{\alpha}\thicksim& \rho^{\alpha-2}e^{-n\rho},
\;\;\;\;\;\;\;\;\;\;\;\;\;\;\;\;\;\;\;\;\;\;\;\;\;\;\;\;\;0<\alpha<3,
\;\;\rho\geq1;
\end{split}
\end{equation}
  \item If $\zeta>0$, $0<\alpha<3$ and $0<\beta<2n-\alpha$, then
\begin{equation}\label{4.6}
\begin{split}
k_{\alpha}\ast k_{\zeta,\beta}
\leq&
\frac{1}{\gamma_{2n}(\alpha+\beta)}
\cdot\frac{1}{\rho^{2n-\alpha-\beta}}+O\left(\frac{1}{\rho^{2n-\alpha-\beta-\epsilon}}\right),\;\;\;0<\rho<1,\\
k_{\alpha}\ast k_{\zeta,\beta}
\lesssim&e^{-\zeta'\rho-n\rho}+\rho^{\alpha-2}e^{-n\rho}\ast k_{\zeta,\beta},\;\;\;\;\;\;\;\;\;\;\;\;\;\;
\rho>1,
\end{split}
\end{equation}
\end{itemize}
where $0<\zeta'<\zeta$ and $\epsilon>0$ is small enough.
If we denote by $B_{\rho}$ the ball  centered at the origin with radius $\rho$, then the volume of $B_{\rho}$ satisfies the estimates
\begin{equation}\label{4.7}
\begin{split}
|B_{\rho}|=&\frac{\omega_{2n-1}}{2n}\rho^{2n}+O(\rho^{2n-1}),\;\;0<\rho<1;\\
|B_{\rho}|\thicksim& e^{2n\rho},\;\;\;\;\;\;\;\;\;\;\;\;\;\;\;\;\;\;\;\;\rho\geq1.
\end{split}
\end{equation}
Therefore, we have, by (\ref{4.1})-(\ref{4.3}), the non-increasing rearrangement of such  kernels satisfies
\begin{itemize}
  \item $\zeta>0$ and $0<\alpha<2n$:
\begin{equation}\label{4.8}
\begin{split}
[k_{\zeta,\alpha}]^{\ast}(t)\leq&\frac{1}{\gamma_{2n}(\alpha)}\cdot\left(\frac{2nt}{\omega_{2n-1}}\right)^{\frac{\alpha-2n}{2n}}
+O\left(t^{\frac{\alpha+\epsilon-2n}{2n}}\right),   \;\;\;0<\alpha<n,\;\;0<t<2,\\
[k_{\zeta,\alpha}]^{\ast}(t)\thicksim& t^{-\frac{1}{2}-\frac{1}{2n}\zeta}(\ln t)^{\frac{\alpha-2}{2}},
\;\;\;\;\;\;\;\;\;\;\;\;\;\;\;\;\;\;\;\;\;\;\;\;\;\;\;\;\;\alpha>0,
\;\;t\geq2;
\end{split}
\end{equation}
  \item $\zeta=0$ and $0<\alpha<3$:
\begin{equation}\label{4.9}
\begin{split}
[k_{\alpha}]^{\ast}(t)\leq&\frac{1}{\gamma_{2n}(\alpha)}\cdot\left(\frac{2nt}{\omega_{2n-1}}\right)^{\frac{\alpha-2n}{2n}}
+O\left(t^{\frac{\alpha+2-2n}{2n}}\right),   \;\;\;0<\alpha<3,\;\;0<t<2,\\
[k_{\alpha}]^{\ast}(t)\thicksim& t^{-\frac{1}{2}}(\ln t)^{\alpha-2},
\;\;\;\;\;\;\;\;\;\;\;\;\;\;\;\;\;\;\;\;\;\;\;\;\;\;\;\;\;\;\;\;\;\;\;\;0<\alpha<3,
\;\;t\geq2;
\end{split}
\end{equation}
  \item  $\zeta>0$, $0<\alpha<3$ and $0<\beta<2n-\alpha$,
\begin{equation}\label{4.10}
\begin{split}
[k_{\alpha}\ast k_{\zeta,\beta}]^{\ast}(t)
\leq&
\frac{1}{\gamma_{2n}(\alpha+\beta)}\cdot\left(\frac{2nt}{\omega_{2n-1}}\right)^{\frac{\alpha+\beta-2n}{2n}}
+O\left(t^{\frac{\alpha+\beta+\epsilon-2n}{2n}}\right),\;\;\;\;\;\;\;\;0<t<2.
\end{split}
\end{equation}
\end{itemize}

Next, we will show the following
\begin{lemma} \label{lm4.1} let $n\geq2$, $0<\alpha<3/2$, $\zeta>0$ and $0<\beta<2n-\alpha$.  Then for each  each $c>0$, we have
$
 \int^{\infty}_{c}|[k_{\alpha}\ast k_{\zeta,\beta}]^{\ast}(t)|^{2}dt<\infty.
$
\end{lemma}

\begin{proof}
It is enough to show that for some $c_{0}>0$, $\int^{\infty}_{c_{0}}|[k_{\alpha}\ast k_{\zeta,\beta}]^{\ast}(t)|^{2}dt<\infty.$
By (\ref{4.6}),  there exists a constant $A>0$ such that
\[
k_{\alpha}\ast k_{\zeta,\beta}
\leq A(e^{-\zeta'\rho-n\rho}+\rho^{\alpha-2}e^{-n\rho}\ast k_{\zeta,\beta}),\;\;0<\zeta'<\zeta,\;\rho>1.
\]
Therefore, for some $c_{0}>0$, we have
\begin{equation*}%\label{4.13}
\begin{split}
\int^{\infty}_{c_{0}}|[k_{\alpha}\ast k_{\zeta,\beta}]^{\ast}(t)|^{2}dt\leq&
\int^{\infty}_{c_{0}}|[A(e^{-\zeta'\rho-n\rho}+\rho^{\alpha-2}e^{-n\rho}\ast k_{\zeta,\beta})]^{*}(t)|^{2}dt\\
\leq&
\int^{\infty}_{0}|[A(e^{-\zeta'\rho-n\rho}+\rho^{\alpha-2}e^{-n\rho}\ast k_{\zeta,\beta})]^{*}(t)|^{2}dt\\
=&A^{2}\int_{\mathbb{B}_{\mathbb{C}}^{n}}|e^{-\zeta'\rho-n\rho}+\rho^{\alpha-2}e^{-n\rho}\ast k_{\zeta,\beta}|^{2}dV\\
\lesssim&\int_{\mathbb{B}_{\mathbb{C}}^{n}}e^{-2\zeta'\rho-2n\rho}dV
+\int_{\mathbb{B}_{\mathbb{C}}^{n}}|\rho^{\alpha-2}e^{-n\rho}\ast k_{\zeta,\beta}|^{2}dV.
\end{split}
\end{equation*}
By (\ref{2.5}), we have
\[
\int_{\mathbb{B}_{\mathbb{C}}^{n}}e^{-2\zeta'\rho-2n\rho}dV=\omega_{2n-1}\int^{\infty}_{0}e^{-2\zeta'\rho-2n\rho}\sinh^{2n-1}\rho\cosh\rho d\rho<\infty,
\]
where $\omega_{2n-1}$ is the volume of $\mathbb{S}^{2n-1}$. On the other hand,
 by the Plancherel formula, we have, for $0<\alpha<3/2$,
\begin{equation*}%\label{4.13}
\begin{split}
&\int_{\mathbb{B}_{\mathbb{C}}^{n}}|\rho^{\alpha-2}e^{-n\rho}\ast k_{\zeta,\beta}|^{2}dV\\
=&C_{n}\int^{+\infty}_{-\infty}\int_{\mathbb{S}^{2n-1}}|\widehat{ k_{\zeta,\beta}}|^{2}\cdot|\widehat{\rho^{\alpha-2} e^{-n\rho}}|^{2}
|\mathfrak{c}(\lambda)|^{-2}d\lambda d\sigma(\varsigma)\\
=&C_{n}\int^{+\infty}_{-\infty}\int_{\mathbb{S}^{2n-1}}\left(\lambda^{2}+\zeta^{2}\right)^{-\frac{\beta}{2}}\cdot
|\widehat{\rho^{\alpha-2} e^{-n\rho}}|^{2}
|\mathfrak{c}(\lambda)|^{-2}d\lambda d\sigma(\varsigma)\\
\lesssim&C_{n}\int^{+\infty}_{-\infty}\int_{\mathbb{S}^{2n-1}}|\widehat{\rho^{\alpha-2} e^{-n\rho}}|^{2}
|\mathfrak{c}(\lambda)|^{-2}d\lambda d\sigma(\varsigma)\\
=&\int_{\mathbb{B}_{\mathbb{C}}^{n}}
\left|\rho^{\alpha-2} e^{-n\rho}\right|^{2}dV=\omega_{2n-1}
\int^{\infty}_{0}\frac{\sinh^{2n-1}\rho\cosh\rho}{\rho^{4-2\alpha}e^{2n\rho}}d\rho<\infty.
\end{split}
\end{equation*}
Thus $\int^{\infty}_{c_{0}}|\phi^{*}(t)|^{2}dt<\infty$ and the desired result follows.
\end{proof}

\section{Proofs of Theorems \ref{th1.7} and \ref{th1.8}}Before the proof of main result in this section, we recall the Kunze-Stein phenomenon on the closed linear group $SU(1,n)$.
For simplicity, we denote by $G=SU(1,n)$ and $K=S(U(1)\times U(n))$.
By $L^{p}(G)$ and $L^{p,q}(G)$, we denote the usual Lebesque space  and Lorentz space, respectively.
We define $L^{p,q}(G/K)$, $L^{p,q}(K\setminus G)$ and $L^{p,q}(K\setminus G/K)$ to be the closed subspaces of $L^{p,q}(G)$ of the
right-$K$-invariant, left-$K$-invariant and $K$-bi-invariant functions, respectively.
 Cowling, Meda and Setti (\cite{co}) gave the following  sharp version of the Kunze-Stein phenomenon for  Lorentz space on $G$
 \begin{equation}\label{bb2.1}
L^{p,q_{1}}(G)\ast L^{p,q_{2}}(G)\subseteq L^{p,q_{3}}(G),
\end{equation}
where $1<p<2$, $1\leq q_{k}\leq \infty (k=1,2,3)$ and $1+1/q_{3}\leq 1/q_{1}+1/q_{2}$ (see  \cite{io} for an endpoint estimate of (\ref{bb2.1})).
\begin{lemma}\label{lma5.1}
There holds, for $p\in(1,2)$,
$$L^{p}(K\setminus G) \ast L^{p}(G/K)  \subset L^{p,\infty}(K\setminus G/K).$$
\end{lemma}
\begin{proof}
By (\ref{bb2.1}), it is enough to show that if $f\in L^{p}(K\setminus G) $ and $h\in  L^{p}(G/K)$, then $f\ast h$ is a $K$-bi-invariant function. For simplicity, we denote by $\mu(dg)$ a left Harr measure on $G$. Since $G$ is semi-simple, $G$ is unimodular and $\mu(dg)$ is also a right Harr measure on $G$.
Therefore  we have, for $k\in K$ and $g\in G$,
\begin{equation*}%\label{4.13}
\begin{split}
f\ast h(kg)=&\int_{G}f(kgg_{1}^{-1})h(g_{1})\mu(dg_{1})=\int_{G}f(gg_{1}^{-1})h(g_{1})\mu(dg_{1})=f\ast h(g)
\end{split}
\end{equation*}
and by substituting $g_{2}=g_{1}k^{-1}$,
\begin{equation*}%\label{4.13}
\begin{split}
f\ast h(gk)=&\int_{G}f(gkg_{1}^{-1})h(g_{1})\mu(dg_{1})=\int_{G}f(g(g_{1}k^{-1})^{-1})h(g_{1})\mu(dg_{1})\\
=&\int_{G}f(gg_{2}^{-1})h(g_{2}k)\mu(dg_{2})=\int_{G}f(gg_{2}^{-1})h(g_{2})\mu(dg_{2})=f\ast h(g).
\end{split}
\end{equation*}
The proof of Lemma \ref{lma5.1} is thereby completed.
\end{proof}

\begin{lemma}\label{lma5.2}
There holds, for $p\in(1,2)$ and $p'=p/(p-1)$,
\begin{equation}\label{b2.1}
L^{p',1}(K\setminus G/K) \ast  L^{p}(G/K) \subset L^{p'}(G/K).
\end{equation}
\end{lemma}
\begin{proof} Since $G$ is unimodular, we have
\begin{equation}\label{aa5.3}
\int_{G}v(g^{-1})\mu(dg)=\int_{G}v(g)\mu(dg),\;\;v\in L^{1}(G).
\end{equation}
Let $f\in L^{p',1}(K\setminus G/K)$ and $h\in L^{p}(G/K)$. By (\ref{2.8}), we have
\begin{equation*}%\label{a5.2}
\begin{split}
f\ast h(g)=h\ast f(g)=\int_{G}h(gg^{-1}_{1})f(g_{1})\mu(dg_{1}).
\end{split}
\end{equation*}
Set $h_{1}(g)=h(g^{-1})$. Then $h_{1}$ is left-$K$-invariant since $h$ is right-$K$-invariant. Furthermore, by (\ref{aa5.3}), we have
$$\|h_{1}\|_{L^{p}(K\setminus G)}=\|h\|_{L^{p}(G/K)}.$$

By duality, we have
\begin{equation*}%\label{a5.2}
\begin{split}
\|f\ast h\|_{L^{p'}(G/K)}=&\sup_{\|u\|_{L^{p}(G/K)}\leq1}\int_{G}u(g)f\ast h(g)\mu(dg)=\sup_{\|u\|_{L^{p}(G/K)}\leq1}\int_{G}u(g)h\ast f(g)\mu(dg)\\
=&\sup_{\|u\|_{L^{p}(G/K)}\leq1}\int_{G}\int_{G}u(g)h_{1}(g_{1}g^{-1})f(g_{1})\mu(dg)\mu(dg_{1})\\
=&\sup_{\|u\|_{L^{p}(G/K)}\leq1}\int_{G}f(g_{1})h_{1}\ast u(g_{1})\mu(dg_{1})\\
\lesssim&\sup_{\|u\|_{L^{p}(G/K)}\leq1}\|f\|_{L^{p',1}(K\setminus G/K)}\|h_{1}\ast u\|_{L^{p,\infty}(K\setminus G/K)}.
\end{split}
\end{equation*}
To get the last inequality above, we use the H\"older's inequality for Lorentz spaces. On the other hand,
by Lemma \ref{lma5.1}, we have
\begin{equation*}
  \|h_{1}\ast u\|_{L^{p,\infty}(K\setminus G/K)}\lesssim \|h_{1}\|_{L^{p}(K\setminus G)}\|u\|_{L^{p}( G/K)}=
  \|h\|_{L^{p}(G/K)}\|u\|_{L^{p}( G/K)}.
\end{equation*}
Therefore,
\begin{equation*}%\label{a5.2}
\begin{split}
\|f\ast h\|_{L^{p'}(G/K)}\lesssim&\sup_{\|u\|_{L^{p}(G/K)}\leq1}\|f\|_{L^{p',1}(K\setminus G/K)}\|h\|_{L^{p}(G/K)}\|u\|_{L^{p}( G/K)}\\
\lesssim&\|f\|_{L^{p',1}(K\setminus G/K)}\|h\|_{L^{p}(G/K)}.
\end{split}
\end{equation*}
The desired result follows.
\end{proof}

\begin{lemma}
\label{lm5.1}
Let $0< \alpha<3$, $0<\beta <2n-\alpha$, $\zeta>0$ and  $\frac{4n}{2n+\alpha+\beta}\leq p<2$. Denote by $p'=p/(p-1)$. Then there exists $C>0$
such that for all $f\in C^{\infty}_{0}(\mathbb{B}_{\mathbb{C}}^{n})$,
\begin{equation}\label{a5.1}
  \|k_{\alpha}\ast k_{\zeta,\beta}\ast f\|_{p'}\leq C\|f\|_{p}.
\end{equation}
\end{lemma}
\begin{proof}
Set
\begin{equation*}
  \eta_{1}(\rho)=\left\{
                   \begin{array}{ll}
                    k_{\alpha}\ast k_{\zeta,\beta} , & \hbox{$0<\rho<1$;} \\
                     0, & \hbox{$\rho\geq1$.}
                   \end{array}
                 \right.
\end{equation*}
and $\eta_{2}(\rho)=k_{\alpha}\ast k_{\zeta,\beta}- \eta_{1}(\rho)$. By (\ref{4.10}),  there exists $t_{0}>0$ such that
\[
\eta_{1}^{\ast}(t)\lesssim t^{\frac{\alpha+\beta-2n}{2n}}, \; t\leq t_{0},\;\; \textrm{and}\;\;\; \eta_{1}^{\ast}(t)=0,\;t>t_{0}.
\]
Therefore, we have, by (\ref{4.3}),
\begin{equation}\label{a5.2}
\begin{split}
\|\eta_{1}\ast f\|_{p'}=\|\eta_{1}\ast f\|_{L^{p',p'}}\leq C \|\eta_{1}\|_{L^{\frac{p'}{2},\infty}}\| f\|_{p}\lesssim \| f\|_{p}.
\end{split}
\end{equation}
Here we use the fact
\[
\|\eta_{1}\|_{L^{\frac{p'}{2},\infty}}=\sup_{0<t<\infty}t^{\frac{2}{p'}}
\eta_{1}^{\ast}(t)\lesssim \sup_{0<t\leq t_{0}}t^{\frac{2}{p'}+\frac{\alpha+\beta-2n}{2n}}<\infty.
\]
On the other hand, by Lemma \ref{lm3.6},
\[
\eta_{2}(\rho)=0\; \textrm{for}\; 0<\rho<1 \;  \textrm{and}\;   \eta_{2}(\rho)\leq e^{(\varepsilon-n)\rho},  \rho\geq1.
\]
We have, by  (\ref{4.10}),
$$\eta_{2}^{\ast}(t)\lesssim 1,\;\;0<t<1\;\;\textrm{and}\;\;\eta_{2}^{\ast}(t)\leq t^{-\frac{n-\varepsilon}{2n}},\;\; t\geq1.$$
Therefore,
\[
\|\eta_{2}\|_{p',1}=\int^{\infty}_{0}t^{\frac{1}{p'}-1}\eta_{2}^{\ast}(t)dt<\infty
\]
if we choose $0<\varepsilon<2n(1/p-1/2)$.
Therefore, by Lemma \ref{lma5.2}, we have
\begin{equation}\label{a5.3}
  \|\eta_{2}\ast f\|_{p'}\leq C \|f\|_{p}.
\end{equation}
Combining (\ref{a5.2}) and (\ref{a5.3}) yields
\begin{equation}\label{a5.4}
  \|k_{\alpha}\ast k_{\zeta,\beta}\ast f\|_{p'}\leq
  \|\eta_{1}\ast f\|_{p'}+  \|\eta_{2}\ast f\|_{p'}\leq C\|f\|_{p}.
\end{equation}
This completes the proof of Lemma \ref{lm5.1}.
\end{proof}

\textbf{Proof of Theorem \ref{th1.7} .}
By Lemma \ref{lm5.1}, we have
\begin{equation*}%\label{aa5.8}
  \|(-\Delta_{\mathbb{B}}-n^{2}+\zeta^{2})^{-\beta/4}(-\Delta_{\mathbb{B}}-n^{2})^{-\alpha/4}f\|_{p'}\leq \|f\|_{p}
\end{equation*}
which is equivalent to  (\ref{1.5}) (see e.g. \cite{be}, Appendix).
 Now, we prove inequality (\ref{1.6}).
For each $p>2$, we choose $\beta'<\beta=2n-\alpha$ be such that $p\leq \frac{4n}{2n-(\alpha+\beta')}$. Then by (\ref{1.5}), we have
\begin{equation*}%\label{5.5}
  \|f\|_{p}\leq C\|(-\Delta_{\mathbb{B}}-n^{2}+\zeta^{2})^{\beta'/4}(-\Delta_{\mathbb{B}}-n^{2})^{\alpha/4}f\|_{2}.
\end{equation*}
On the other hand, by the Plancherel formula, we have
\begin{equation*}%\label{5.6}
  \begin{split}
&\|(-\Delta_{\mathbb{B}}-n^{2}+\zeta^{2})^{\beta/4}(-\Delta_{\mathbb{B}}-n^{2})^{\alpha/4}f\|_{2}\\
=&C_{n}\int^{+\infty}_{-\infty}\int_{\mathbb{S}^{2n-1}}(\lambda^{2}+\zeta^{2})^{\beta/2}|\lambda|^{\alpha}|\widehat{f}(\lambda,\zeta)|^{2}
|\mathfrak{c}(\lambda)|^{-2}d\lambda d\sigma(\varsigma)\\
\geq &\zeta^{\beta-\beta'}
C_{n}\int^{+\infty}_{-\infty}\int_{\mathbb{S}^{2n-1}}(\lambda^{2}+\zeta^{2})^{\beta'/2}|\lambda|^{\alpha}|\widehat{f}(\lambda,\zeta)|^{2}
|\mathfrak{c}(\lambda)|^{-2}d\lambda d\sigma(\varsigma)\\
=&\zeta^{\beta-\beta'}\|(-\Delta_{\mathbb{B}}-n^{2}+\zeta^{2})^{\beta'/4}(-\Delta_{\mathbb{B}}-n^{2})^{\alpha/4}f\|_{2}\geq C\|f\|_{p}.
  \end{split}
\end{equation*}
The desired result follows. \
\\

\textbf{Proof of Theorem \ref{th1.8} .}
Set $u=\varrho^{\frac{k-n-a}{2}}f$, we have, by Theorem \ref{th1.6},
\begin{equation}\label{a5.7}
  \begin{split}
 & 4^{k}\int_{\mathbb{H}^{2n-1}}\int^{\infty}_{0}u\prod^{k}_{j=1}\left[-\varrho\partial_{\varrho\varrho}-a\partial_{\varrho}-\varrho T^{2}- \Delta_{b}+i(k+1-2j)T\right]u\frac{dzdtd\varrho}{\varrho^{1-a}}\\
  =&\int_{\mathbb{H}^{2n-1}}\int^{\infty}_{0}
  f\prod^{k}_{j=1}\left[-\Delta_{\mathbb{B}}-n^{2}+(a-k+2j-2)^{2}\right]f\frac{dzdtd\varrho}{\varrho^{n+1}}\\
  =&4\int_{\mathcal{U}^{n}}
  f\prod^{k}_{j=1}\left[-\Delta_{\mathbb{B}}-n^{2}+(a-k+2j-2)^{2}\right]fdV.
  \end{split}
\end{equation}
Since $\textrm{spec}(-\Delta_{\mathbb{B}})=[n^{2},+\infty)$, we have
\begin{equation*}%\label{5.8}
  \begin{split}
&\int_{\mathcal{U}^{n}}
 f\prod^{k}_{j=1}\left[-\Delta_{\mathbb{B}}-n^{2}+(a-k+2j-2)^{2}\right]fdV\\
  \geq&
  \prod^{k}_{j=1}(a-k+2j-2)^{2}\int_{\mathcal{U}^{n}}
  f^{2}dV
  \end{split}
\end{equation*}
and the constant $\prod\limits^{k}_{j=1}(a-k+2j-2)^{2}$ is sharp.
Furthermore, by Plancherel formula, we have
\begin{equation*}%\label{5.9}
  \begin{split}
&\int_{\mathcal{U}^{n}}
  f\prod^{k}_{j=1}\left[-\Delta_{\mathbb{B}}-n^{2}+(a-k+2j-2)^{2}\right]fdV-\prod^{k}_{j=1}(a-k+2j-2)^{2}\int_{\mathcal{U}^{n}}
  f^{2}dV\\
  =&C_{n}\int^{+\infty}_{-\infty}\int_{\mathbb{S}^{2n-1}}\left[
  \prod^{k}_{j=1}\left(\lambda^{2}+(a-k+2j-2)^{2}\right)-\prod^{k}_{j=1}(a-k+2j-2)^{2}\right]|\widehat{f}(\lambda,\zeta)|^{2}
|\mathfrak{c}(\lambda)|^{-2}d\lambda d\sigma(\varsigma).
  \end{split}
\end{equation*}
Let $\delta>0$ be such that
\[
  \prod^{k}_{j=1}\left(\lambda^{2}+(a-k+2j-2)^{2}\right)-\prod^{k}_{j=1}(a-k+2j-2)^{2}\geq \lambda^{2}(\lambda^{2}+\delta)^{k-1},\;\forall \lambda
  \in \mathbb{R}.
\]
Then by Theorem \ref{th1.7}, we have
\begin{equation}\label{a5.9}
  \begin{split}
&\int_{\mathcal{U}^{n}}
  f\prod^{k}_{j=1}\left[-\Delta_{\mathbb{B}}-n^{2}+(1/2-k+2j-2)^{2}\right]fdV-\prod^{k}_{j=1}(a-k+2j-2)^{2}\int_{\mathcal{U}^{n}}
  f^{2}dV\\
\geq&C_{n}\int^{+\infty}_{-\infty}\int_{\mathbb{S}^{2n-1}}\lambda^{2}(\lambda^{2}+\delta)^{k-1}
 |\widehat{f}(\lambda,\zeta)|^{2}
|\mathfrak{c}(\lambda)|^{-2}d\lambda d\sigma(\varsigma)\\
=&\int_{\mathcal{U}^{n}}
  f(-\Delta_{\mathbb{B}}-n^{2})\left(-\Delta_{\mathbb{B}}-n^{2}+\delta\right)^{k-1}fdV\geq C\|f\|^{2}_{p}.
  \end{split}
\end{equation}
Combining (\ref{a5.7}) and (\ref{a5.9}) yields
\begin{equation*}%\label{1.7}
  \begin{split}
 & 4^{k}\int_{\mathbb{H}^{2n-1}}\int^{\infty}_{0}u\prod^{k}_{j=1}\left[-\varrho\partial_{\varrho\varrho}-a\partial_{\varrho}-\varrho T^{2}- \Delta_{b}+i(k+1-2j)T\right]u\frac{dzdtd\varrho}{\varrho^{1-a}}\\
  &-\prod^{k}_{j=1}(a-k+2j-2)^{2}\int_{\mathbb{H}^{2n-1}}\int^{\infty}_{0}\frac{u^{2}}{\varrho^{k+1-a}}dzdtd\varrho\\
  =&4\int_{\mathcal{U}^{n}}
  f\prod^{k}_{j=1}\left[-\Delta_{\mathbb{B}}-n^{2}+(1/2-k+2j-2)^{2}\right]fdV-
  4\prod^{n}_{j=1}\frac{(2j-1)^{2}}{4}\int_{\mathcal{U}^{n}}
  f^{2}dV\\
  \geq&C\|f\|^{2}_{p}=C\left(\int_{\mathbb{H}^{2n-1}}\int^{\infty}_{0}|u|^{p}\varrho^{\gamma}dzdtd\varrho\right)^{\frac{2}{p}}.
  \end{split}
\end{equation*}
Similarly, we can also obtain the Hardy-Sobolev-Maz'ya inequality on $\mathbb{B}_{\mathbb{C}}^{n}$. Since the proof is similar, we omit it.
This completes proof of Theorem \ref{th1.8}.

\section{Proofs of Theorems \ref{th1.10} and \ref{th1.11}}
Firstly,  we shall prove Theorem  \ref{th1.10}. The main idea is to adapt the level set argument  developed by Lam and the first author  to derive a global
Moser-Trudinger inequality
from a local one (see \cite{ll,ll2}).\
\\

\textbf{Proof of Theorem \ref{th1.10}}   Let $u\in
C^{\infty}_{0}(\mathbb{B}^{n}_{\mathbb{C}})$ be such that
\[
\|(-\Delta_{\mathbb{B}}-n^{2}+\zeta^{2})^{(2n-\alpha)/4}(-\Delta_{\mathbb{B}}-n^{2})^{\alpha/4}u\|_{2}\leq1.
\]
 Set $\Omega(u)=\{x\in\mathbb{B}_{\mathbb{C}}^{n}:|u(x)|\geq1\}$. By Theorem \ref{th1.8},
 we have, for $p>2$,
\begin{equation*}%\label{5.3}
\begin{split}
|\Omega(u)|=&\int_{\Omega(u)}dV\leq\int_{\mathbb{B}^{n}_{\mathbb{C}}}|u|^{p}dV\lesssim 1.
\end{split}
\end{equation*}
Therefore, we may assume
\begin{equation}\label{5.3}
\begin{split}
|\Omega(u)|\leq \Omega_{0}
\end{split}
\end{equation}
for some constant $\Omega_{0}$ which is independent of $u$.

We write
\begin{equation*}%\label{4.2}
\begin{split}
&\int_{\mathbb{B}_{\mathbb{C}}^{n}}(e^{\beta_{0}(n,2n) u^{2}}-1-\beta_{0}(n,2n) u^{2})dV\\
=&\int_{\Omega(u)}(e^{\beta_{0}(n,2n) u^{2}}-1-\beta_{0}(n,2n) u^{2})dV+
\int_{\mathbb{B}^{n}_{\mathbb{C}}\setminus\Omega(u)}(e^{\beta_{0}(n,2n) u^{2}}-1-\beta_{0}(n,2n) u^{2})dV\\
\leq&\int_{\Omega(u)}e^{\beta_{0}(n,2n) u^{2}}dV+
\int_{\mathbb{B}^{n}_{\mathbb{C}}\setminus\Omega(u)}(e^{\beta_{0}(n,2n)
u^{2}}-1-\beta_{0}(n,2n) u^{2})dV.
\end{split}
\end{equation*}
Notice that on the domain $\mathbb{B}^{n}_{\mathbb{C}}\setminus\Omega(u)$, we
have $|u(z)|<1$. Thus,
\begin{equation*}%\label{5.6}
\begin{split}
&\int_{\mathbb{B}^{n}_{\mathbb{C}}\setminus\Omega(u)}(e^{\beta_{0}(n,2n) u^{2}}-1-\beta_{0}(n,2n) u^{2})dV\\
=&\int_{\mathbb{B}^{n}_{\mathbb{C}}\setminus\Omega(u)}\sum^{\infty}_{n=2}\frac{(\beta_{0}(n,2n) u^{2})^{n}}{n!}dV\\
\leq&\int_{\mathbb{B}^{n}_{\mathbb{C}}\setminus\Omega(u)}\sum^{\infty}_{n=2}\frac{(\beta_{0}(n,2n) )^{n}u^{4}}{n!}dV\\
\leq&\sum^{\infty}_{n=2}\frac{(\beta_{0}(n,2n) )^{n}}{n!}\int_{\mathbb{B}^{n}_{\mathbb{C}}}|u(x)|^{4}dV.
\end{split}
\end{equation*}
Therefore, by Theorem \ref{th1.8},
$\int_{\mathbb{B}^{n}\setminus\Omega(u)}(e^{\beta_{0}(n,2n) u^{2}}-1-\beta_{0}(n,2n) u^{2})dV$ is bounded by some  constant  which is independent of $u$.

To finish the proof, it is enough to show
$\int_{\Omega(u)}e^{\beta_{0}(n,2n) u^{2}}dV$ is also bounded by some universal
constant.
Set
\[
v=(-\Delta_{\mathbb{B}}-n^{2}+\zeta^{2})^{(2n-\alpha)/4}(-\Delta_{\mathbb{B}}-n^{2})^{\alpha/4}
u.
\]
Then
\begin{equation}\label{5.5}
\begin{split}
\int_{\mathbb{B}_{\mathbb{C}}^{n}}|v|^{2}dV\leq1
\end{split}
\end{equation}
and  we can write $u$ as a potential
\begin{equation}\label{5.6}
\begin{split}
u=(-\Delta_{\mathbb{B}}-n^{2}+\zeta^{2})^{-(2n-\alpha)/4}(-\Delta_{\mathbb{B}}-n^{2})^{-\alpha/4}v=v\ast\phi,
\end{split}
\end{equation}
where $\phi=(-\Delta_{\mathbb{B}}-n^{2}+\zeta^{2})^{-(2n-\alpha)/4}(-\Delta_{\mathbb{B}}-n^{2})^{-\alpha/4}=k_{\zeta,(2n-\alpha)/2}\ast k_{\alpha/2}$.
By (\ref{4.10}) and Lemma \ref{lm4.1},
\begin{equation*}
  \phi^{\ast}(t)\leq \frac{1}{\gamma_{2n}(n)}\cdot\left(\frac{2nt}{\omega_{2n-1}}\right)^{-\frac{1}{2}}
+O\left(t^{\frac{\epsilon-n}{2n}}\right),\;\;\;\;0<t<2\;\;\;\; \textrm{and}\;\;\;\; \int_{c}^{\infty}|\phi^{\ast}(t)|^{2}dt<\infty,\;\forall c>0.
\end{equation*}
Closely  following the   proof of  Theorem 1.7  in \cite{LiLuy1},  we have that there exists a constant $C$ which is independent of  $u$ and $\Omega(u)$ such that
 \begin{equation*}%\label{b5.3}
\begin{split}
&\int_{\Omega(u)}e^{\beta_{0}(n,2n) u^{2}}dV=\int^{|\Omega(u)|}_{0}\exp(\beta_{0}(n,2n)|u^{\ast}(t)|^{2})dt\leq\int^{\Omega_{0}}_{0}\exp(\beta_{0}(n,2n)|u^{\ast}(t)|^{2})dt\leq C.
\end{split}
\end{equation*}
 The proof of Theorem \ref{th1.10} is thereby completed.
\
\\

\textbf{Proof of  Theorem  \ref{th1.11}}
It is enough to show that for some $\zeta>0$,
\begin{equation*}%\label{5.12}
\begin{split}
&\|(-\Delta_{\mathbb{B}}-n^{2}+\zeta^{2})^{(n-1)/2}(-\Delta_{\mathbb{B}}-n^{2})^{1/2}u\|_{2}\\
\leq&4^{n-1}\int_{\mathbb{H}^{2n-1}}\int^{\infty}_{0}\varrho^{-a/2}u\prod^{n}_{j=1}\left[-\varrho\partial_{\varrho\varrho}-a\partial_{\varrho}+\varrho T^{2}- \Delta_{b}+i(k+1-2j)T\right](\varrho^{-a/2}u)\frac{dzdtd\varrho}{\varrho^{1-a}}\\
  &-\frac{1}{4}\prod^{n}_{j=1}(a-n+2j-2)^{2}\int_{\mathbb{H}^{2n-1}}\int^{\infty}_{0}\frac{(\varrho^{-a/2}u)^{2}}{\varrho^{n+1-a}}dzdtd\varrho.
\end{split}
\end{equation*}
In fact, by Theorem \ref{th1.7}  and the Plancherel formula,
\begin{equation*}%\label{5.13}
\begin{split}
&4^{n-1}\int_{\mathbb{H}^{2n-1}}\int^{\infty}_{0}\varrho^{-a/2}u\prod^{n}_{j=1}\left[-\varrho\partial_{\varrho\varrho}-a\partial_{\varrho}+\varrho T^{2}- \Delta_{b}+i(k+1-2j)T\right](\varrho^{-a/2}u)\frac{dzdtd\varrho}{\varrho^{1-a}}\\
  &-\frac{1}{4}\prod^{n}_{j=1}(a-n+2j-2)^{2}\int_{\mathbb{H}^{2n-1}}\int^{\infty}_{0}\frac{(\varrho^{-a/2}u)^{2}}{\varrho^{n+1-a}}dzdtd\varrho\\
  =&\int_{\mathcal{U}^{n}}u\prod^{n}_{j=1}\left[-\Delta_{\mathbb{B}}-n^{2}+(a-k+2j-2)^{2}\right]udV
-\prod^{n}_{j=1}(a-n+2j-2)^{2}\int_{\mathcal{U}^{n}}u^{2}dV\\
=&C_{n}\int^{+\infty}_{-\infty}\int_{\mathbb{S}^{2n-1}}\left[
  \prod^{n}_{j=1}\left(\lambda^{2}+(a-k+2j-2)^{2}\right)-\prod^{n}_{j=1}(a-k+2j-2)^{2}\right]|\widehat{u}(\lambda,\zeta)|^{2}
|\mathfrak{c}(\lambda)|^{-2}d\lambda d\sigma(\varsigma)\\
\geq&C_{n}\int^{+\infty}_{-\infty}\int_{\mathbb{S}^{2n-1}}\lambda^{2}(\lambda^{2}+\delta)^{k-1}
 |\widehat{f}(\lambda,\zeta)|^{2}
|\mathfrak{c}(\lambda)|^{-2}d\lambda d\sigma(\varsigma)\\
=&\|(-\Delta_{\mathbb{B}}-n^{2}+\delta)^{(n-1)/2}(-\Delta_{\mathbb{B}}-n^{2})^{1/2}u\|^{2}_{2},
\end{split}
\end{equation*}
where
$\delta>0$ is  such that
\[
  \prod^{n}_{j=1}\left(\lambda^{2}+(a-k+2j-2)^{2}\right)-\prod^{n}_{j=1}(a-k+2j-2)^{2}\geq \lambda^{2}(\lambda^{2}+\delta)^{n-1},\;\forall \lambda
  \in \mathbb{R}.
\]
This completes the proof of  Theorem  \ref{th1.11}.\\

\section{Proofs of Theorems \ref{th1.12} and  \ref{th1.13}}
\textbf{Proof of Theorem \ref{th1.12}}
Since $u\in W^{\alpha,p}(\mathbb{B}^{n}_{\mathbb{C}})$, we write $f=(-\Delta_{\mathbb{B}}-n^{2}+\zeta^{2})^{\frac{\alpha}{2}}u$. Then
$\|f\|_{p}\leq1$ and
$$u=(-\Delta_{\mathbb{B}}-n^{2}+\zeta^{2})^{-\frac{\alpha}{2}}f=f\ast k_{\zeta,\alpha}.$$
 Applying O'Neil's lemma (\cite{on}), we have for $t>0$,
\[
u^{\ast}(t)\leq\frac{1}{t}\int^{t}_{0}f^{\ast}(s)ds\int^{t}_{0} k_{\zeta,\alpha}^{\ast}(s)ds+
\int^{\infty}_{t}f^{\ast}(s) k_{\zeta,\alpha}^{\ast}(s)ds.
\]
By (\ref{4.8}), we have
\begin{equation*}%\label{4.8}
\begin{split}
[k_{\zeta,\alpha}]^{\ast}(t)\leq&\frac{1}{\gamma_{2n}(\alpha)}\left(\frac{2nt}{\omega_{2n-1}}\right)^{\frac{\alpha-2n}{2n}}
+O\left(t^{\frac{\alpha+\epsilon-2n}{2n}}\right),   \;0<t<2,\;\; \textrm{and }\;\;
\int^{\infty}_{c}|[k_{\zeta,\alpha}]^{\ast}(t)|^{p'}dt<\infty,\;\;\forall c>0.
\end{split}
\end{equation*}
Using O'Neil's lemma and closely  following the   proof of  Theorem 1.13  in \cite{LiLuy2},  we have that there exists a constant $C$ which is independent of  $u$  such that
 \begin{equation*}%\label{5.1}
\begin{split}
&\frac{1}{|E|}\int_{E}\exp(\beta(2n,\alpha)|u|^{p'})dV\leq
\frac{1}{|E|}\int^{|E|}_{0}\exp(\beta(2n,\alpha)|u^{\ast}(t)|^{p'})dt\\
\leq&\frac{1}{|E|}\int^{|E|}_{0}\exp\left(\beta(2n,\alpha)\left|\frac{1}{t}\int^{t}_{0}f^{\ast}(s)ds\int^{t}_{0} k_{\zeta,\alpha}^{\ast}(s)ds+
\int^{\infty}_{t}f^{\ast}(s) k_{\zeta,\alpha}^{\ast}(s)ds\right|^{p'}\right)dt\leq C.
\end{split}
\end{equation*}

The sharpness of the constant $\beta(2n,\alpha)$ can be verified by the process similar to that in \cite{ad,ksw}  and thus the
 proof of Theorem \ref{th1.12} is completed.\
\\

 Next we will prove Theorem \ref{th1.13}. The main idea is also apply the method of the level set argument for functions under consideration
and then derive a global
 inequality
from a local one.

\textbf{Proof of Theorem \ref{th1.13}}   Let $u\in W^{\alpha,p}(\mathbb{B}^{n}_{\mathbb{C}})$ with $\int_{\mathbb{B}^{n}_{\mathbb{C}}}|(-\Delta_{\mathbb{B}}-
n^{2}+\zeta^{2})^{\frac{\alpha}{2}} u|^{p}dV\leq1$. By Corollary \ref{lm2.2}, we have
\[
\int_{\mathbb{B}^{n}_{\mathbb{C}}}|u|^{p}dV\lesssim \int_{\mathbb{B}^{n}_{\mathbb{C}}}|(-\Delta_{\mathbb{B}}-
n^{2}+\zeta^{2})^{\frac{\alpha}{2}} u|^{p}dV\leq1
\]
provided $\zeta>2n|\frac{1}{2}-\frac{1}{p}|$.
 Set $\Omega(u)=\{z\in\mathbb{B}^{n}_{\mathbb{C}}:|u(z)|\geq1\}$. Then
 we have
\begin{equation*}%\label{5.3}
\begin{split}
|\Omega(u)|=&\int_{\Omega(u)}dV\leq\int_{\mathbb{B}^{n}_{\mathbb{C}}}|u|^{p}dV\leq  \Omega_{0},
\end{split}
\end{equation*}
where $\Omega_{0}$ is a constant  independent of $u$.
We write
\begin{equation*}%\label{4.2}
\int_{\mathbb{B}^{n}_{\mathbb{C}}}\Phi_{p}(\beta(2n,\alpha)|u|^{p'})dV
=\int_{\Omega(u)}\Phi_{p}(\beta(2n,\alpha)|u|^{p'})dV+\int_{\mathbb{B}^{n}_{\mathbb{C}}\setminus\Omega(u)}\Phi_{p}(\beta(2n,\alpha)|u|^{p'})dV.
\end{equation*}
Notice that on the domain $\mathbb{B}^{n}_{\mathbb{C}}\setminus\Omega(u)$, we
have $|u(z)|<1$. Thus,
\begin{equation}\label{5.13}
\begin{split}
\int_{\mathbb{B}^{n}_{\mathbb{C}}\setminus\Omega(u)}\Phi_{p}(\beta(2n,\alpha)|u|^{p'})dV\leq
&\sum^{\infty}_{k=j_{p}-1}\frac{\beta(2n,\alpha)^{k}}{k!}\int_{\mathbb{B}^{n}_{\mathbb{C}}\setminus\Omega(u)}\sum^{\infty}_{n=2}|u|^{p'k}dV\\
\leq&\sum^{\infty}_{k=j_{p}-1}\frac{\beta(2n,\alpha)^{k}}{k!}\int_{\mathbb{B}^{n}_{\mathbb{C}}\setminus\Omega(u)}\sum^{\infty}_{n=2}|u|^{p}dV\\
\leq&\sum^{\infty}_{k=j_{p}-1}\frac{\beta(2n,\alpha)^{k}}{k!}\|u\|^{p}_{p}\leq C.
\end{split}
\end{equation}
Moreover, by Theorem \ref{th1.12}, if $\zeta$ satisfies $\zeta>0$ if $1<p<2$
and $\zeta>2n\left|\frac{1}{p}-\frac{1}{2}\right|$ if $p\geq2$, then
\begin{equation}\label{5.14}
\begin{split}
\int_{\Omega(u)}\Phi_{p}(\beta(2n,\alpha)|u|^{p'})dV\leq&\int_{\Omega(u)}\exp(\beta(2n,\alpha)|u|^{p'})dV\leq C.
\end{split}
\end{equation}
Combining (\ref{5.13}) and (\ref{5.14}) yields
\begin{equation*}%\label{4.2}
\begin{split}
\int_{\mathbb{B}^{n}_{\mathbb{C}}}\Phi_{p}(\beta(2n,\alpha)|u|^{p'})dV=&
=\int_{\Omega(u)}\Phi_{p}(\beta(2n,\alpha)|u|^{p'})dV+\int_{\mathbb{B}^{n}_{\mathbb{C}}\setminus\Omega(u)}\Phi_{p}(\beta(2n,\alpha)|u|^{p'})dV\leq C
\end{split}
\end{equation*}
provided that  $\zeta$ satisfies   $\zeta>2n\left|\frac{1}{p}-\frac{1}{2}\right|$.

On the other hand,
\begin{equation*}%\label{4.2}
\begin{split}
\int_{\mathbb{B}^{n}_{\mathbb{C}}}e^{\beta(2n,\alpha)|u|^{p'}}dx=&
=\int_{\mathbb{B}^{n}_{\mathbb{C}}}\Phi_{p}(\beta(2n,\alpha)|u|^{p'})dz+\sum^{j_{p}-2}_{j=0}\frac{\beta(2n,\alpha)^{j}}{j!}\int_{\mathbb{B}^{n}_{\mathbb{C}}}
|u|^{jp'}dz\\
\leq&\int_{\mathbb{B}^{n}_{\mathbb{C}}}\Phi_{p}(\beta(2n,\alpha)|u|^{p'})dV+\sum^{j_{p}-2}_{j=0}\frac{\beta(2n,\alpha)^{j}}{j!}\int_{\mathbb{B}^{n}}|u|^{jp'}dz\\
\leq& C+\sum^{j_{p}-2}_{j=0}\frac{\beta(2n,\alpha)^{j}}{j!}\int_{\mathbb{B}^{n}_{\mathbb{C}}}|u|^{jp'}dz\leq C'.
\end{split}
\end{equation*}
To get the last inequation, we use the fact
\[
\left(\int_{\mathbb{B}^{n}_{\mathbb{C}}}|u|^{q}dz\right)^{\frac{p}{q}}\leq \left(\int_{\mathbb{B}^{n}_{\mathbb{C}}}dz\right)^{p-q}
\int_{\mathbb{B}^{n}_{\mathbb{C}}}|u|^{p}dz\lesssim \int_{\mathbb{B}^{n}_{\mathbb{C}}}|u|^{p}dV\lesssim
1,\;\forall1<q\leq p.
\]
The sharpness of the constant $\beta(2n,\alpha)$ can be verified by the process similar to that in the proof
of Theorem \ref{th1.12}.

\section*{ Acknowledgement}
The main results of this paper have been presented in a number of conferences by the authors:  ``The Legacy of Joseph Fourier after 250 years" at the Sanya International Mathematical Forum in December, 2018 and ``Geometric inequalities and applications to geometry and PDEs" in Sanya in January, 2020 and the  AMS special session ``Geometric and functional inequalities and applications to PDEs" in September, 2020.

\end{document}